\pdfoutput=1
 \PassOptionsToPackage{numbers, compress}{natbib}
 \PassOptionsToPackage{noend, vlined, ruled}{algorithm2e}
\def\arXiv{1} 
\documentclass[11pt]{article}

\newcommand{\notarxiv}[1]{foo}
\newcommand{\arxiv}[1]{ba}
\ifdefined \arXiv
\renewcommand{\arxiv}[1]{#1}%
\renewcommand{\notarxiv}[1]{\ignorespaces}%
\else%
\renewcommand{\arxiv}[1]{\ignorespaces}%
\renewcommand{\notarxiv}[1]{#1}%
\fi%

\usepackage{thmtools}
\usepackage{thm-restate}

 \notarxiv{
 	\undef\corollary
	 \undef\definition
	 \undef\assumption

	\setcitestyle{numbers,square,comma} 
	\declaretheorem[name=Lemma,sibling=theorem]{lem}
	\declaretheorem[name=Proposition,sibling=theorem]{prop}
	\declaretheorem[name=Corollary,sibling=theorem]{corollary}
	\declaretheorem[name=Assumption,sibling=theorem]{assumption}
	\declaretheorem[name=Definition,sibling=theorem]{definition}

	\undef\lemma
	\undef\proposition	

}

\arxiv{
	\usepackage{amsthm}
	\usepackage[linesnumbered,ruled,vlined,boxed,algo2e]{algorithm2e}

	\usepackage[dvipsnames]{xcolor}
	\usepackage[top=1in, right=1in, left=1in, bottom=1in]{geometry}
	\usepackage[numbers,square]{natbib}
	\usepackage{url}
	\usepackage[hidelinks]{hyperref}
	\hypersetup{
		colorlinks=true,
		linkcolor=blue!70!black,
		citecolor=blue!70!black,
		urlcolor=blue!70!black}
	
	\theoremstyle{plain}
	
	\newtheorem{theorem}{Theorem}
	\newtheorem{lem}{Lemma}
	
	\newtheorem{assumption}{Assumption}
	\newtheorem{prop}{Proposition}
	\newtheorem{corollary}{Corollary}
	\newtheorem{definition}{Definition}
	
	\theoremstyle{definition}

	\newtheorem*{example*}{Example}

}

\usepackage{enumitem}
\usepackage{amssymb}
\usepackage{amsbsy}
\usepackage{amsmath}
\usepackage{threeparttable}

\usepackage{cleveref}
\Crefname{lem}{Lemma}{Lemmas}
\Crefname{prop}{Proposition}{Propositions}
\Crefname{assumption}{Assumption}{Assumptions}

\usepackage{amsfonts}
\usepackage{latexsym}
\usepackage{graphicx}
\usepackage{color}
\usepackage{xifthen}
\usepackage{xspace}
\usepackage{mathtools}
\usepackage{bbm}
\usepackage{multirow}
\usepackage[normalem]{ulem}
\usepackage{array}
\usepackage[utf8]{inputenc}
\usepackage[T1]{fontenc}
\usepackage{etoolbox}
\newtoggle{restatements}

\newtoggle{heavyplots}

\usepackage{xfrac}
\usepackage{xargs}

\usepackage{titlesec}

\usepackage{setspace}

\usepackage{verbatim}

\usepackage{esvect}

\usepackage{nicefrac}

\usepackage{cases}
\usepackage{empheq}

 \usepackage{booktabs}
 \usepackage{multirow}
 
 \usepackage{caption}

\newcommand{\mc}[1]{\mathcal{#1}}

\newcommand{\wt}[1]{\widetilde{#1}}  %
\newcommand{\what}[1]{\widehat{#1}}  %

\DeclarePairedDelimiter{\abs}{\lvert}{\rvert} %
\DeclarePairedDelimiter{\brk}{[}{]}
\DeclarePairedDelimiter{\crl}{\{}{\}}
\DeclarePairedDelimiter{\prn}{(}{)}

\DeclarePairedDelimiter{\norm}{\|}{\|}
\DeclarePairedDelimiter{\tri}{\langle}{\rangle}

\DeclarePairedDelimiter{\ceil}{\lceil}{\rceil}
\DeclarePairedDelimiter{\floor}{\lfloor}{\rfloor}

\newcommand{\overeq}[1]{\overset{#1}{=}}
\newcommand{\overle}[1]{\overset{#1}{\le}}

\NewDocumentCommand\Ex{s O{} m }{%
	\mathbb{E}%
	\begingroup
	\IfBooleanTF{#1}
	{\ExInn*{#3}}
	{\ExInn[#2]{#3}}%
	\endgroup
}

\DeclarePairedDelimiterX\ExInn[1]{[}{]}{%
	\activatebar
	#1%
}

\RenewDocumentCommand\Pr{sO{}r()}{%
	\mathbb{P}%
	\begingroup
	\IfBooleanTF{#1}
	{\PrInn*{#3}}
	{\PrInn[#2]{#3}}%
	\endgroup
}

\DeclarePairedDelimiterX\PrInn[1](){%
	\activatebar
	#1%
}

\newcommand{\activatebar}{%
	\begingroup\lccode`~=`|
	\lowercase{\endgroup\def~}{\;\delimsize\vert\;}%
	\mathcode`|=\string"8000
}

\newcommand{\R}{\mathbb{R}} %
\newcommand{\N}{\mathbb{N}} %
\newcommand{\E}{\mathbb{E}} %
\renewcommand{\P}{\mathbb{P}}	%
\newcommand{\I}{\mathbb{I}} %

\newcommandx{\CE}[2][1="", 2=""]{\E\left[#1 \;\middle\vert\; #2\right]} %

\usepackage{xcolor}
\usepackage[algo2e]{algorithm2e}

\SetCommentSty{mycommfont}
\SetKwInput{KwInput}{Input} 
\SetKwInput{KwParameter}{Parameters}                %
\SetKwInput{KwOutput}{Output}              %
\SetKwInput{KwReturn}{Return}              %
\SetKwComment{Comment}{$\triangleright$\ }{}
\SetKwInOut{KwParameters}{Parameters}
\SetCommentSty{mycommfont}

\newcommand{\gradest}{\textsc{MorGradEst}}

\newcommand{\varianceReduction}{\textsc{VarianceReduction}}

\newcommand{\linesearch}{\lambda\textsc{-Bisection}}

\makeatletter
\let\oldnl\nl%
\newcommand{\nonl}{\renewcommand{\nl}{\let\nl\oldnl}}%
\makeatother

\DeclareMathOperator*{\argmax}{arg\,max}

\DeclareMathOperator*{\argmin}{arg\,min}

\providecommand{\abs}{\mathop{\rm abs}}

\providecommand{\minimize}{\mathop{\rm minimize}}

\newcommand{\half}{\frac{1}{2}}

\newcommand{\defeq}{\coloneqq}

\newcommand{\grad}{\nabla}

\newcommand{\xset}{\mathcal{X}}

\newcommand{\zset}{\mathcal{Z}}
\newcommand{\uset}{\mathcal{U}}

\newcommand{\eps}{\epsilon}

\newcommand{\Otil}[1]{\widetilde{O}( #1 )}

\newcommand{\inner}[2]{\left<#1,#2\right>}

\newcommand{\ball}{\mathbb{B}}

\newcommand{\indic}[1]{\I{\{#1\}}}

\newcommand{\ones}{\mathbf{1}}

\newcommand{\opt}{_\star}
\newcommand{\xopt}{x\opt}
\newcommand{\Var}{\mathrm{Var}}

\newcommand{\Lsm}[1][\epsilon]{\mc{L}_{\mathrm{smax},{#1}}}
\newcommand{\LsmReg}[1][\epsilon, \lambda]{\mc{L}_{\mathrm{smax},{#1}}}
\newcommand{\Leps}[1][\psi, \epsilon]{\mc{L}_{#1}}

\newcommandx{\upsilonpsi}[1][1= \psi ,usedefault]{\Upsilon_{#1}}

\newcommand{\upsilonepsreg}[1][\eps, \lambda]{\Upsilon_{#1}}
\newcommandx{\upsiloneps}[1][1= {\epsilon} ,usedefault]{\Upsilon_{#1}}
\newcommand{\Lpenalized}[1][\psi]{\mc{L}_{#1}}

\newcommand{\bx}{\bar{x}}  %
\newcommand{\bp}{\bar{p}}

\newcommand{\by}{\bar{y}}
\newcommand{\reps}{r_\epsilon}
\newcommand{\psieps}[1][\epsilon]{\psi_{#1}}

\newcommand{\constraintMul}[1][]{\nu}
\newcommand{\divergenceFunction}[1][]{\psi}

\newcommand{\oracle}[2][\lambda,\delta]{\mathcal{O}_{#1}(#2)}
\newcommand{\oracles}[1][\lambda,\delta]{\mathcal{O}_{#1}}

\newcommand{\xhat}{\hat{x}}
\newcommand{\hx}{\xhat}
\newcommand{\ghat}{\hat{g}}

\newcommand{\gradx}{\ghat^{\mathrm{x}}}
\newcommand{\grady}{\ghat^{\mathrm{y}}}
\newcommand{\zx}{z^{\mathrm{x}}}
\newcommand{\zq}{z^{\mathrm{q}}}
\newcommand{\deltax}{\delta^{\mathrm{x}}}
\newcommand{\deltaq}{\delta^{\mathrm{q}}}
\newcommand{\gx}{\ghat^{\mathrm{x}}}
\newcommand{\gq}{\ghat^{\mathrm{q}}}

\newcommand{\BOO}{BROO\xspace}
\newcommandx{\broocost}[2][1= \delta,2= \lambda, usedefault]{\mathcal{C}_{#2}\prn*{#1}}
\newcommandx{\hpbroocost}[2][1= \delta,2= \lambda, usedefault]{\mathcal{C}^p_{#2}\prn*{#1}}
\newcommandx{\ncost}{\mathcal{C}_F}
\newcommandx{\mlmcConstant}{c_{x',\bx}}

\newcommandx{\jmax}[1][1= \textup{max}, usedefault]{j_{#1}}

\newcommand{\groupObjective}{\mathcal{L}_\textup{g-DRO}}

\newcommand{\ngroups}{M}

\newcommand{\innerLog}{H}

\newcommand{\Tthreshold}[1][\textup{threshold}]{T_{#1}} 
\notarxiv{
\usepackage{times}
}

\usepackage[utf8]{inputenc} %

\title{Distributionally Robust Optimization via Ball Oracle Acceleration}

\notarxiv{
\coltauthor{%
 \Name{Author Name1} \Email{abc@sample.com}\\
 \addr Address 1
 \AND
 \Name{Author Name2} \Email{xyz@sample.com}\\
 \addr Address 2%
}
}

\arxiv{
	\title{Distributionally Robust Optimization via Ball Oracle Acceleration}

	\author{Yair Carmon\\\href{mailto:ycarmon@cs.tau.ac.il}{\texttt{ycarmon@tauex.tau.ac.il}}\and Danielle Hausler \\
	\href{hausler@mail.tau.ac.il}{\texttt{hausler@mail.tau.ac.il}}
    }
	\date{}
}

\begin{document}

\maketitle

\begin{abstract}%
  We develop and analyze algorithms for distributionally robust optimization (DRO) of convex losses. In particular, we consider group-structured and bounded $f$-divergence uncertainty sets. 
  Our approach relies on an accelerated method that queries a ball optimization oracle, i.e., a subroutine that minimizes the objective within a small ball around the query point. 
  Our main contribution is efficient implementations of this oracle for DRO objectives. 
  For DRO with $N$ non-smooth loss functions, the resulting algorithms find an $\epsilon$-accurate solution with  $\widetilde{O}\left(N\epsilon^{-2/3} + \epsilon^{-2}\right)$ first-order oracle queries to individual loss functions. 
  Compared to existing algorithms for this problem, we improve complexity by a factor of up to $\epsilon^{-4/3}$. 
\end{abstract}

\section{Introduction}\label{sec:intro}

The increasing use of machine learning models in high-stakes applications highlights the importance of reliable performance 
across changing domains and populations \cite{buolamwini2018gender,oakden2020hidden,koh2021wilds}.
An emerging body of research addresses this challenge by replacing Empirical Risk Minimization (ERM)
with Distributionally Robust Optimization (DRO) \cite{ben2013robust,shapiro2017distributionally,sagawa2020distributionally, duchi2021learning, lin2022distributionally}, with applications in natural language processing \cite{oren2019distributionally,zhou2021distributionally,koh2021wilds}, reinforcement learning \cite{chow2015risk, urpi2021riskaverse} and algorithmic fairness \cite{hashimoto2018fairness,wang2020robust}.
While ERM minimizes the average training loss, DRO minimizes the worst-case expected loss over all probability distributions in an \emph{uncertainty set} $\uset$, that is, it minimizes
\begin{equation}\label{eq:MainProblem}
	\mc{L}_\textup{DRO}(x) \defeq {\sup_{Q \in \uset} \E_{S \sim Q}[\ell_S(x)]},
\end{equation}
where $\ell_S(x)$ is the loss a model $x \in \xset$ incurs on a sample $S$. 
This work develops new algorithms for DRO, focusing on formulations where $\uset$ contains distributions supported on $N$ training points, where $N$ is potentially large. 
We consider two well-studied DRO variants: (1) Group DRO \cite{wen2014robust,hu2018does,sagawa2020distributionally}, and (2) $f$-divergence DRO \cite{ben2013robust,duchi2021learning}. 
\paragraph{Group DRO}
Machine learning models may rely on spurious correlations (that hold for most training examples but are wrongly linked to the target) 
and therefore suffer high loss on minority groups where these correlations do not hold \cite{hovy2015tagging,hashimoto2018fairness, buolamwini2018gender}. 
To obtain high performance across all groups, Group DRO  minimizes the worst-case loss over groups.
Given a set $\uset = \crl*{w_1, \ldots, w_{\ngroups}}$ of $\ngroups$ distributions over $[N]$,  the Group DRO objective is\footnote{
	Typically, each $w_i$ is uniform over a subset (``group'') of the $N$ training points. However, most approaches (and ours included) extend to the setting of arbitrary $w_i$'s, which was previously considered in~\cite{wen2014robust}.}
\begin{equation}\label{eq:mainProblemGroupDRO}
\groupObjective(x) \defeq  \max_{i \in [\ngroups]} 
\E_{j\sim w_i} \ell_j(x) = 
 \max_{i \in [\ngroups]}  \sum_{j=1}^N w_{ij} \ell_j(x).
\end{equation}
If we define the loss of group $i$ as  $\mc{L}_i(x) \defeq \sum_{j=1}^N w_{ij} \ell_j(x)$ then objective~\eqref{eq:mainProblemGroupDRO} is equivalent to $ \max_{q \in \Delta^\ngroups} \sum_{i \in [\ngroups]}q_i\mc{L}_i(x) $
with $\Delta^\ngroups \defeq \crl{q \in \R^\ngroups_{\ge 0} \mid \vec{1}^Tq = 1}$. Note that, unlike ERM, Group DRO requires additional supervision in the form of subgroup identities encoded by $\{w_i\}$.

\paragraph{DRO with $f$-divergence}
Another approach to DRO, which requires only as much supervision as ERM, takes $\uset$ to be an $f$-divergence ball around the empirical (training)  distribution. For every convex function $f: \R_+ \rightarrow \R\cup \crl*{+\infty} $ such that $f(1) = 0$, the $f$-divergence between distributions $q$ and $p$ over $[N]$ is $ D_f(q,p) \defeq   \sum_{i \in [N]} p_i f\prn*{q_i/p_i}$. The $f$-divergence DRO problem corresponds to the uncertainty set $\uset = \{q\in\Delta^N: D_f(q,\frac{1}{N}\mathbf{1}))\le 1\}$, i.e., 
\begin{equation}\label{eq:mainProblemfDiverDRO}
	\mathcal{L}_{f\text{-div}}(x) \defeq \max_{q\in\Delta^N: \frac{1}{N}\sum_{i\in[N]}f(Nq_i) \le 1}\sum_{i \in [N]}q_i \ell_i(x). 
\end{equation}
This formulation generalizes several well-studied instances of DRO, with the two most notable examples being conditional value at risk (CVaR) and $\chi^2$ uncertainty sets.  CVaR at level $\alpha$ corresponds to $f(x) = \indic{x < \frac{1}{\alpha}}$, and has many applications in finance such as portfolio optimization and credit risk evaluation~\cite{rockafellar2000optimization, krokhmal2002portfolio} as well as in machine learning \cite{oren2019distributionally, levy2020large, curi2020adaptive, zhai2021boosted, chow2015risk, urpi2021riskaverse}. The $\chi^2$ uncertainty set with size $\rho >0$ corresponds to $f(x) \defeq \frac{1}{2\rho}(x-1)^2$ and the resulting DRO problem is closely linked to variance regularization~\cite{duchi2019variance} and has been extensively studied in statistics and machine learning \cite{namkoong2016stochastic,hashimoto2018fairness,duchi2019variance,levy2020large,zhou2021distributionally}.

\paragraph{Complexity notion}
In this paper, we design improved-complexity methods for solving the problems \eqref{eq:mainProblemGroupDRO} and \eqref{eq:mainProblemfDiverDRO} under the assumption that the loss $\ell_i$
is convex and Lipschitz for all $i$. We measure complexity by the (expected) requiered  number of  $\ell_i(x)$ and $\nabla \ell_i(x)$ evaluations to obtain $\eps$-suboptimal solution, i.e., return $x$ such that $\mc{L}_\textup{DRO}(x)-\min_{x_\star \in \xset}\mc{L}_\textup{DRO}(x_\star)\le \eps$ with constant probability. 
Table \ref{table: complexityResults} summarizes our complexity bounds and compares them to prior art.
Throughout the introduction we assume (for simplicity) a unit domain size and that each loss is $1$-Lipschitz.

	\begin{table}[]
	\captionsetup{font=small}
	\centering
		\begin{tabular}{@{}llll@{}}
			\toprule
			Smoothness                                       & 
			Method       &
			Group DRO~\eqref{eq:mainProblemGroupDRO}        &
			$f$-divergence DRO~\eqref{eq:mainProblemfDiverDRO}    \\                                 
			\midrule                                          
			\multirow{2}{*}{None ($L=\infty$)}               & Subgradient method \cite{nesterov2018lectures}                & $N\epsilon^{-2}$                                  & $N\epsilon^{-2}$                        \\
			& Stoch. primal-dual \cite{nemirovski2009robust} $^{*}$        & $M\epsilon^{-2}$                                  & $N\epsilon^{-2}$ \\
			& MLMC stoch. gradient \cite{levy2020large}                     & -                                                 & $\rho \epsilon^{-3}$ or $\alpha^{-1}\epsilon^{-2}$ $^{\dagger}$ \\
			& Ours                                                                     & $N\epsilon^{-2/3} + \epsilon^{-2}$                & $N\epsilon^{-2/3} + \epsilon^{-2}$        \\ \midrule
			\multirow{2}{*}{Weak ($L \approx 1/\epsilon$)}  & AGD on softmax \cite{nesterov2005smooth}                          & $N\epsilon^{-1}$                                  & $N\epsilon^{-1}$                          \\ 
			& Ours                                                                     & $N\epsilon^{-2/3} + N^{3/4}\epsilon^{-1}$         &  $N\epsilon^{-2/3} + \sqrt{N}\epsilon^{-1}$                                                           \\ \midrule
		\end{tabular}
		\caption{\label{table:summary} Number of $\nabla \ell_i$ and $\ell_i$ evaluations to obtain $\E\brk*{\mc{L}_\text{DRO}(x)}-\min_{x_\star \in \xset}\mc{L}_\text{DRO}(x_\star) \le \epsilon$, where $N$ is the number of training points and (for Group DRO) 
			$M$ is the number of groups. The stated rates omit constant and polylogarithmic factors.
			$^{*}$ Requires an additional uniform bound on losses (see~\Cref{ssec:PrimalDualRegretBound}).
			$^{\dagger}$  These rates hold only for specific $f$-divergences: CVaR at level $\alpha$ or $\chi^2$-divergence with size $\rho$, respectively.
		}
	\label{table: complexityResults}
\end{table}

\paragraph{Prior art}
Let us review existsing methods that solve Group DRO and $f$-divergence DRO  problems. 
For a dataset with $N$ training points, the subgradient method \cite{nesterov2018lectures} finds an $\eps$ approximate solution in $\eps^{-2}$ iterations.
Computing a single subgradient costs $N$ functions evaluations (since we need to find the maximizing $q$). 
Therefore, the complexity of this method is $O\prn*{N\eps^{-2}}$. 

DRO can also be viewed as a game between a minimizing $x$-player and a maximizing $q$-player, which makes it amenable to primal-dual methods \cite{nemirovski2009robust, namkoong2016stochastic, sagawa2020distributionally}. 
If we further assume that the losses are bounded then, for $q \in \Delta^m$,  stochastic mirror descent with local norms obtains a regret bound of $O\prn[\big]{\sqrt{m \log (m) /T}}$ (see \Cref{ssec:PrimalDualRegretBound}). 
As a consequence, for Group DRO (where $m=M$) the complexity is $\widetilde{O}\prn*{M\eps^{-2}}$, and for $f$-divergence DRO ($m=N$) the complexity is $\widetilde{O}\prn*{N\eps^{-2}}$.
 
\citet{levy2020large} studied $\chi^2$-divergence and CVaR DRO problems, and proposed using standard gradient methods with a gradient estimator based on multilevel Monte Carlo (MLMC) \cite{blanchet2015unbiased}. 
For $\chi^2$-divergence with ball of size $\rho$ they proved a complexity bound of $\widetilde{O}\prn*{{\rho}{\eps^{-3}}}$, and for CVaR at level $\alpha$ they established complexity  $\widetilde{O}\prn*{{\alpha^{-1}\epsilon^{-2}}}$.
However, for large uncertainty sets (when $\rho$ or $\alpha^{-1}$ approach ${N}$) their method does not improve over the subgradient method.

Stronger complexity bounds are available under the weak smoothness assumption that each $\ell_i$ has $O\prn*{\eps^{-1}}$-Lipschitz gradient. In particular, 
We can apply Nesterov's accelerated gradient descent method \cite{nesterov2005smooth} on an entropy-regularized version of our objective  to solve the problem with complexity $\widetilde{O}\prn*{N\eps^{-1}}$; see \Cref{ssec:AGDonSoftmax} for more details. 

\paragraph{Our contribution}
We propose algorithms that solve the problems \eqref{eq:mainProblemGroupDRO} and~\eqref{eq:mainProblemfDiverDRO} with complexity $\widetilde{O}\prn*{N\eps^{-2/3}+\eps^{-2}}$. 
Compared to previous works, we obtain better convergence rates for DRO with general  $f$-divergence when $N \gg 1$ and for Group DRO when $\ngroups \gg N \eps^{4/3}$.
When the losses have $O\prn{\epsilon^{-1}}$-Lipschitz gradient, we solve $f$-divergence DRO  with complexity $\widetilde{O}\prn[\big]{N\eps^{-2/3}+\sqrt{N}\eps^{-1}}$, 
and, under an even weaker mean-square smoothness assumption ($\E_{j \sim w_i}\norm{\nabla \ell_j(x) - \nabla \ell_j(y)}^2 \le O(\epsilon^{-2})\norm{x-y}^2$ for all $x,y$ and $i$), we solve Group DRO with complexity $\widetilde{O}\prn*{N\eps^{-2/3}+N^{3/4}\eps^{-1}}$. 

Our complexity bounds are independent of the structure of $f$ and $\{w_i\}$, allowing us to consider arbitrarily  $f$-divergence balls and support a large number of (potentially overlapping) groups.
Our rates are optimal up to logarithmic factors for the special case of minimizing $\max_{i\in [N]}\ell_i(x)$, which corresponds to Group DRO with $N$ distinct groups and $f$-divergence DRO with $f=0$ \cite{woodworth2016tight,zhou2019lower,carmon2021thinking}. 

\paragraph{Our approach}
Our algorithms are based on a technique for accelerating optimization with a \emph{ball optimization oracle}, introduced by  \citet{carmon2020acceleration} and further developed in~\cite{carmon2021thinking,asi2021stochastic}.
Given a function $F$ and a query point $x$, the ball optimization oracle returns an approximate minimizer of $F$ inside a ball around $x$ with radius $r$; the works~\cite{carmon2020acceleration,carmon2021thinking,asi2021stochastic} show how to minimize $F$ using $\widetilde{O}(r^{-2/3})$ oracle calls. Our development consists of efficiently implementing ball oracles with radius $r=\widetilde{O}(\epsilon)$ for the DRO problems~\eqref{eq:mainProblemGroupDRO} and~\eqref{eq:mainProblemfDiverDRO}, leveraging the small ball constraint to apply stochastic gradient estimators that would have exponential variance and/or cost without it.

\citet{carmon2021thinking} previously executed this strategy for  minimizing the maximum loss, i.e., $\max_{q\in\Delta^N} \sum_i q_i \ell_i(x)$, which is a special case of both Group DRO and $f$-divergence DRO. However, the ball-oracle implementations of~\cite{carmon2021thinking} do not directly apply to the DRO problems that we consider; our oracle implementations differ significantly and intimately rely on the Group DRO and $f$-divergence problem structure. We now briefly review the main differences between our approach and~\cite{carmon2021thinking}, highlighting our key technical innovations along the way.

Since the Group DRO objective is $\max_{q\in\Delta^M} \sum_i q_i  \mc{L}_i(x)$ for $\mc{L}_i(x) = \sum_{j\in[N]}w_{ij}\ell_j(x)$, one may naively apply the technique of~\cite{carmon2021thinking} with $\mc{L}_i$ replacing $\ell_i$. However, every step of such a method would involve computing quantities of the form $e^{\mc{L}_i(x)/\epsilon'}$ (for some $\epsilon' = \widetilde{\Theta}(\epsilon)$), which can be up to $N$ times more expensive than computing $e^{\ell_{j}(x)/\epsilon'}$ for a single $j$. To avoid such expensive computation we use MLMC~\cite{blanchet2015unbiased} to obtain an unbiased estimate of $e^{\mc{L}_i(x)/\epsilon}$ with complexity $O(1)$ and appropriately bounded variance. 
In the weakly-smooth case we also adapt our estimator to facilitate variance reduction~\cite{johnson2013accelerating,allen2018katyusha}

For $f$-divergence, we mainly consider the Lagrangian form $\max_{q\in\Delta^N}\sum_{i\in[N]} \crl*{q_i \ell_i(x) - \psi(q_i)}$ for some convex $\psi$. 
\citet{carmon2021thinking} solved such a problem with $\psi(t)= \epsilon' t\log t$
in order to stabilize $q^\star(x) = \argmax_{q\in\Delta^N}\sum_{i\in[N]}\crl*{ q_i \ell_i(x) - \psi(q_i)}$. 
Their gradient estimation approach requires computing importance weighting correction terms of the form $q^\star_i(x)/q^\star_i(\bx)$ for reference point $\bx$ and a nearby point $x$.
For the special case of $\psi(t)= \epsilon' t\log t$, such correction terms can be computed with complexity $O(1)$. 
However, this is no longer true for general $\psi$: naively computing the importance weight requires computing $q^\star(x)$ and hence has complexity $N$. 

To address this challenge, we consider the well-known dual form \cite{ben2013robust,shapiro2017distributionally}
\begin{equation*}
	\max_{q\in\Delta^N}\sum_{i\in[N]} \crl*{q_i \ell_i(x) - \psi(q_i)}
	=
	\min_{y\in\R} \crl[\Bigg]{ \Upsilon(x,y)\defeq\sum_{i\in[N]} {\psi^*}(\ell_i(x)-y) + y}
\end{equation*}
where $\psi^*(v) = \max_{t\ge 0}\crl*{v t-\psi(t)}$ is the Fenchel dual of $\psi$. Several prior works consider applying stochastic gradient methods to minimize $\Upsilon$ over $x$ and $y$~\cite{namkoong2016stochastic,levy2020large,jin2021non}, but for general $\psi$ these techniques fail due to ${\psi^*}'$ being large and unstable~\cite{namkoong2016stochastic}. We tame this instability by sampling indices from $q^\star(\bar{x})$ and applying importance weighting corrections of the form ${\psi^*}'(\ell_i(x)-y)/q^\star(\bar{x})$. To show that these corrections are appropriately bounded we exploit the small domain size and also entropy-regularize $\psi$, i.e., replace it with $\psi_\epsilon(t) = \psi(t) + \epsilon' t \log t$. Underlying our result are two technical observations: (i) $\log ({\psi_\epsilon^*}'(\cdot))$ is $1/\epsilon'$-Lipschitz for all convex $\psi$, and (ii) for 1-Lipschitz losses, $y^\star(x) = \argmin_{y\in\R} \Upsilon(x,y)$ satisfies $|y^\star(x)-y^\star(x')| \le \norm{x-x'}$ for all $x,x'$. To the best of our knowledge, these observations are new and potentially of independent interest.

\subsection{Related work}
\paragraph{MLMC estimators}
The multilevel Monte Carlo (MLMC) technique was introduced by \citet{giles2008multilevel} and \citet{heinrich2001multilevel} 
in order to reduce the computational cost of Monte Carlo estimation of integrals. 
\citet{blanchet2015unbiased} extended this techniques to estimating functions of expectation and proposed several applications, including stochastic optimization~\cite{blanchet2019unbiased}.
In this work we use their estimator for two distinct purposes:  
(1) obtaining unbiased Moreau envelope gradient estimates for ball oracle acceleration as proposed by~\citet{asi2021stochastic}, and (2) estimating the 
the exponential of an expectation for Group DRO. \citet{levy2020large} also rely on MLMC for DRO, but quite differently than we do: they directly estimate the DRO objective gradient via MLMC, while we estimate different quantities.

\paragraph{Other DRO methods}
Several additional works proposed algorithms with theoretical guarantees for $f$-divergence DRO. 
\citet{jin2021non} considered non-convex and smooth losses. \citet{song2021coordinate} proposed an algorithm for linear models with complexity comparable to the ``AGD on softmax'' approach (\Cref{ssec:AGDonSoftmax}).  
 \citet{namkoong2016stochastic} proposed a primal-dual algorithm that is suitable for small uncertainty $\chi^2$ sets (with size $\rho \ll \frac{1}{N}$)
and \citet{curi2020adaptive} proposed a primal-dual algorithm specialized for CVaR.
Other works consider DRO with uncertainty sets defined using the Wasserstein distance \cite{gao2016distributionally,esfahani2018data,sinha2018certifiable, kent2021frank}. 
Another relevant line of works proposes refinements for DRO that address some of the challenges in applying it to learning problems \cite{zhai2021boosted, zhai2021doro, wang2022is}. 
\section{Preliminaries}\label{sec:prelim}
\newcommand{\kmax}{K_{\max}}

\paragraph{Notation}
We write $ \norm{\cdot}$ for the Euclidean norm. 
We denote by $ \ball_r(x_0)$ the Euclidean ball of radius $r$ around $x_0$.
We let $\Delta^n \defeq \crl{q \in \R^n_{\ge 0} \mid \ones^Tq = 1}$
denote the probability simplex in $\R^n$.
For the sequence $z_m,\ldots,z_n$ we use the shorthand $z_m^n$. Using $F$ as a generic placeholder (typically for a loss function $\ell_i$), we make frequent use of the following assumption.

\begin{assumption}\label{assumption:GlobalAssumption}
    The function $F: \xset \rightarrow \R$ is convex and and G-Lipschitz, i.e., for all $x,y\in\xset$ we have $\abs*{F(x) - F(y)} \le G\norm*{x-y}$. 
    In addition, the set $\xset$ has Euclidean diameter at most $R$.
\end{assumption}
\noindent
Throughout, $N$ denotes the number of losses and, in \Cref{sec:groupDRO}, $\ngroups$ denotes the number of groups. We use $\epsilon$ for our target accuracy and $\reps\defeq \epsilon' / G$ for the ball radius, where $\epsilon'=\epsilon / (2\log \ngroups)$ for Group-DRO (\Cref{sec:groupDRO}) and $\epsilon'=\epsilon / (2\log N)$ for $f$-divergence DRO (\Cref{sec:psiDivergence}).

\paragraph{Complexity model}
We measure an algorithm's complexity by its \emph{expected} number of $\ell_i$ and $\nabla \ell_i$ evaluations; bounds on expected evaluation number can be readily converted to more standard probability 1 bounds \cite[see][Appendix A.3]{asi2021stochastic}. Moreover if $\xset\subset \R^d$, $d=\Omega(\log N)$,\footnote{
	The assumption $d=\Omega(\log N)$ is only necessary for our results on $f$-divergence DRO (\Cref{sec:psiDivergence}), where the runtime of computing $\argmin_{y\in \R} \Upsilon(x,y)$ is $O(Nd + N\log N)$ due to the need to sort the losses.
} and the time to evaluate $\ell_i$ and $\grad \ell_i$ is $O(d)$,
the expected runtime of all the algorithms we consider is at most $d$ times the evaluation complexity.

\subsection{Ball oracle acceleration}
We now briefly summarize the complexity bounds given by the framework of~\cite{carmon2020acceleration,carmon2021thinking,asi2021stochastic} for accelerated minimization using queries to (inexact) ball optimization oracles, defined as follows.

\begin{definition}\label{def:BROO}
	An algorithm is a Ball Regularized Optimization Oracle of radius 
	$r$ ($r$-BROO) for function $F: \xset \to \R$ if for query point $\bx\in\xset$, regularization parameter $\lambda >0$ and desired accuracy $\delta>0$ it returns $\oracles(\bx)\in\xset$ satisfying
    \begin{equation}\label{eq:BROOapprox}
        \E\brk*{F(\oracles(\bx)) + \frac{\lambda}{2}\norm*{\oracles(\bx)- \bx}^2} \le \min_{x \in \ball_r(\bx)\cap \xset}\crl*{F(x)+\frac{\lambda}{2}\norm*{x-\bx}^2}+\frac{\lambda}{2}\delta^2.
    \end{equation}
\end{definition}

\newcommand{\meps}{m_\epsilon}
\newcommand{\lmin}{\lambda_{\mathrm{m}}}

\begin{restatable}{prop}{MainProp}\label{prop:MainProp}
    Let $F$ satisfy \Cref{assumption:GlobalAssumption}, let $\ncost$ be the complexity of evaluating $F$ exactly, and let $\broocost$ bound the complexity of an $r$-BROO query with $\delta, \lambda$. Assume that $\broocost$ is non-increasing in $\lambda$ and at most polynomial in $1/\delta$. For any $\epsilon>0$, \Cref{alg:acceleratedProxPoint} returns $x$ such that $F(x)-\min_{\xopt \in \xset}F(\xopt)\le \epsilon$ with probability at least $\half$. For $\meps = O\prn[\big]{\log \frac{GR^2}{\epsilon r}}$ and $\lmin = O\prn[\big]{\frac{\meps^2 \epsilon}{r^{4/3}R^{2/3}}}$, the complexity of the algorithm is
    \begin{equation}\label{eq:MLMCBROOguarantee}
    	O\prn*{
    		\prn*{\frac{R}{r}}^{2/3} \brk*{
    			 \prn*{\sum_{j=0}^{\meps}\frac{1}{2^j}\broocost[\frac{r}{ 2^{j/2}\meps^2}][\lmin]}\meps 
    			+
    			\prn*{ \broocost[r][\lmin] + \ncost }\meps^3, 
    		}
    	}.
    \end{equation}   
\end{restatable}
Informally, the proposition shows that $\widetilde{O}((R/r)^{2/3})$ BROO calls with $\lambda = \widetilde{\Omega}(\epsilon / (r^{3/4} R^{2/3}))$ and accuracy $\delta=\widetilde{O}(r)$ suffice to find an $\epsilon$-accurate solution. As we show in the sequel, for $\broocost = \widetilde{O}\prn[\big]{N+(\frac{G}{\lambda\delta})^2}$ the resulting complexity bound is $\widetilde{O}\prn[\big]{N (\frac{GR}{\epsilon})^{2/3} + (\frac{GR}{\epsilon})^2}$. 
The summation over $j$ in bound~\eqref{eq:MLMCBROOguarantee} stems from the use of MLMC to de-bias the BROO output (i.e., make it exact in expectation): compared to the original proposal of~\citet{asi2021stochastic}, our version of the procedure in \Cref{sec:Preliminariesproofs} slightly alters this MLMC scheme by de-biasing one accurate BROO call instead of averaging many inaccurate de-biased calls, improving our bounds by logarithmic factors.

\newcommand{\gammahat}{\what{\gamma}}
\newcommand{\dhat}{\what{\mathcal{D}}}
\newcommandx{\gammamlmc}[1][1=x, usedefault]{\what{\mathcal{M}}\brk*{\gamma_i(#1)}}

\section{Group DRO}\label{sec:groupDRO}
In this section we develop our BROO implementations for the Group DRO objective \eqref{eq:mainProblemGroupDRO}. 
In \Cref{ssec:Exponentiated-softmax} we describe an ``exponentiated group-softmax'' function that approximates  $\groupObjective$ with additive error at most $\eps/2$. 
We then apply stochastic gradient methods on this function to obtain \BOO  implementations that yield improved rates for Group DRO via~\Cref{prop:MainProp}: we first consider the non-smooth case in \Cref{ssec:Gradient-estimator} and then 
 the weakly-smooth case in \Cref{ssec:Implementation of accelerated variance reduction}. 
\subsection{Exponentiated group-softmax}\label{ssec:Exponentiated-softmax}
Given a cheap and unbiased stochastic gradient estimator of $\nabla \groupObjective$, we could use a variant of SGD and minimize $\groupObjective$ to $\eps$-suboptimal solution using $O(\eps^{-2})$ steps. However, obtaining an unbiased estimator is challenging due to the maximum operator in $\groupObjective$. 
We first describe a straightforward extension of the exponentiated softmax of \citet{carmon2021thinking} that approximates $\groupObjective$ and has the form of a weighted finite sum which is more amenable for stochastic optimization.
For target accuracy $\epsilon$, regularization parameter $\lambda\ge 0$, center point $\bx\in\xset$ and  $\epsilon' = \epsilon / (2 \log M) >0$,
the (regularized) group-softmax function is
\begin{align}\label{eq:SoftMax}
        \LsmReg(x) \defeq \epsilon' \log \prn*{\sum_{i \in [\ngroups]}e^{\frac{\mc{L}_i(x)}{ \eps'} }}+\frac{\lambda}{2}\norm{x-\bx}^2 \text{~~where~~} \mc{L}_i(x) = \sum_{j \in [N]}w_{ij}\ell_{j}(x).
\end{align}
We will implement a \BOO for $\Lsm \defeq \LsmReg[\epsilon,0]$, which is a uniform approximation of $\groupObjective$: $ \abs*{ \groupObjective(x)  - \Lsm(x)} \le \epsilon/2$ for all $x\in\xset$; see \Cref{ssec:ExpSoftMax-proof} for details. 

The (regularized) exponentiated group-softmax is
\begin{align}\label{eq:ExpSoftMax}
        \Gamma_{\epsilon, \lambda}(x) \defeq 
         \sum_{i \in [\ngroups]}\bp_i \gamma_i(x)
         \text{~~where~~} \gamma_i(x) = \epsilon' e^{\frac{\mc{L}_i(x) - \mc{L}_i(\bx)+\frac{\lambda}{2}\norm{x-\bx}^2}{\epsilon'}}
         \text{~~and~~}
         \bp_i = \frac{e^{\frac{\mc{L}_i(\bx)}{\epsilon'}}}{\sum_{i \in [\ngroups]}e^{\frac{\mc{L}_i(\bx)}{\epsilon'} }}.
\end{align}
In the following lemma we (easily) extend \citet[][Lemma 1]{carmon2021thinking} to exponentiated group-softmax, showing that $\Gamma_{\epsilon, \lambda}$ is well-behaved inside a ball of (appropriately small) radius $r$ around $\bx$ and  facilitates minimizing $\LsmReg$ in that ball; see \Cref{ssec:ExpSoftMax-proof} for the proof. 
\begin{restatable}{lem}{SmGammaProperties}%
	\label{lem:SmGammaProperties}
	Let each $\ell_i$ satisfy \Cref{assumption:GlobalAssumption}, and consider the restriction of 
     $\LsmReg$~\eqref{eq:SoftMax} and  $\Gamma_{\eps, \lambda}$~\eqref{eq:ExpSoftMax} to $\ball_r(\bx)$. Then the functions have the same minimizer $x_\star \in \ball_r(\bx)$ and, 
    if $\lambda \le O\prn*{G/r}$ and $r \le O\prn*{\eps' / G}$, then (a)
    $\Gamma_{\epsilon,\lambda}$ is $\Omega(\lambda)$-strongly convex, (b) each $\gamma_i$ is
     $O(G)$-Lipschitz and (c) 
    for every $x \in \ball_r(\bx)$ we have
    $
\LsmReg(x) -\LsmReg(x_\star) \le O\prn*{\Gamma_{\epsilon,\lambda}(x) - \Gamma_{\epsilon,\lambda}(x_\star) }. 
$ 
\end{restatable} 

\subsection{\BOO implementation for Group DRO  non-smooth losses}\label{ssec:Gradient-estimator}
To motivate our \BOO implementation, let us review how~\cite{carmon2021thinking} use the exponentiated softmax in the special case of size-1 groups, i.e., $\mc{L}_i = \ell_i$, and explain the difficulty that their approach faces when the group structure is introduced. 
The \BOO implementation in~\cite{carmon2021thinking} is based on SGD variant with the stochastic gradient estimator  $\ghat(x) = e^{\prn{\mc{L}_i(x)-\mc{L}_i(\bx)}/{\eps'}}\nabla \mc{L}_i(x)$ where $i \sim \bp_i$. 
However, for Group DRO where $\mc{L}_i = \sum_{j\in[N]} w_{ij} \ell_j$, the estimator $\ghat(x)$ can be up to $N$ times more expensive to compute.
Approximating $\ghat(x)$ by drawing $j,j' \sim w_i$  and taking  $e^{\prn{\ell_j(x)-\ell_j(\bx)}/{\eps'}}\nabla \ell_{j'}(x)$ will result in a biased estimator since 
$\E_{j \sim w_i} e^{\prn{\ell_j(x)-\ell_j(\bx)}/{\eps'}} \ne e^{\prn{\mc{L}_i(x)-\mc{L}_i(\bx)}/{\eps'}}$. 
To address this challenge we propose a new gradient estimator based on the multilevel Monte Carlo (MLMC) method \cite{blanchet2015unbiased}.

The MLMC unbiased estimator for  $\gamma_i(x) = \epsilon' e^{\prn{\mc{L}_i(x)-\mc{L}_i(\bx)}/{\eps'}}$, which we denote by $\gammamlmc$,
is defined as follows:
\[ 
\text{Draw~} J\sim \mathrm{Geom}\prn*{1-\tfrac{1}{\sqrt{8}}} \text{~,~} S_1, \ldots, S_n \overset{\textup{iid}}{\sim} w_i  \text{~and let~}  \gammamlmc \defeq   \gammahat(x; S_1) + \frac{\dhat_{2^J}}{p_J},
 \] 
where $p_j \defeq \P(J=j) = \prn*{1/\sqrt{8}}^j \prn*{1-\frac{1}{\sqrt{8}}}$ and, for $n \in 2\mathbb{N}$, we define 
\[ \dhat_n \defeq \gammahat(x; S_1^n) -  \frac{ \gammahat\prn*{x; S_1^{n/2}} + \gammahat\prn*{x; S_{n/2 + 1}^n}}{2}
\text{~and~}  \gammahat(x; S_1^n) \defeq  \eps' e^{\frac{1}{n}\sum_{j=1}^n \frac{\ell_{S_j}(x)- \ell_{S_j}(\bx)+\frac{\lambda}{2}\norm{x-\bx}^2}{\epsilon'}}. \]

With the MLMC estimator for $\gamma_i$ in hand, we estimate the gradient of $\Gamma_{\epsilon,\lambda}$ as follows:
\begin{flalign}\label{eq:mlmcGradEst}
    \text{Draw~} i \sim p(\bx) \text{~,~} j \sim w_i \text{~and set~}
      \hat{g}(x) = \frac{1}{\epsilon'}\gammamlmc \prn*{\nabla \ell_{j}(x)+\lambda\prn*{x-\bx} }.
\end{flalign}
In the following lemma we summarize the important properties of the MLMC and gradient estimators; 
see \Cref{ssec:MLMCbounds-proof} for the proof. 
\begin{restatable}{lem}{MLMCbounds}\label{lem:MLMCbounds}
	Let each $\ell_i$ satisfy \Cref{assumption:GlobalAssumption}, and let $r \le \frac{\eps'}{G}$ , $\lambda \le \frac{G}{r}$ and $x \in \ball_r(\bx)$. 
    Then $\gammamlmc$ and $\ghat(x)$ are unbiased for $\gamma_i(x)$ and $\grad \Gamma_{\epsilon, \lambda}(x)$, respectively, and have bounded second moments: $ \E \brk[\big]{\gammamlmc} ^ 2  \le  O\prn*{\frac{G^4\norm*{x-\bx}^4}{\epsilon'^2} + \eps'^2 }$
    and $\E\norm*{\ghat(x)}^2 \le O\prn*{G^2}$. In addition, the complexity of computing $\gammamlmc$ and $\ghat(x)$ is $O(1)$.
\end{restatable}
Due to \Cref{lem:MLMCbounds} and since 
$\Gamma_{\eps, \lambda}$ is $\Omega(\lambda)$-strongly convex, we can use the Epoch-SGD algorithm of \citet{hazan2014beyond} with our gradient estimator \eqref{eq:mlmcGradEst}. 
This algorithm has rate of convergence $O\prn*{G^2 / (\lambda T)}$ and our gradient estimator requires additional $N$ function evaluations for precomputing the sampling probabilities $\crl*{\bp_i}$. 
We thus arrive at the following complexity bound. 
\begin{restatable}{theorem}{BrooComplexGoupDRO}\label{thm:BrooComplexGoupDRO}
    Let each $\ell_j$ satisfy \Cref{assumption:GlobalAssumption}, let $\epsilon, \delta , \lambda >0$
    and let $\reps = \epsilon /(2G\log M)$. For any query point $\bx\in \R^d$, regularization strength $\lambda \le O(G / \reps)$ and accuracy $\delta$, 
    EpochSGD \cite[Algorithm 1]{hazan2014beyond}) with the gradient estimator \eqref{eq:mlmcGradEst} outputs a valid $\reps$-BROO
    response and has complexity $   \broocost = O\prn[\big]{N+ \frac{G^2}{\lambda^2 \delta^2}}$.
    Consequently, the complexity of finding an $\epsilon$-suboptimal minimizer of $\groupObjective$ \eqref{eq:mainProblemGroupDRO} with probability at least $\half$ is 
	\begin{equation*}
		O\prn*{N \prn*{\frac{GR}{\epsilon} }^{2/3} \log^{11/3} \innerLog+  \prn*{\frac{GR}{\epsilon}}^2\log^{2} \innerLog } \text{~~where~~} \innerLog \defeq M\frac{GR}{\eps}.
	\end{equation*}
\end{restatable}
We provide the proof for \Cref{thm:BrooComplexGoupDRO} in \Cref{ssec:GroupDROEpochSGD-proof}; the final complexity bound follows from straightforward calculations which we now briefly outline.
According to \Cref{prop:MainProp}, finding an $\frac{\eps}{2}$-suboptimal solution for $\Lsm$ (and consequently an $\epsilon$-suboptimal solution for $\groupObjective$)
involves $\widetilde{O}\prn[\big]{\prn*{R/\reps}^{2/3}}$ BROO 
calls with accuracy $\delta = \widetilde{\Omega}\prn*{\reps 2^{-J/2}}$ and regularization strength $\lambda \ge \lmin$, where $J=\min\{\textup{Geom}(\half), m\}$. We may therefore bound the complexity of each such call by 
\begin{equation*}
	\sum_{j=0}^m 2^{-j} \broocost[\reps 2^{-j/2}][\lmin] = 
	\sum_{j=0}^m 2^{-j} \widetilde{O} \prn*{ N + \frac{2^j G^2}{\lmin^2 \reps^2}  } \overeq{(\star)} \widetilde{O} \prn*{ N + \prn*{\frac{GR}{\epsilon}}^{2} \prn*{\frac{\reps}{R}}^{2/3}},
\end{equation*}
where $(\star)$ follows from substituting $\lmin = \widetilde{\Omega}\prn[\big]{\eps \reps^{-4/3}R^{-2/3}}$ and $m=\Otil{1}$. Multiplying this bound by $\widetilde{O}\prn[\big]{\prn*{R/\reps}^{2/3}}$ yields (up to polylogarithmic factors) the conclusion of \Cref{thm:BrooComplexGoupDRO}.

\subsection{Accelerated variance reduction for mean-square smooth  losses}\label{ssec:Implementation of accelerated variance reduction}
In this section we provide an algorithm with an improved rate of convergence under the following mean-square smoothness assumption.
\begin{assumption}\label{assumption:smooth}
    For all $x, x'  \in \ball_r(\bx)$ and $i\in [\ngroups]$, $\E_{j \sim w_i} \norm*{\nabla \ell_{j}(x) -\nabla \ell_{j}(x')}^2 \le L^2 \norm*{x-x'}^2$. 
\end{assumption}
Note that assuming $L$-Lipschitz gradient for each $\ell_i$ implies \Cref{assumption:smooth}, but not the other way around.
To take advantage of \Cref{assumption:smooth},
we first rewrite the function $\Gamma_{\eps, \lambda}(x)$ in a way that is more amenable to variance reduction:
\begin{align*}
   & \Gamma_{\eps, \lambda}(x) \defeq \sum_{i \in [\ngroups]} \mlmcConstant p_i(x')\gamma_i(x,x'),
\end{align*}
where  $\gamma_i(x,x') \defeq \eps'  e^{\frac{\mc{L}_i(x)-\mc{L}_i(x')+\frac{\lambda}{2}\norm*{x-\bx}^2}{\eps'}}$, $ \mlmcConstant=\prn[\Bigg]{\frac{\sum_{j \in [\ngroups]}e^{\frac{\mc{L}_j(x')}{\eps'}}}{\sum_{j \in [\ngroups]}e^{\frac{\mc{L}_j(\bx)}{\eps'}}}}$
and $p_i(x') \defeq \frac{e^{\frac{\mc{L}_i(x')}{\eps'}}}{\sum_{j \in [\ngroups]}e^{\frac{\mc{L}_j(x')}{\eps'}}}.$
(Note that $ \gamma_i(x,\bx) = \gamma_i(x)$).

Given a reference point $x'$, to compute a reduced-variance estimator of $\grad \Gamma_{\eps, \lambda}(x)$, 
we draw $ i \sim p_i(x') $ and $j \sim w_i$, and set:
\begin{flalign}\label{eq:SVRG}
    \hat{g}_{x'}(x) \defeq \nabla \Gamma_{\eps, \lambda}(x') 
    + \frac{\mlmcConstant}{\epsilon'}\brk*{\gammamlmc[x,x'] \nabla \ell^\lambda_{j}(x) - \gamma_i(x',x') \nabla \ell^\lambda_{j}(x')}
\end{flalign}
where $\nabla \ell^\lambda_{j}(x) \defeq \nabla \ell_{j}(x) + \lambda\prn*{x-\bx}$ and $\gammamlmc[x,x']$ is an MLMC estimator for $\gamma_i(x,x')$ defined analogously to $\gammamlmc$ (see details in \Cref{ssec:GroupDROSVRGproperties-proof}). 
The estimator \eqref{eq:SVRG} is not precisely standard SVRG \cite{johnson2013accelerating}
since we use $\gammamlmc[x,x']$ as an estimator for $\gamma_i(x,x')$. 
Simple calculations show that $ \E \hat{g}_{x'}(x) = \nabla  \Gamma_{\eps, \lambda}(x) $ and the following lemma shows that $ \hat{g}$ satisfies a type of variance bound conducive to variance-reduction schemes; see \Cref{ssec:GroupDROSVRGproperties-proof} for the proof. 
\begin{restatable}{lem}{SVRGvarianceBound}\label{lem:SVRGvarianceBound}
    Let each $\ell_{j}$ satisfy \Cref{assumption:smooth,assumption:GlobalAssumption}. For any $\lambda \le \frac{G}{r}$, $r = \frac{\epsilon'}{G}$ and $x,x' \in \ball_{r}(\bx)$, 
    the variance of $\hat{g}_{x'}(x)$ is bounded by $\Var\prn*{\ghat_{x'}(x)} \le O\prn*{\prn[\big]{L+\lambda + \frac{G^2}{\epsilon'}}^2\norm*{x-x'}^2}$.
\end{restatable}
Accelerated variance reduction methods for convex functions typically require a stronger variance bound of the form 
$\Var\prn*{\ghat_{x'}(x)} \le 2L\prn*{F(x')-  F(x)-\tri*{\nabla F(x), x'-x}}$ for every $x$~\cite[cf.][Lemma 2.4]{allen2017katyusha}. The guarantee of~\Cref{lem:SVRGvarianceBound} is weaker, but still allows for certain accelerated rates via, e.g., the Katyusha X algorithm \cite{allen2018katyusha}. With it, we obtain the following guarantee.
\begin{restatable}{theorem}{svrgBrooComplex}\label{thm:svrgBrooComplex}
    Let  each $\ell_{j}$ satisfy \Cref{assumption:smooth,assumption:GlobalAssumption}. Let $\epsilon> 0$, $\epsilon' = \epsilon / \prn*{2 \log M}$ and $\reps = \epsilon' / G$. For any query point $\bx\in \R^d$, regularization strength $\lambda \le O(G/\reps)$ and accuracy $\delta$, KatyushaX$^s$ \cite[Algorithm 2]{allen2018katyusha} with the gradient estimator  \eqref{eq:SVRG} outputs a valid $\reps$-BROO response 
    and has complexity $\broocost[\delta]= O\prn[\Big]{\prn[\Big]{N+\frac{N^{3/4}\prn*{G+\sqrt{\epsilon'L}}}{\sqrt{\lambda\epsilon'}}}\log\prn*{\frac{G r_{\eps}}{\lambda \delta^2}}}$.
    Consequently, the complexity of finding an $\epsilon$-suboptimal minimizer of $\groupObjective$~\eqref{eq:mainProblemGroupDRO} with probability at least $\frac{1}{2}$ is 
	\begin{align*}
		O \prn*{N \prn*{\frac{GR}{\eps} }^{2/3} \log^{14/3} \innerLog  + N^{3/4}\prn*{\frac{GR}{\eps}+ \sqrt{\frac{L R^2}{\eps}}}\log^{7/2} \innerLog    }
		\text{~~where~~} \innerLog \defeq M \frac{GR}{\eps}.
	\end{align*}
\end{restatable}
We provide the proof of \Cref{thm:svrgBrooComplex} in \Cref{ssec:GroupDROSVRG-proof}. For the special case of Group DRO with a single group satisfying \Cref{assumption:smooth} with $L=\Theta\prn*{G^2/\eps}$, i.e.\ minimizing the average loss, we have the lower bound $\widetilde{\Omega}\prn[\big]{N+N^{3/4}\frac{GR}{\eps}}$ \cite{zhou2019lower} 
and for the case of $N$ distinct groups, i.e.\ minimizing the maximal loss, we have the lower bound $\widetilde{\Omega}\prn*{N\eps^{-2/3}}$ \cite{carmon2021thinking}. This implies that, in the weakly mean-square smooth setting, both the term scaling as $N^{3/4}\eps^{-1}$ and the term scaling as $N\eps^{-2/3}$ are unimprovable.

\newcommandx{\probratio}[1][1=x, usedefault]{\frac{p_i(#1)}{\bp_i} }
\section{DRO with $f$-divergence}\label{sec:psiDivergence}
In this section we develop our BROO implementation for the $f$-divergence objective \eqref{eq:mainProblemfDiverDRO}. 
In \Cref{ssec:The dual problem} we 
reduce the original DRO problem to a regularized form using Lagrange multipliers.
Next, in \Cref{ssec:Stability of the likelihood ratio} we show that adding  negative entropy regularization to the objective produces the stability properties necessary for efficient ball optimization.  
In \Cref{ssec:The gradient estimator} we describe a BROO implementation for the non-smooth case using a variant of Epoch-SGD \cite{hazan2014beyond}, 
and in \Cref{ssec:Accelerated variance reduction} we implement the BROO under a weak-smoothness assumption by carefully restarting an accelerated variance reduction method \cite{allen2017katyusha}. 

\subsection{The dual problem}\label{ssec:The dual problem}
We first note that, by Lagrange duality, the objective \eqref{eq:mainProblemfDiverDRO}  is equivalent to
\begin{equation*}
	\mathcal{L}_{f\text{-div}}(x) \defeq \max_{q\in\Delta^N: \sum_{i\in[N]}\frac{f(Nq_i)}{N} \le 1}\sum_{i \in [N]}q_i \ell_i(x)
	=
	\min_{\nu \ge 0}\crl[\Bigg]{ \nu +\max_{q\in\Delta^N}
		\sum_{i\in[N]} \prn*{q_i \ell_i(x) - \frac{\nu}{N}f(Nq_i)}
		}.
\end{equation*}
Writing $\psi(x) \defeq \frac{\constraintMul}{N}f(Nx)$, we therefore consider objectives of the form
\begin{equation}\label{eq:penalizedFdiver}
    \Lpenalized(x) 
    \defeq
    \max_{q\in\Delta^N} \sum_{i\in[N]} \prn*{q_i \ell_i(x) - \psi(q_i)}
     = \min_{y\in\R}\crl[\Bigg]{ 
     	\Upsilon(x,y)\defeq \sum_{i \in [N]} \divergenceFunction^*(\ell_i(x) - Gy) + Gy
     }
 \end{equation}
where for the last equality we again used Lagrange duality, 
with $\divergenceFunction^*(v) \defeq \max_{t \in \text{dom}(\divergenceFunction)}\crl*{vt - \divergenceFunction(t)}$  the Fenchel dual of $\psi$ (for more details see \Cref{ssec:dualFormulation}).
We show that under weak assumptions (introducing logarithmic dependence on bounds on $f$ and the losses) we can solve the constrained problem \eqref{eq:mainProblemfDiverDRO} to accuracy $\epsilon$ by computing a polylogarithmic number of $O(\epsilon)$-accurate minimizers of \eqref{eq:penalizedFdiver}; see \Cref{ssec:RegConstraintProblem} for  details.
Thus, for the remainder of this section we focus on minimizing $\Lpenalized$ for arbitrary convex $\psi:\R_+ \to \R$.

\subsection{Stabilizing the gradient estimator}\label{ssec:Stability of the likelihood ratio}
While minimizing~\eqref{eq:penalizedFdiver} can be viewed as ERM (over $x$ and $y$), straightforward application of SGD does not solve it efficiently. To see  this, consider the standard gradient estimator formed by sampling $i\sim\mathsf{Unif}([N])$ and taking $ \gradx = N{\divergenceFunction^{*}}'\prn*{\ell_i(x)-Gy}\nabla \ell_i(x)$ and $ \grady = G\prn[\big]{1 - N{\divergenceFunction^{*}}'\prn*{\ell_i(x)-Gy}}$. For general $\divergenceFunction$, this estimator will have unbounded second moments, and therefore SGD using them would lack a convergence guarantee. 
As an extreme example, consider $\divergenceFunction = 0$ (corresponding to minimizing the maximum loss) whose conjugate function $\divergenceFunction^*(v)$ is $0$ for $v\le 0$ and $\infty$ for $v>0$, leading to meaningless stochastic gradients. 

We obtain bounded gradient estimates in two steps. First, we find a better distribution for $i$ using a reference point $\bx\in\xset$ with corresponding $\by = \argmin_{y\in\R}\Upsilon(\bx, y)$. Namely, we note that the optimality condition for $\by$ implies that ${\divergenceFunction^*}'(\ell_i(\bx)-\by)$ is a pmf over $[N]$. Therefore, we may sample $i\sim {\divergenceFunction^*}'(\ell_i(\bx)-\by)$ and estimate the gradient of $\Upsilon$ at $(x,y)$ using $\gradx = \rho_i(x,y)\grad \ell_i(x)$ and $\grady = G\prn[\big]{1 - \rho_i(x,y)}$, where $\rho_i(x,y)=\frac{{\divergenceFunction^*}'(\ell_i(x)-y)}{{\divergenceFunction^*}'(\ell_i(\bx)-\by)}$. However, for general $\divergenceFunction$ (and $\divergenceFunction=0$ in particular), the ratio $\rho_i(x,y)$ can be unbounded even when $x,y$ are arbitrarily close to $\bx,\by$. 

Our second step ensures that $\rho_i(x,y)$ is bounded around $\bx,\by$ by adding a small negative entropy term to $\divergenceFunction$, defining
\begin{equation}\label{eq:phieps}
    \psieps(q) \defeq \divergenceFunction(q) + \epsilon'q \log q
    \mbox{~~where~~}
    \eps' \defeq \frac{\eps}{2 \log N},
\end{equation}
and
\begin{equation}\label{eq:upsiloneps defintion}
    \Leps(x) = \min_{y
    	\in \R}\upsiloneps(x,y) \text{~~ with~~}  \upsiloneps(x,y) \defeq \sum_{i \in [N]} \psieps^*(\ell_i(x) - Gy) + Gy.
\end{equation}
Due to our choice of $\epsilon'$, we have $\abs{\Lpenalized(x)- \Leps(x)} \le \eps/2$ for all $x \in \R^d$, and consequently an $\epsilon/2$-accurate minimizer of $\Leps$ is also an $\epsilon$-accurate for $\Lpenalized$ (see \Cref{lem:SMapproxMaxFdiv} in \Cref{ssec:LpsiProperties}). 
When $\divergenceFunction=0$ we have $\psieps^{*}(v)=e^{(v-1)/\epsilon'}$ and therefore the corresponding $\rho_i(x,y)=e^{(\ell_i(x)-\ell_i(\bx)-G(y-\by))/\epsilon'}$.%
\footnote{
	Entropy regularization of $\psi=0$ is especially nice since 
	$\frac{{\psieps^{*}}'(\ell_i(x)-Gy)}{\sum_j {\psieps^{*}}'(\ell_j(x)-Gy)}=\frac{e^{\ell_i(x)/\epsilon'}}{\sum_j e^{\ell_j(x)/\epsilon'}}$ is independent of $y$. This fact allows one to  minimize $\Leps$ over $x$ directly via either the exponentiated softmax trick (as in \Cref{sec:groupDRO} and~\cite{carmon2021thinking}) or rejection sampling (as in~\cite{asi2021stochastic}). For general $\divergenceFunction$, however, minimizing $\upsiloneps$ over both $x$ and $y$ is essential.
} The following lemma, which might be of independent interest, shows that the same conclusion holds for any convex $\divergenceFunction$.
\begin{restatable}{lem}{psiProperties}\label{lem:psiProperties}
    For any convex $\divergenceFunction:\R_+\to \R$ and
    $\psieps$ defined in~\eqref{eq:penalizedFdiver},
      $\log\prn[\big]{{\psieps^{*}}'(\cdot)}$ is $\frac{1}{\epsilon'}$-Lipschitz.
\end{restatable}
\noindent
See proof in \Cref{ssec:StabilityLikelihoodProofs}. 
Thus, ${{\psieps^{*}}'(v)}/{{\psieps^{*}}'(\bar{v})} = e^{\log {\psieps^{*}}'(v) - \log {\psieps^{*}}'(\bar{v}) } \le e^{(v-\bar{v})/\epsilon'}$ and $\rho_i(x,y)\le e^{(\ell_i(x)-\ell_i(\bx)-G(y-\by))/\epsilon'}$ continues to hold. Therefore, if $\abs*{y-\by}\le \epsilon'/G = \reps$ and
$x\in\ball_{\reps}(\bx)$ (so that $\abs{\ell_i(x)-\ell_i(\bx)}\le \epsilon'$ if $\ell_i$ satisfies \Cref{assumption:GlobalAssumption}), we have the bound $\rho_i(x,y)\le e^2$. 

It remains to show that we may indeed restrict $y$ to be within distance $\reps$ from $\by$. To this end, we make the following observation which plays a key part in our analysis and might also be of independent interest (see proof in \Cref{ssec:StabilityLikelihoodProofs}).
\begin{restatable}{lem}{BoundEtaDistance}\label{lem:BoundEtaDistance}
	For $y^\star(x) = \argmin_{y\in\R} \upsiloneps(x,y)$, we have $\abs{y^\star(x)-y^\star(x')} \le \frac{1}{G}\norm*{\ell(x)-\ell(x')}_\infty$ for all $x,x'\in\xset$. Moreover, if each $\ell_i$ is $G$-Lipschitz, we have $\abs{y^\star(x)-y^\star(x')}\le \norm{x-x'}$.
\end{restatable}
\noindent
\Cref{lem:BoundEtaDistance} implies that  $x^\star,y^\star=\argmin_{x\in\ball_{\reps}(\bx), y\in\R}\upsiloneps(x,y)$ satisfy $\abs{y^\star - \by} \le \norm{x^\star-\bx}\le\reps$. Therefore, when minimizing $\upsiloneps$ (or any regularized version of it) inside the ball $\ball_{\reps}(\bx)$, we may restrict $y$ to $[\by-\reps,\by+\reps]$ without loss of generality. We also note that \Cref{lem:BoundEtaDistance} holds for all values of $\epsilon$ and is therefore valid even without entropy regularization.

\subsection{BROO implementation for $f$-divergence DRO with non-smooth losses}\label{ssec:The gradient estimator}
By the discussion above, to implement a BROO for $\Leps(x)$ (with radius $\reps=\epsilon'/G$, regularization $\lambda$, and query $\bx\in\xset$) it suffices to minimize $\upsilonepsreg(x,y)\defeq \upsiloneps(x,y)+\frac{\lambda}{2}\norm{x-\bx}^2$ over $x\in\ball_{\reps}(\bx)$ and $y\in [\by-\reps,\by+\reps]$, where $\by=\argmin_{y \in \R} \upsiloneps(\bx,y)$. To that end we estimate the gradient of $\upsilonepsreg(x,y)$ as follows. Letting $\bp_i =  {\psieps^{*}}'\prn*{\ell_i(\bx)-G \by}$ (making $\bp$ a pmf by optimality of $\by$), we sample $i \sim \bp$ and set
\begin{equation}\label{eq:dualGradients}
    \gradx(x,y) = \frac{{\psieps^{*}}'(\ell_i(x)-Gy)}{\bp_i} \nabla \ell_i(x,y) \text{~~and~~}
	\grady(x,y) = G \prn*{ 1 -\frac{{\psieps^{*}}'(\ell_i(x)-Gy)}{\bp_i}}.
\end{equation}  
\Cref{lem:psiProperties} implies the following bounds on our gradient estimator; see proof in \Cref{ssec:DualEpochSGDProofs}.
\begin{restatable}{lem}{dualGradEstProperties}\label{lem:dualGradEstProperties} 
	Let each $\ell_i$ be $G$-Lipschitz, let $\bx\in\xset$ and $\by = \argmin_{y \in \R} \upsilonepsreg(\bx,y)$. Let $\reps=\frac{\eps'}{G}$, then for all $x \in \ball_{\reps}(\bx)$ and $y \in \brk*{\by - \reps, \by + \reps}$, the gradient estimators $\gradx$ and $\grady$  satisfy 
    the following properties
    \begin{enumerate}\label{eq:gradientsDefintion}
        \item {$\E_{i \sim \bp_i}\brk*{\gradx(x,y)} = \nabla_x \upsiloneps(x, y) \text{~~and~~} \E_{i \sim \bp_i}\brk*{\grady(x,y)} = \nabla_y \upsiloneps(x, y).$}
        \item {$\E_{i \sim \bp_i}\norm*{\gradx(x,y)}^2 \le e^4 G^2  \text{~~and~~} \E_{i \sim \bp_i}\abs*{\grady(x,y)}^2 \le e^4 G^2$.} 
    \end{enumerate}
\end{restatable}

To implement the BROO using our gradient estimator we develop a variant of the Epoch-SGD algorithm of \citet{hazan2014beyond} (\Cref{alg:dualEpochSGD} in \Cref{ssec:DualEpochSGDProofs}). Similarly to Epoch-SGD, we apply standard SGD on $\upsilonepsreg$ (with gradient estimator~\eqref{eq:dualGradients}) in ``epochs'' whose length doubles in every repetition. Our algorithm differs slightly in how each epoch is initialized. Standard Epoch-SGD initializes with the average of the previous epoch's iterates, and strong convexity shows that the suboptimality and distance to the optimum shrink by a constant factor after every epoch. However, since $\upsilonepsreg$ is strongly convex only in $x$ and not in $y$, we cannot directly use this scheme. Instead, we set the initial $y$ variable to be  $\argmin_{y}\upsilonepsreg(x',y)$, where $x'$ is the initial $x$ variable still defined as the previous epoch's average; this initialization has complexity $N$, but we only preform it a logarithmic number of times. Using our initialization scheme and \Cref{lem:BoundEtaDistance}, we recover the original Epoch-SGD contraction argument, yielding the following complexity bound (see proof in \Cref{ssec:DualEpochSGDProofs}).
\begin{restatable}{theorem}{BROOcomplexityDualProblem}\label{thm:BROOcomplexityDualProblem}
    Let each $\ell_{i}$ satisfy \Cref{assumption:GlobalAssumption}. Let $\epsilon, \lambda , \delta >0 $, and $r_\epsilon = \epsilon /(2G\log N)$. For any query point $\bx  \in \R^d$, regularization strength $\lambda \le O(G / r_\epsilon)$ and accuracy $\delta < \reps/2$, 
    \Cref{alg:dualEpochSGD} outputs a valid $r_\epsilon$-BROO response for $\Leps$ and has complexity
    $  \broocost = O\prn[\big]{ \frac{G^2}{\lambda^2\delta^2}+  N \log\prn*{\frac{\reps}{\delta}}} $.
    Consequently, the complexity of finding an $\epsilon$-suboptimal minimizer of $\Lpenalized$~\eqref{eq:penalizedFdiver} with probability at least $\half$ is 
    \begin{align*}
        O \prn*{ 
            N \prn*{\frac{GR}{\epsilon} }^{2/3} \log^{11/3} \innerLog + \prn*{\frac{GR}{\epsilon} }^2 \log^2 \innerLog}
              \text{~~where~~} \innerLog \defeq N\frac{GR}{\eps}.
    \end{align*}
\end{restatable}

\subsection{Accelerated variance reduction for smooth losses}\label{ssec:Accelerated variance reduction}
In this section, we take advantage of the following smoothness assumption. 
\begin{assumption}\label{assumption:SmoothFunctions}
    For every $i \in [N]$ the loss $\ell_i$ is $L$-smooth, i.e., has $L$-Lipschitz gradient.
\end{assumption}
To implement the BROO via variance reduction techniques, we first rewrite the objective function as a weighted finite sum:
\begin{equation*}
    \upsilonepsreg(x, y) = \sum_{i \in [N]} \bp_i \upsilon_i(x, y)  
    \text{~~where~~}
    \upsilon_i(x, y) \defeq \frac{\psieps^*\prn*{\ell_i(x) - G y}}{\bp_i} +  Gy + \frac{\lambda}{2}\norm{x-\bx}^2
\end{equation*}
and, as before $\bp_i = \psieps^*\prn*{\ell_i(\bx) - G \by}$ for some ball center $\bx\in\xset$ and $\by=\argmin_{y \in \R} \upsiloneps(\bx,y)$. In the following lemma, we bound the smoothness of the functions $\upsilon_i$, deferring  the proof to \Cref{ssec:VarianceReductionSVRGproofs}. 
\begin{restatable}{lem}{smoothnessOfGamma}\label{lem:smoothnessOfGamma}
    For any $i\in[N]$, let $\ell_i $ be $G$-Lipschitz and $L$ -smooth, let $\reps=\frac{\eps'}{G}$ and $\lambda=O\prn[\big]{\frac{G}{\reps}}$. The restriction of $\upsilon_i$ to $x\in\ball_{\reps}(\bx)$ and $y\in[\by-\reps,\by+\reps]$ is $O\prn*{G}$-Lipschitz and $O\prn[\big]{L +\frac{G^2}{\epsilon'} }$-smooth.
\end{restatable}

Since $\upsilonepsreg$ is a finite sum of smooth functions, we can obtain reduced-variance gradient estimates by the standard SVRG technique~\cite{johnson2013accelerating}. For any reference point $x',y'$ (not necessarily equal to $\bx,\by$), the estimator is
\begin{align}\label{eq:katyushaSmoothGradEst}
	\hat{g}_{x',y'}(x,y) = \nabla \upsiloneps(x',y') + \nabla \upsilon_i(x,y) -  \nabla \upsilon_i(x',y'),
\end{align}
where $\grad$ is with respect to the vector $[x,y]$. 

Similar to non-smooth case, obtaining an efficient BROO implementation is complicated by the fact that $\upsilonepsreg$ is strongly-convex in $x$ but not in $y$. Our solution is also similar: we propose a restart scheme and minimize over $y$ exactly between restarts (\Cref{alg:DualKatyusha} in \Cref{ssec:VarianceReductionSVRGproofs}). More precisely, we repeatedly apply an accelerated variance reduction scheme that does not require strong convexity, such as Katyusha$^{ns}$~\cite{allen2018katyusha}, each time with complexity budget $\Otil{N + \sqrt{N L'/\lambda}}$, where $L'=L+G^2/\epsilon'$. We start each repetition by the $x$ variable output by the previous Katyusha$^{ns}$ call, and with $y=\argmin_{y\in\R}\upsilonepsreg(y,x)$ for that $x$. Using \Cref{lem:BoundEtaDistance} in lieu of strong-convexity in $y$, we show that error halves after each restart, and therefore a logarithmic number of restarts suffices. We arrive at the following complexity bound (see proof in \Cref{ssec:VarianceReductionSVRGproofs}).

\begin{restatable}{theorem}{DualSVRGBROOcomplexity}\label{thm:DualSVRGBROOcomplexity}
    Let each $\ell_{i}$ satisfy \Cref{assumption:GlobalAssumption,assumption:SmoothFunctions}, 
    let $\epsilon,\lambda, \delta >0 $, and $\reps = \frac{\eps}{2G\log N}$. For any query point $\bx \in \R^d$, regularization strength $\lambda \le O(\frac{G}{r_\epsilon})$ and accuracy $\delta$, \Cref{alg:DualKatyusha} outputs a valid $\reps$-BROO response for $\Leps$ 
    and has complexity 
    $        \broocost = O\prn[\Big]{ \prn[\Big]{N+\frac{\sqrt{N}\prn*{G+\sqrt{\epsilon'L}}}{\sqrt{\lambda\epsilon'}}}\log{\frac{G r_{\eps}}{\lambda \delta^2}}}$. Consequently,  
    the complexity of finding an $\epsilon$-suboptimal minimizer of $\Lpenalized$~\eqref{eq:penalizedFdiver} with probability at least $\half$ is
    \begin{align*}
        O\prn*{ N\prn*{\frac{GR}{\epsilon}}^{2/3}\log^{14/3} \innerLog+ \sqrt{N}\prn*{\frac{GR}{\epsilon} + \sqrt{\frac{L R^2}{\epsilon}}}\log^{5/2} \innerLog }
        \text{~~where~~} \innerLog \defeq N\frac{GR}{\eps}.
    \end{align*}
\end{restatable} 
\section{Discussion}\label{sec:discussion}
We now discuss possible improvements and extensions of our results.  

First, it would be interesting to extend our approach to DRO objectives
$\max_{q\in\uset} \sum_{i\in [N]} q_i\ell_i(x)$ with uncertainty set $\uset$ that is an arbitrary subset of the simplex. While the subgradient method, the primal-dual method (\Cref{ssec:PrimalDualRegretBound}), and ``AGD on softmax'' (\Cref{ssec:AGDonSoftmax}) all apply to any $\uset\subseteq \Delta^N$, our methods strongly rely on the structure of $\uset$ induced by Group- $f$-divergence DRO, and extending them to unstructured $\uset$'s seems challenging.

Second, it would be interesting to generalize our results in the ``opposite'' direction of getting better complexity bounds for problems with additional structure. For Group-DRO our bounds are essentially optimal when the number of groups $\ngroups=\Omega(N)$, but are suboptimal when $\ngroups=O(1)$.  
We leave it as a question for further research if it is possible to obtain a stronger bound such as $\widetilde{O}\prn*{N+M\eps^{-2/3}+\eps^{-2}}$, which recovers our result for $M=N$ but improves on it for smaller values of $M$.
Taking CVaR at level $\alpha$ as a special case of $f$-divergence DRO, our bounds are optimal when $\alpha$ is close to $1/N$ but suboptimal for larger value of $\alpha$; it would be interesting to obtain bounds such as $\widetilde{O}\prn*{N+\alpha^{-1}\eps^{-2/3}+\eps^{-2}}$.

A third possible extension of our research is DRO in the non-convex setting. 
For this purpose, it might be possible to use the technique of \citet{carmon2017convex} for turning accelerated convex optimization algorithms to improved-complexity methods for smooth non-convex optimization.

Finally, we note that turning the algorithms we propose into practical DRO methods faces several challenges. A main challenge is the costly bisection procedure common to all Monteiro-Svaiter-type acceleration schemes~\cite{monteiro2013accelerated,gasnikov19near,carmon2020acceleration,song2021unified}. 
Moreover, similarly to~\cite{carmon2021thinking}, our ball radius $\reps$ and smoothing level $\epsilon'$ depend on the desired accuracy $\epsilon$; a more adaptive setting for these parameters is likely important. 
Our work indicates that the ball optimization approach offers significant complexity gains for DRO, motivating future research on translating it into practice.

\arxiv{\newcommand{\acks}[1]{\section*{Acknowledgment} #1}}
\acks{
	This research was partially supported by the Israeli Science Foundation (ISF) grant no.\ 2486/21, the Len Blavatnik and the Blavatnik Family foundation, and The Yandex Initiative for Machine Learning.
}

\arxiv{\bibliographystyle{abbrvnat}}

\appendix
\arxiv{\part*{Appendix}} 
\section{Alternative algorithms for solving DRO problems}\label{sec:AppOther}
\subsection{Primal-dual stochastic mirror descent}\label{ssec:PrimalDualRegretBound}
In this section, we present a primal-dual method capable of solving all the DRO problems our paper considers, under an additional assumption of bounded losses: for every $j$ and $x$ we assume $\abs*{\ell_j(x)} \le B_\ell$.  
Consider the primal-dual problem 
\begin{equation}\label{eq:FunctionPrimalDual}
     \minimize_{x\in \xset} \max_{q\in \uset} \crl[\Bigg]{\mc{L}_{\mathrm{pd}}(x,q)\defeq \sum_{i \in [m]} q_i \mc{L}_i(x) }
\end{equation}
where $\xset\subseteq B_R(x_0)$ is a closed convex set as before and $\uset$ is now an arbitrary closed convex subset of the simplex $\Delta^m$ and $\mc{L}_i(x) = \sum_{j \in [N]}w_{ij} \ell_j(x)$ are ``group losses'' with $w_i \in \Delta^N$ for every $i\in [m]$. This formulation subsumes both Group DRO (where $m=M$ and $\uset=\Delta^M$) and $f$-divergence DRO (where $m=N$, $\mc{L}_i(x) = \ell_i(x)$, and $\uset$ is an $f$-divergence ball).

As discussed in the introduction, several works have proposed primal-dual methods for DRO, but we could not find in these works the precise rate we prove here (in \Cref{prop:primal-dual} below) in its full generality. Our proof is a straightforward specialization of the more general results of \citet{carmon2020coordinate}.

The particular algorithm we consider is primal-dual stochastic mirror descent, with distances generated by the squared Euclidean norm on $\xset$ and entropy on $\uset$ and gradient clipping for the $\uset$ iterates, corresponding to the following recursion:
\begin{equation}\label{eq:primalDualAlg}
	\begin{aligned}
	x_{t+1} &= \argmin_{x\in\xset}\crl*{ \inner{\eta \gx(x_t,q_t)}{x} + \frac{
		\log m}{R^2}\norm{x-x_t}^2}
	~~\mbox{and}~~ \\
	q_{t+1} &= \argmax_{x\in\xset}\crl*{ \inner{\Pi_{[-1,1]^m}(\eta \gq(x_t,q_t))}{q} + \sum_{i\in [m]} [q]_i \log \frac{[q]_i}{[q_t]_i}},
	\end{aligned}
\end{equation}
where $\eta$ is a step-size parameter, $\Pi_{[-1,1]^m}$ is the Euclidean projection to the unit box (i.e., entry-wise clipping to $[-1,1]$), and $\gx$ and $\gq$ are unbiased estimators for $\grad_x \mc{L}_{\mathrm{pd}}$ and $\grad_q \mc{L}_{\mathrm{pd}}$, respectively, given by 
\begin{align}\label{eq:MirrorDescentGrad}
	& \gx(z) \defeq \nabla \ell_j(\zx)  \text{~~with $i \sim z^q$ and $j \sim w_i$ }
	\\ \nonumber & 
	\gq(z) \defeq m  \ell_j(\zx) e_i  \text{~~with $i \sim \mathsf{Unif}([m])$ and $j \sim w_i$},
\end{align}
with $e_i\in\R^m$ being the $i$th standard basis vector in $\R^m$.

This method yields the following convergence guarantees.
\begin{prop}\label{prop:primal-dual}
	Assume that each  $\ell_j$ convex and $G$-Lipschitz and satisfies $\abs*{\ell_j(x)} \le B_\ell$ for every $x\in\xset$.  
	For $T\in \N$ let  $\bar{x}_T = \frac{1}{T}\sum_{t=0}^T x_t$ and $\bar{q}_T = \frac{1}{T}\sum_{t=0}^T q_t$, where $\{x_t, q_t\}$ are the iterates  defined in~\eqref{eq:primalDualAlg}, with $\eta = O\prn*{\frac{\epsilon \log m}{G^2 R^2 + m B_\ell^2}}$. Then, for any $\epsilon >0$, if $T \ge O\prn*{\frac{G^2 R^2 + B_\ell^2 m \log m }{\eps^2}}$ we have that 
	\begin{equation*}
		\E\mc{L}_\textup{DRO}(\bx_T) - \min_{\xopt\in\xset} \mc{L}_\textup{DRO}(\xopt) \le  
		\E \max_{q\in\uset} \mc{L}_{\mathrm{pd}}(\bx_T,q) - \E \min_{x\in\xset}  \mc{L}_{\mathrm{pd}}(x,\bar{q}_T)  \le\eps,
	\end{equation*}
	where $\mc{L}_\textup{DRO}(x) = \max_{q\in\uset} \mc{L}_{\mathrm{pd}}(x,q)$.
\end{prop}
\begin{proof}
The proposition is a direct corollary of a more general result by  \citet{carmon2020coordinate}. To show this, we rewrite the iterations~\eqref{eq:primalDualAlg} using ``local norm setup'' notation of~\cite{carmon2020coordinate}. In particular, let $\zset = \xset\times \uset$ and for every $z = (\zx, \zq) \in \zset$ define the local norm of $\delta \in \zset^*$ at $z$ as
\[\norm*{\delta}_z \defeq \sqrt{ \frac{R^2}{\log m}\norm*{\deltax}^2_2 + \sum_{i \in [m]} [\zq]_i [\deltaq]_i^2}. \] 
In addition, we define the generating distance function 
\[ r(z) = r(\zx,\zq) \defeq \frac{\log m}{R^2}\norm*{\zx}^2_2 + \sum_{i \in [m]} [\zq]_i \log[\zq]_i\]
and write its associated Bregman divergence as $V_{z}(z') = r(z')-r(z)-\inner{\grad r(z)}{z'-z}$. Next, we let $\Theta \defeq \max_{z,z' \in \mc{Z}} \crl*{r(z)- r(z')}$ and observe that
since $\zx \in \xset \subset \ball_R(x_0)$ and $\zq \in \uset \subset \Delta^m$, then $\Theta = 2\log m$. Last, we define the function clip$: \mc{Z}^* \rightarrow \mc{Z}^*$ as follows:
\[ \mathrm{clip}(\deltax,\deltaq) \defeq \prn*{\deltax, \Pi_{[-1,1]^m}(\deltaq)},\]
where $\Pi_{[-1,1]^m}$ denotes entry-wise clipping to $[-1,1]$. By an argument directly analogous to \cite[Proposition 1]{carmon2020coordinate}, the quintuplet $\prn*{\mc{Z}, \norm{\cdot}., r, \Theta, \mathrm{clip}}$ forms a valid local norm setup \cite[Definition 1]{carmon2020coordinate}.

With this notation, the iterations~\eqref{eq:primalDualAlg} have the concise form  
\begin{align}\label{eq:MirrorDescentUpdates}
	z_{t+1} = \argmin_{w \in \mc{Z}} \crl*{
		\tri*{\mathrm{clip}\prn*{\eta \ghat(z_t)}, w} +
		V_{z_t}(w)
	},
\end{align} 
where $ \ghat(z)\defeq \prn*{\gx, -\gq}$, with $\gx$ and $\gq$ as defined in~\eqref{eq:MirrorDescentGrad} above. It is then straight-forward to verify that $\E \ghat(z) = \prn*{\nabla_x \mc{L}_{\mathrm{pd}}(z), - \nabla_q \mc{L}_{\mathrm{pd}}(z)}$ for every $z\in\zset$, and that
\begin{align*}
	\E \brk*{\norm{\ghat(z)}^2_w} & = \E_{i \sim \zq, j \sim w_i} \brk*{\frac{
		R^2}{\log m}\norm*{\nabla \ell_j(\zx)}^2} + \E_{i \sim \mathrm{Unif}([m]), j \sim w_i}\brk*{ [\zq]_i m^2 \ell_j^2(\zx)}
	\le \frac{G^2 R^2}{\log m} + m B_\ell^2  
\end{align*}
for every $z,w\in\zset$. Therefore, $\ghat$ is an $L$-local estimator \cite[Definition 3]{carmon2020coordinate} with $L^2=G^2 R^2  / (log m) + m B_\ell^2$ so that $L^2\Theta = 2G^2 R^2 + 2B_\ell^2 m\log m$.  \Cref{prop:primal-dual} now follows immediately from \cite[Proposition 2]{carmon2020coordinate}.
\end{proof}

\subsection{AGD on the softmax: complexity bound}\label{ssec:AGDonSoftmax}
In this appendix we briefly develop the complexity guarantees of the ``AGD on softmax'' approach mentioned in the introduction. While the idea is well known, we could not find in the literature an analysis of the method for the general DRO setting (i.e., maximization over $q$ in arbitrary subsets of the simplex), so we provide it here.

We wish to minimize, over $x\in\xset$,
 \[\mc{L}_\textup{DRO}(x) \defeq \max_{q \in \uset}\sum_{i \in [N]}q_i \ell_i(x)\]
where each $\ell_i$ is convex, $G$-Lipschitz and $L$-smooth and $\uset$ is an arbitrary closed convex subset of the simplex $\Delta^N$; note that this includes both Group DRO and $f$-divergence DRO as special cases.
We define the approximation 
\[\widetilde{\mc{L}}_\textup{DRO}(x) \defeq  \max_{q \in \uset \subset \Delta^N} \crl*{ \sum_{i \in [N]}q_i \ell_i(x) - \eps' q_i \log q_i} \] 
with $\eps' = \frac{\eps}{2 \log N}$.
In addition, since $\sum_{i \in [N]}  q_i \log q_i \in [- \log N,0]$ we have that 
\begin{align*}
    \abs*{{\mc{L}}_\textup{DRO}(x) - \widetilde{\mc{L}}_\textup{DRO}(x)} &= \abs*{ \max_{q \in \uset \subset \Delta^N}\crl*{\sum_{i \in [N]}q_i \ell_i(x)} -  \max_{q \subset \Delta^N}\crl*{ \sum_{i \in [N]}q_i \ell_i(x) - \eps' q_i \log q_i}}
    \\ & \le  \abs*{ \sum_{i \in [N]}  \eps' q_i \log q_i} \le \eps/2.
\end{align*}
Thus for $x$ satisfying  $  \widetilde{\mc{L}}_\textup{DRO}(x)-\min_{x\in\xset} \widetilde{\mc{L}}_\textup{DRO}(x) \le \eps/2$
we have that ${\mc{L}}_\textup{DRO}(x)- \min_{\xopt\in\xset} {\mc{L}}_\textup{DRO}(\xopt) \le \eps $ as well.

Next, we show that $ \widetilde{\mc{L}}_\textup{DRO}$ is $\widetilde{O}\prn*{1/\eps}$-smooth when each $\ell_i$ is $O(1/\epsilon)$-smooth. 
The function $\Psi(q)= \sum_{i \in [N]}\eps' q_i \log (q_i)$ is $\eps'$-strongly convex w.r.t to the $\norm{\cdot}_1$ norm, therefore the conjugate function $\Psi^*(\cdot)$ is $\frac{1}{\eps'}$-smooth w.r.t to the dual norm $\norm{\cdot}_\infty$, such that 
\begin{equation} \label{eq:SmoothConj}
    \norm{\nabla \Psi^{*}(v)- \nabla \Psi^{*}(v')}_1 \le \frac{1}{\eps'}\norm*{v-v'}_\infty
\end{equation}
In addition, let $q^{\star}\prn*{\ell(x)} = \nabla  \Psi^{*}\prn*{\ell(x)}= \argmax_{q \in \uset \subset \Delta^N} \crl*{\ell(x)-\Psi(q)}$ and note that $ \widetilde{\mc{L}}_\textup{DRO}(x) = \Psi^{*}\prn*{\ell(x)}$. 
Using this for every $x,y \in \xset$ we have 
\begin{align*}
    \norm*{ \nabla  \widetilde{\mc{L}}_\textup{DRO}(x) - \nabla  \widetilde{\mc{L}}_\textup{DRO}(y)} &= \norm*{ \sum_{i \in [N]} \nabla \ell_i(x)q^{\star}_i\prn*{\ell(x)} -  \nabla \ell_i(y)q^{\star}_i\prn*{\ell(y)}}
    \\ & \le \norm*{ \sum_{i \in [N]} \nabla \ell_i(x) \brk*{q^{\star}_i\prn*{\ell(x)}-q^{\star}_i\prn*{\ell(y)} }}  + \norm*{\sum_{i \in [N]} q^{\star}_i\prn*{\ell(y)} \brk*{\nabla \ell_i(x) - \nabla \ell_i(y)}}
    \\ & \stackrel{(\romannumeral1)}{\le} G\norm{q^{\star}\prn*{\ell(x)}-q^{\star}\prn*{\ell(y)} }_1 + L\norm{x-y}
    \\ & \stackrel{(\romannumeral2)}{\le} \frac{G}{\eps'}\norm*{\ell(x)-\ell(y)}_\infty +  L\norm{x-y}
    \\ & \stackrel{(\romannumeral3)}{\le} \prn*{\frac{G^2}{\eps'} + L } \norm{x-y}
\end{align*}
where $(\romannumeral1)$ follows since every $\ell_i$ is $G$ Lipschitz and $L$ smooth, in addition for every $\ell_i(x) \in \R$ we have that $q^{\star}(\ell(x)) \in \uset \subset \Delta^N$, therefore $\sum_{i \in [N]}q^{\star}_i(\ell(x)) = 1$,  
$(\romannumeral2)$ follows from the inequality in \eqref{eq:SmoothConj} and $(\romannumeral3)$ follows since each $\ell_i$ is $G$-Lipschitz. 

Since $ \widetilde{\mc{L}}_\textup{DRO}$ is $\widetilde{L}=L+\frac{G^2}{\epsilon'}$-smooth, Nesterov's accelerated gradient descent \cite{nesterov2005smooth} method is efficient for minimizing it. 
This method finds $x$ such that  $ \widetilde{\mc{L}}_\textup{DRO}(x)- \min_x  \widetilde{\mc{L}}_\textup{DRO}(x) \le \eps / 2$ 
with $O\prn[\Big]{\sqrt{{\widetilde{L}R^2}/{\eps}}}=\widetilde{O}\prn*{\frac{GR}{\eps}}$ iterations when $L=O(G^2/\epsilon)$. Note that in every iteration we need to compute the full gradient 
of $ \widetilde{\mc{L}}_\textup{DRO}$, which requires to evaluate each $\nabla \ell_i$ and $\ell_i$. Therefore, in the weakly smooth setting the complexity of this method is $\widetilde{O}(\frac{N GR}{\epsilon})$.

\section{Proof of \Cref{prop:MainProp}}\label{sec:Preliminariesproofs}
In this section we provide the proof for  \Cref{prop:MainProp} that follows from the analysis of \Cref{alg:acceleratedProxPoint}; this section closely follows \citet{asi2021stochastic}, and we refer the readers to that paper for a more detailed exposition.

\begin{algorithm2e}[t]
	\DontPrintSemicolon
    \LinesNumbered
	\caption{Stochastic accelerated proximal point method}	\label{alg:acceleratedProxPoint}
	\KwInput{BROO $\oracle{\cdot}$, $T_{\max}$, initialization $x_0=v_0$ and $A_0 \ge 0$}
	\KwParameters{Approximation parameters $ \crl*{\delta_k, \beta_k, \sigma_k}$, stopping parameters $A_{\max}$ and $\kmax$ }
	 \For{$k = 0,1,2,\cdots $}{
		    $\lambda_{k+1} = \linesearch\prn*{x_k, v_k, A_k}$\label{line:next-lambda}\;
		   $a_{k+1} = \frac{1}{2\lambda_{k+1}}\sqrt{1+4\lambda_{k+1}A_k}$ and $A_{k+1} = A_k + a_{k+1}$\;
		   $y_k = \frac{A_k}{A_{k+1}}x_k + \frac{a_{k+1}}{A_{k+1}}v_k$\;
		   $x_{k+1} = \oracle[\lambda_{k+1}, \delta_{k+1}]{y_k} $\label{line:prox-point}\;
		    $g_{k+1} = \gradest\prn*{\oracle{\cdot}, y_k,  \lambda_k, \frac{\beta_{k+1}}{\lambda_{k+1}},\frac{\sigma_k^2}{\lambda_{k+1}}}$ \label{line:grad-est}\; 
		    $v_{k+1} = \text{Proj}_{\xset}\prn*{v_k - \half a_{k+1}g_{k+1}}$ 
	
	    \If{$A_{k+1} \ge A_{\max}$ \textbf{\textup{or}} $k+1 = \kmax$}{ \Return $x_{k+1}$}}
\end{algorithm2e}
\begin{algorithm2e}[t]
	\DontPrintSemicolon
    \LinesNumbered
	\caption{$\gradest$}	\label{alg:gradEst}
	\KwInput{BROO $\oracle{\cdot}$, query point $y$, regularization $\lambda$, bias  $\beta$, and variance $\sigma^2$. 
    }
	Set $T_{\max} = \frac{2 G^2}{\lambda^2  \min \crl*{\beta^2, \half\sigma^2}}$,  $T_0 = \frac{14 G^2 \log T_{\max}}{ \sigma^2}$, and  $\delta_0 = \frac{G}{\lambda \sqrt{T_0}}$\;
	$x_0 = \oracle[\lambda, \delta_0]{y_k}$\; \label{line:prox-estimator}
	Sample $J \sim \text{Geom}\prn*{\half}$\;
	\If {$2^J \le T_{\max}$}{
	  $\delta_J = \frac{G}{\lambda  \sqrt{2^J T_0}}$ \;
	   $ x_J , x_{J-1} = \oracle[\lambda,\delta_J]{y},\oracle[\lambda,\delta_{J-1}]{y}$\;
	  $\xhat = x_0 + 2^J\prn*{ x_J - x_{J-1}}$
	}
	\Else{
		  $\xhat = x_0 $
	}
	\Return $\lambda\prn*{y-\xhat}$
\end{algorithm2e}

We begin with a short description of \Cref{alg:acceleratedProxPoint}. 
This algorithm iteratively computes a $\frac{\lambda \delta^2}{2}$-approximate minimizer of $F_\lambda(x) = F(x) + \frac{\lambda}{2}\norm*{x-y}^2$ within a small ball of radius $r$ around $y$.  
To keep the ball constraint inactive it uses a bisection procedure that outputs the regularization strength value $\lambda$, such that for the minimizer $\xhat = \argmin_{x \in \xset} F_\lambda(x)$ 
with high probability we have $\norm*{\xhat-y} \le r$. It then computes a (nearly) unbiased gradient estimator of the Moreau envelope $M_\lambda(y) = \min_{x\in\xset}F_\lambda(x)$ and uses a momentum-like scheme to compute the  next ball center.   

\Cref{alg:acceleratedProxPoint} is a variant of the accelerated proximal point method in \cite[][Algorithm 4]{asi2021stochastic},
and we now describe the differences between the two.
The main difference is that the bias reduction scheme of \citet{asi2021stochastic} averages $\widetilde{O}\prn*{\frac{G^2}{\sigma^2}+1}$ calls to an estimator with accuracy $\delta'_J=O\prn*{\frac{G}{\lambda 2^{J/2}}}$
where $J \sim \text{Geom}\prn*{\frac{1}{2}, T_{\max}}$ and $G$ is the Lipschitz constant of $F$.  
In contrast, we use a single call to an estimator with higher accuracy $\delta_J = \widetilde{O}\prn*{\delta'_J / \sqrt{G^2/\sigma^2}}$; cf.\ the implementations of $\gradest$ subroutine in each algorithm. 
There are additional differences between the error tolerance settings of our algorithms, as described below.

The guarantees of \cite[Proposition 2]{asi2021stochastic} require the following choice of approximation parameters: 
\[ \varphi_k = \frac{\lambda_k \delta_k^2}{2} = \frac{\eps}{60 \lambda_k a_k} \text{~,~} \beta_k =\frac{\eps}{120R} \text{~and~} \sigma_k^2 = \frac{\eps}{60a_k}\]
which implies $\max_{k \le \kmax} \crl*{\lambda_k a_k \varphi_k + a_k \sigma_k^2 + 2R \beta_k} \le \frac{\eps}{20}$ (where $\beta_k$ in our notation is $\delta_k$ in the notation of \cite{asi2021stochastic}).
These parameters were chosen so that together with  \cite[Lemma 6]{asi2021stochastic} they give the following bound 
\begin{equation}\label{eq: Lemma6Bound}
    \E \brk*{A_K \prn*{F(x_K)-F(x_\star)} +\frac{1}{6}\sum_{i \le K}\lambda_iA_i \norm*{\xhat_i - y_{i-1}}^2 } \le A_0 \prn*{F(x_0)-F(x_\star)} + \frac{\eps}{20}\E A_K+ R^2 ,
\end{equation}
which is a key component in the proof of  \cite[Proposition 2]{asi2021stochastic}.
However, for improved efficiency we set different parameters:
\[ \varphi_k = \frac{\lambda_k \delta_k^2}{2} =\frac{\lambda_k r^2}{900 \log^3 \prn*{\frac{GR^2}{\eps r}}} \text{~and~} \sigma_k^2 = \frac{\lambda_k^2 r^2}{900 \log^3 \prn*{\frac{GR^2}{\eps r}}} .\]
To obtain the guarantees of \cite[Proposition 2]{asi2021stochastic} for our implementation, in the following lemma we reprove \cite[Lemma 6]{asi2021stochastic} with our parameters and show the same bound as in \eqref{eq: Lemma6Bound} (with a slightly different constant factor).

\begin{lem}[{modification of \cite[Lemma 6]{asi2021stochastic}}]\label{lem:lemma6}
    Let $F$ satisfy \Cref{assumption:GlobalAssumption} with a minimizer $x_\star$.
    Let 
    \[ \varphi_k = \frac{\lambda_k r^2}{900\log^{3}\prn*{\frac{GR^2}{r\eps}}} \text{~,~} \sigma_k^2 = \frac{\lambda_k^2 r^2}{900\log^{3}\prn*{\frac{GR^2}{r\eps}}} \text{~,~} \beta_k = \frac{\eps}{120R}\text{~,~} A_0 =\frac{R}{G} \text{~and~} A_{\max}=\frac{9R^2}{\eps}. \] 
    Define $\xhat_k \defeq \argmin_{x \in \xset}\crl*{F(x)+ \frac{\lambda}{2}\norm*{x-y_k}^2}$ and assume that for each $k$ we have $\norm{\xhat_k - y_{k-1}} \le r$ and that one of the following must occur 
    \begin{enumerate}
        \item  $\lambda_k < 2 \lmin = \frac{2\eps}{r^{4/3}R^{2/3}}\log^{2}\prn*{\frac{GR^2}{r\eps}}$, or 
        \item $\norm{\xhat_k - y_{k-1}} \ge \frac{3}{4}r$.
    \end{enumerate} 
    Then we have that 
    \begin{equation*}
        \E \brk*{A_K \prn*{F(x_K)- F(x_\star)-\frac{\eps}{20}} + \frac{1}{12} \sum_{i \le K} \lambda_i A_i \norm{\xhat_i -y_{i-1}}^2} \le A_0 \prn*{F(x_0) - F(x_\star)} +R^2.
    \end{equation*}
\end{lem}
\begin{proof}
    Define the filtration 
    \[ \mc{F}_k = \sigma\prn*{x_1, v_1, A_1, \zeta_1, \ldots, x_k, v_k, A_k,\zeta_k}\] 
    where $\zeta_i$ is the internal randomness of $\linesearch\prn*{x_k, v_k, A_k}$ and note that  
    $A_{k+1}, y_k, \xhat_{k+1} $ are deterministic when conditioned on $x_k, v_k, A_k,\zeta_k$. 
    Following the proof of  \cite[][Lemma 6]{asi2021stochastic}, we define  
   \begin{equation*}
   	 M_k = A_k \prn*{F(x_k)-F(x_\star)-\frac{\eps}{20}} + \frac{1}{12}\sum_{i \le k}\lambda_i A_i \norm*{\xhat_i -y_{i-1}}^2 + \norm*{v_k - x_\star}^2
   \end{equation*}
    and show it is a supermartingle adapted to filteration $\mathcal{F}_k$. 
    From \cite[][Lemma 5]{asi2021stochastic}  we  have
    \begin{flalign}\label{eq:lemma5bound}
        \E\brk*{M_{k+1}| \mc{F}_k} & \le A_k \prn*{F(x_k)-F(x_\star)} + \norm{v_k -x_\star}^2 - \frac{1}{6}\lambda_{k+1}A_{k+1}\norm*{\xhat_{k+1}-y_k}^2 
        \nonumber \\ & + \mu_{k+1} - A_{k+1}\frac{\eps}{c} 
        + \frac{1}{12} \sum_{i \le k+1} \lambda_i A_i \norm{\xhat_{i}-y_{i-1}}^2
    \end{flalign}
    where
     \begin{equation*}
        \mu_{k+1} \defeq \lambda_{k+1}a^2_{k+1}\varphi_{k+1} + a_{k+1}^2 \sigma^2_{k+1} + 2Ra_{k+1}\beta_{k+1}.
     \end{equation*}
    Substituting the values of $\varphi_{k+1},  \sigma^2_{k+1}, \beta_{k+1}$ into the definition of $\mu_{k+1}$ gives 
    \begin{equation*}
        \mu_{k+1} = \frac{2 \lambda_{k+1}^2 a_{k+1}^2 r^2}{900 \log^3\prn*{\frac{GR^2}{r \eps}}} + a_{k+1}\frac{\eps}{60}.
    \end{equation*}
    Recall that $A_{k+1} = a_{k+1}^2 \lambda_{k+1}$, therefore
    \begin{equation*}
        \mu_{k+1} = \frac{2 \lambda_{k+1} A_{k+1} r^2}{900 \log^3\prn*{\frac{GR^2}{r \eps}}} + a_{k+1}\frac{\eps}{60} = \frac{2 a_{k+1} \lambda_{k+1}^{3/2} \sqrt{A_{k+1}} r^2}{900 \log^3\prn*{\frac{GR^2}{r \eps}}} + a_{k+1}\frac{\eps}{60}.
    \end{equation*}
    If $\norm{\xhat_k - y_{k-1}} \ge \frac{3}{4}r$, we have that
    \begin{equation*}
        \mu_{k+1} \le \frac{1}{12}\lambda_{k+1} A_{k+1} \norm{\hx_{k+1}-y_k}^2 + a_{k+1}\frac{\eps}{60}.
    \end{equation*}
    Else, if $\lambda_k < 2 \lmin = \frac{2\eps}{r^{4/3}R^{2/3}}\log^{2}\prn*{\frac{GR^2}{r\eps}}$, we have
    \begin{equation*}
        \mu_{k+1} <\frac{2 a_{k+1} 2\lmin^{3/2} \sqrt{A_{k+1}} r^2}{900 \log^3\prn*{\frac{GR^2}{r \eps}}} + a_{k+1}\frac{\eps}{60} = \frac{4 a_{k+1} \eps^{3/2} \sqrt{A_{k+1}} }{900 R} + a_{k+1}\frac{\eps}{60}.
    \end{equation*}
    Now, note that 
     $ \lmin \ge \widetilde{O}\prn*{\frac{\eps}{R^2}} \ge \frac{1}{A_{\max}} = \frac{\eps}{9R^2}$ and therefore
   $a_{K} = \sqrt{\frac{1}{\lambda_{K}^2}+ \frac{4A_{K-1}}{\lambda_{K}}} \le  1.2A_{\max}.$    
    From the definition of $A_k$ we have that  $A_{\kmax -1} \le A_{\max} $, therefore for every $k \le K$
    \begin{equation*}
        A_k \le A_K = a_K + A_{K-1} \le 2.2A_{\max}.
    \end{equation*}
    Thus, when $\lambda_k < 2 \lmin $ the bound on $\mu_{k+1}$ becomes 
    \begin{equation*}
        \mu_{k+1} < \frac{4 a_{k+1} \eps^{3/2} \sqrt{2.2\frac{R^2}{\eps}} }{900 R} + a_{k+1}\frac{\eps}{60} \le  a_{k+1}\frac{\eps}{20} \le \frac{1}{12}\lambda_{k+1} A_{k+1}\norm{\hx_{k+1}-y_k}^2 + a_{k+1}\frac{\eps}{20}
    \end{equation*}
    where the last inequality follows since $A_k \ge 0$ and $ \lambda_k \ge 0$.
    Therefore, for every $k \le K$ we have that  
    \[ \mu_{k+1} \le \frac{1}{12}\lambda_{k+1} A_{k+1}\norm{\hx_{k+1}-y_k}^2 + a_{k+1}\frac{\eps}{20}.\]
    Noting that  $\E \abs{M_k} < \infty$ and substituting the bound on $\mu_{k+1}$ into \eqref{eq:lemma5bound} we get 
    \begin{flalign*}
        \E\brk*{M_{k+1}| \mc{F}_k} & \le A_k \prn*{F(x_k)-F(x_\star) - \frac{\eps}{20}} + \norm{v_k -x_\star}^2
        + \frac{1}{12} \sum_{i \le k} \lambda_i A_i \norm{\xhat_{i}-y_{i-1}}^2 = M_k.
    \end{flalign*}
    Therefore $M_k$ is a supermartingle adapted to filtration $\mathcal{F}_k$. 
    Since $K$ is a stopping time adapted to filtration $\mathcal{F}_k$, by the optional stopping theorem for supermartingles
    we have 
    \begin{equation*}
        \E M_K \le M_0 = A_0 \prn*{F(x_0)-F(x_\star)-\frac{\eps}{20}} + \norm*{v_0 - x_\star}^2 \le A_0 \prn*{F(x_0)-F(x_\star)}+ R^2.
    \end{equation*}
\end{proof}
For \cref{line:next-lambda} of \Cref{alg:acceleratedProxPoint}, we use the same $\linesearch$ implementation of \cite{carmon2021thinking}. 
This implementation requires calling to a \emph{high-probability Ball Regularized Optimization Oracle} (high-probability BROO) and we give the definition of it bellow. 
\begin{definition}\label{def:HighProbBROO}
	An algorithm is a probability $1-p$ Ball Regularized Optimization Oracle of radius 
	$r$ ($r$-BROO) for function $F: \xset \to \R$ if for query point $\bx\in\xset$,probability $p$, regularization parameter $\lambda >0$ and desired accuracy $\delta>0$ it returns $\oracles(\bx)\in\xset$ 
    that with probability at least $1-p$ satisfies
    \begin{equation} \label{eq:BROOhighProbBound}
        {F(\oracles(\bx)) + \frac{\lambda}{2}\norm*{\oracles(\bx)- \bx}^2} \le \min_{x \in \ball_r(\bx)\cap \xset}\crl*{F(x)+\frac{\lambda}{2}\norm*{x-\bx}^2}+\frac{\lambda}{2}\delta^2.
    \end{equation}
\end{definition}
In the following lemma we give the complexity guarantee for a high probability $r$-BROO. 
\begin{lem}\label{lem:highProbBROO}
    Let $\ncost$ be the complexity of evaluating $F$ exactly, and $\broocost$ be an $r$-BROO implementation complexity. Then the complexity of implementing a probability $1-p$  $r$-BROO of \Cref{def:HighProbBROO} is
     \begin{align*}
        \log\prn*{\frac{1}{p}}\brk*{ \broocost [\frac{\delta}{\sqrt{2}}]+  \ncost }
        .
    \end{align*}
\end{lem}
\begin{proof}
    To obtain a high-probability $r$-BROO we run $\log_2\prn*{\frac{1}{p}}$ copies of $r$-BROO with query point $\bx$, regularization strength $\lambda$ and accuracy $\delta / \sqrt{2}$ and take the best output, i.e., the output with the minimal value of $F$.
    Applying Markov's inequality to a single run of $r$-BROO with output $x$, gives 
    \begin{align*}
        \P\prn*{F(x) + \frac{\lambda}{2}\norm*{x- \bx}^2 - \min_{x \in \ball_r(\bx)\cap \xset}\crl*{F(x)+\frac{\lambda}{2}\norm*{x-\bx}^2} \ge  \frac{\lambda}{2}\delta^2} \le \half
    \end{align*}
    therefore with probability at least $1-\prn*{\frac{1}{2}}^{\log_2 \prn*{\frac{1}{p}}} = 1- p$,  the best output $x'$
    satisfies 
    \[ {F(x') + \frac{\lambda}{2}\norm*{x'- \bx}^2 - \min_{x \in \ball_r(\bx)\cap \xset}\crl*{F(x)+\frac{\lambda}{2}\norm*{x-\bx}^2} \le  \frac{\lambda}{2}\delta^2}. \]
    Procedure require $\log(1/p)$ BROO calls and the same number of exact function evaluation (to choose the best BROO output), resulting in the claimed complexity bound 
    $ \log\prn*{\frac{1}{p}}\brk*{ \broocost [\frac{\delta}{\sqrt{2}}]+  \ncost}.$
\end{proof}

To compute the gradient estimator in \cref{line:grad-est} of \Cref{alg:acceleratedProxPoint}, we use \Cref{alg:gradEst}.
Our implementation is slightly different than \cite{asi2021stochastic}
and in the following lemma we show it produces an estimator with the same bias and variance guarantees of \cite{asi2021stochastic}. 
\begin{lem}\label{lem:GradEstBiasVarianceBounds}
	Let $F:\xset\to\R$ satisfy \Cref{assumption:GlobalAssumption} and for query point $y\in\xset$ and regularization strength $\lambda>0$ define $x' = \argmin_{x\in\xset}\crl{ F(x) + \frac{\lambda}{2}\norm{x-y}^2}$ and $g = \lambda\prn*{y-x'}$. Then, for any bias and variance parameters $\beta,\sigma>0$ \Cref{alg:gradEst} outputs $\ghat =\lambda\prn*{y-\xhat}$ satisfying
    \begin{align*}
       \norm*{ \E \hat{g} - g} \le \beta \text{~~and~~} \E\norm*{\hat{g}-\E \hat{g}}^2\le \sigma^2.
    \end{align*}
\end{lem}
\begin{proof}
   First note that if $x=\oracle{\bx}$ is the output of an $r$-BROO with accuracy $\delta$ and if $x' = \argmin_{x \in \ball_r(\bx)}F(x)$, from \Cref{def:BROO} and strong convexity (of $F(x)+\frac{\lambda}{2}\norm*{x-\bx}^2$) we have: 
   \begin{align*}
       \frac{\lambda}{2}\E \norm*{x-x'}^2 \le \E\brk*{F(x)+\frac{\lambda}{2}\norm*{x-\bx}^2} - \brk*{F(x')+\frac{\lambda}{2}\norm*{x'-\bx}^2} \le \frac{\lambda \delta^2}{2}
   \end{align*}
   giving 
   \begin{align}\label{eq:BROOguarantee}
  \E \norm*{x-x'}^2\le   \delta^2.
\end{align}
   Let $j_{\max} = \floor{\log_2T_{\max}}$, then from the definition of $\xhat$ in \Cref{alg:gradEst} we have that 
   \begin{align*}
    \E \xhat =\E x_0 + \sum_{j=1}^{\jmax}\P(J=j)2^j\prn*{\E x_j -\E x_{j-1}}= \E x_{\jmax}.
    \end{align*}
   Therefore
   \begin{align*}
    \norm*{ \E \hat{g} - g} = \lambda   \norm*{ \E x_{\jmax}- x'} \stackrel{(\romannumeral1)}{\le}  \lambda \sqrt{\E \norm*{x_{\jmax} - x'}^2} \stackrel{(\romannumeral2)}{\le} \lambda \delta_{\jmax} = \lambda \frac{G \sqrt{\min\crl*{\beta^2 , \half \sigma^2}}}{\lambda \sqrt{2G^2T_0}}\le \min\crl*{\beta , \half \sigma}
   \end{align*}
   with $(\romannumeral1)$ following from Jensen inequality and $(\romannumeral2)$ following from the guarantee in \eqref{eq:BROOguarantee}.
   To bound the variance of  $\hat{g}$ note that 
   \begin{flalign*}
    \E\norm*{\hat{g}-\E \hat{g}}^2 = \lambda^2 \E\norm*{\xhat- \E \xhat}^2  \le  \lambda^2 \E \norm*{\xhat- x'}^2 \le  \lambda^2 \prn*{ 2\E \norm*{\xhat- x_0}^2+ 2\E \norm*{x_0 - x'}^2}
   \end{flalign*}
   where the last inequality follows from $\norm*{a+b}^2 \le 2\norm*{a}^2 +2\norm*{b}^2$.
   The definition of $\xhat$ gives 
   \begin{flalign*}
       \E \norm*{\xhat - x_0}^2 = \sum_{j=1}^{\jmax} 2^{j}\E\norm*{x_j - x_{j-1}}^2
   \end{flalign*} 
   and from the guarantee in \eqref{eq:BROOguarantee} we get  \[\E\norm*{x_j - x_{j-1}}^2 \le 2\E\norm*{x_j - x'}^2 + 2\E\norm*{x_{j-1}- x'}^2 \le 6\delta_j^2= \frac{6G^2}{\lambda^2 T_0 2^j}\]
   thus
   \begin{flalign*}
    \E \norm*{\xhat - x_0}^2 \le \jmax \frac{6G^2}{\lambda^2 T_0 }.
\end{flalign*} 
In addition we have $\E \norm*{x_0 - x'}^2 \le \delta^2_0 = \frac{G^2}{\lambda^2 T_0}$ and
substituting back we get  
\begin{flalign*}
    \E\norm*{\hat{g}-\E \hat{g}}^2= \lambda^2 \E\norm*{\xhat- \E \xhat}^2  \le  \lambda^2 \prn*{ 12 \jmax \frac{G^2}{\lambda^2 T_0 }+ 2 \frac{G^2}{\lambda^2 T_0}} \le 14  \jmax \frac{G^2}{ T_0 } \le \sigma^2.
   \end{flalign*}
\end{proof}
Combining the previous statements, we prove our main proposition. 
\MainProp* 
\begin{proof} We divide the proof into a correctness argument and a complexity calculation.
    \paragraph{Correctness.}
    To prove the correctness of \Cref{prop:MainProp} we first need to show that the guarantees of \cite[Proposition 2]{asi2021stochastic}
    still hold for our implementation that includes different parameters $\varphi_k$ and $\sigma_k^2$:
    \begin{align*}
        \varphi_k =  \frac{\lambda_k r^2}{900\log^{3}\prn*{\frac{GR^2}{\eps r}}} \text{~~and~~} \sigma_k^2 = \frac{\lambda_k^2 r^2}{900\log^{3}\prn*{\frac{GR^2}{\eps r}}}
    \end{align*}
     and different implementation of  \cref{line:prox-point,line:grad-est}.

    Following \Cref{lem:lemma6}, the guarantees of \cite[Proposition 2]{asi2021stochastic} are still valid
    with our different choice of $\varphi_k$ and $\sigma_k^2$. In addition, following \Cref{lem:GradEstBiasVarianceBounds}, our implementation of \cref{line:grad-est} is valid since it produces the same guarantees that the implementation in \cite{asi2021stochastic} gives. 
    Last, if the implementation of \cref{line:prox-point} is valid it needs to satisfy  
    \begin{flalign*}
        \E\brk*{F(x_{k+1}) +\frac{\lambda_{k+1}}{2}\norm{x_{k+1}-y_k}^2} - \min_{x \in \xset}\crl*{F(x) +\frac{\lambda_{k+1}}{2}\norm{x-y_k}^2} \le \varphi_{k+1}.
    \end{flalign*} 
    Note that $x_{k+1}$ in our implementation is the output of $r$-BROO with accuracy $\delta_k \le \frac{r}{\sqrt{14 \cdot 900}\log^{3/2}\prn*{\frac{GR^2}{\eps r}}}$, 
    therefore by \Cref{def:BROO} it satisfies 
    \begin{flalign*}
        \E\brk*{F(x_{k+1}) +\frac{\lambda_{k+1}}{2}\norm{x_{k+1}-y_k}^2} - \min_{x \in \ball_r(y_k)}\crl*{F(x) +\frac{\lambda_{k+1}}{2}\norm{x-y_k}^2}
        & \le \frac{\lambda_{k+1}\delta^2_{k+1}}{2} \\ & \le \frac{\lambda_{k+1}r^2}{900\log^{3}\prn*{\frac{GR^2}{\eps r}}}  = \varphi_{k+1}
    \end{flalign*}
    and for valid output of $\linesearch$ we have  
    \[ \min_{x \in \ball_r(y_k)}\crl*{F(x) +\frac{\lambda_{k+1}}{2}\norm{x-y_k}^2} = \min_{x \in \xset }\crl*{F(x) +\frac{\lambda_{k+1}}{2}\norm{x-y_k}^2}\]
    implying that \cref{line:prox-point} is valid. 
    Now let $p_\text{BROO}$ be the probability that all calls to $r$-BROO result in a valid output.
    Following  \cite[Proposition 2]{asi2021stochastic}, for $\kmax =  O\prn*{\prn*{\frac{R}{r}}^{2/3}\meps}$ with probability at least $1 - \prn*{1-\frac{2}{3}} - \prn*{1- p_\text{BROO}} = p_\text{BROO} - \frac{1}{3}$ 
    the algorithm outputs $x$ that satisfies  \[ F(x) - F(\xhat) \le \eps / 2 .\]
    Let $p$ be the probability that a single BROO implementation produce invalid output and let $K_\textup{bisect-max}$ be the maximal number of calls to high-probability $r$-BROO within \cref{line:next-lambda}.
    Then, $p_\text{BROO} \ge 1 - \kmax K_\textup{bisect-max}p$ and for
    \begin{align*}
        p \le \frac{1}{6 \kmax K_\textup{bisect-max}}
    \end{align*}
    with probability at least $\half$ \Cref{alg:acceleratedProxPoint} outputs $\frac{\eps}{2}$-suboptimal minimizer of $F$.

    \paragraph{Complexity.} To bound the complexity of \Cref{alg:acceleratedProxPoint} we first bound the complexity of  \cref{line:next-lambda}
    and the complexity of \cref{line:grad-est} in the $k$-th iteration of \Cref{alg:acceleratedProxPoint}.
    Note that, for $x_{k+1}$ in \cref{line:prox-point} we can use $x_0$ from \Cref{alg:gradEst}, and therefore the complexity of \cref{line:grad-est} already includes the complexity of \cref{line:prox-point}.

    Following \cite[Proposition 2]{carmon2021thinking}, $\linesearch$ in \cref{line:next-lambda} requires $ \meps = O\prn*{\log\prn*{\frac{GR^2}{\eps r}}}$ calls to 
    a high-probability $r$-BROO with accuracy $\delta = \frac{r}{30}$.
    From \Cref{lem:highProbBROO} the complexity of a single call to a probability $1-p$ $r$-BROO is $O\prn*{ \log\prn*{\frac{1}{p}}\brk*{ \broocost [r]+  \ncost }}$.
    We set $p = \frac{1}{6 \kmax K_\textup{bisect-max}}$, and since $\kmax =  O\prn*{\prn*{\frac{R}{r}}^{2/3}\meps}$ and $K_{\textup{bisect-max}} =\meps$ we get $\log\prn*{\frac{1}{p}} = O\prn*{ \log \prn*{\frac{R}{r}\meps^2}} \le \meps$ . 
    Therefore, the total complexity of $\linesearch$ is 
    $ O\prn*{ \meps^2 \brk*{ \broocost[r]+  \ncost }}$.
    
    For the complexity of \Cref{line:grad-est} note that \Cref{alg:gradEst} calls to an $r$-BROO with accuracy $\delta_J \ge \frac{r}{30\sqrt{14}2^{J/2}\meps^2}$ 
    where $J \sim \text{Geom}\prn*{\half, \jmax}$ and $\jmax \le \meps$. 
    Therefore, we can bound the complexity of \Cref{alg:gradEst} by
    \[O \prn*{\sum_{j=0}^{\meps}\frac{1}{2^j}\broocost[\frac{r}{2^{j/2} \meps^2}][\lmin]} .\]

    The complexity of the entire algorithm is at most $\kmax$ times the complexity of a single iteration. 
    Using $\kmax =  O\prn*{\prn*{\frac{R}{r}}^{2/3}\meps}$, the total complexity becomes 
    \[ O\prn*{\prn*{\frac{R}{r}}^{2/3}\prn*{ 
        \meps \sum_{j=0}^{\meps}\frac{1}{2^j}\broocost[\frac{r}{2^{j/2} \meps^2}][\lmin] +
   \meps^3 \brk*{ \broocost[r]+  \ncost }}}. \] 
\end{proof} %
\section{Group DRO}\label{sec:groupDROProofs}
In this section we provide the proofs for the results of \Cref{sec:groupDRO}. 
In \Cref{ssec:ExpSoftMax-proof} we first prove that the 
group-softmax is a uniform approximation of $\groupObjective$, then,
through extension of  \cite{carmon2021thinking}, we show that we can approximate $\groupObjective$ using the
group-exponentiated softmax instead. 
Next, in \Cref{ssec:MLMCbounds-proof} we prove \Cref{lem:MLMCbounds} and bound the moments of the MLMC  and gradient estimators. 
We then prove the complexity guarantees of \Cref{thm:BrooComplexGoupDRO} in \Cref{ssec:GroupDROEpochSGD-proof}.
Last, in \Cref{ssec:GroupDROSVRGproperties-proof} under the mean-square smoothness assumption, we provide the properties of the gradient estimator in \eqref{eq:SVRG} and in \Cref{ssec:GroupDROSVRG-proof} 
we prove the complexity guarantees of \Cref{thm:svrgBrooComplex}. 
\subsection{Exponentiated group-softmax}\label{ssec:ExpSoftMax-proof}
Recall the definition of the (regularized) group-softmax 
\[
    \LsmReg(x) \defeq \epsilon' \log \prn*{\sum_{i \in [\ngroups]}e^{\frac{\mc{L}_i(x)}{ \eps'} }}+\frac{\lambda}{2}\norm{x-\bx}^2 \text{~~where~~} \mc{L}_i(x) = \sum_{j \in [N]}w_{ij}\ell_{j}(x)
\]
with  $\LsmReg[\eps, 0](x)= \Lsm(x)$.
In addition, recall the definition of the (regularized) group-exponentiated softmax 
\[        \Gamma_{\epsilon, \lambda}(x) \defeq 
         \sum_{i \in [\ngroups]}\bp_i \gamma_i(x)
         \text{~~where~~} \gamma_i(x) = \epsilon' e^{\frac{\mc{L}_i(x) - \mc{L}_i(\bx)+\frac{\lambda}{2}\norm{x-\bx}^2}{\epsilon'}}
         \text{~~and~~}
         \bp_i = \frac{e^{\frac{\mc{L}_i(\bx)}{\epsilon'}}}{\sum_{i \in [\ngroups]}e^{\frac{\mc{L}_i(\bx)}{\epsilon'} }}.\]
\begin{lem}\label{lem:GroupSMApproxGroupObj}
    Let $\Lsm$ be the group-softmax defined in \cref{eq:SoftMax} and $\groupObjective$ be the Group DRO objective defined in \eqref{eq:mainProblemGroupDRO}. 
    Let $\epsilon >0$ and  $\epsilon' = \epsilon / (2 \log M) >0$.
    Then for all $x\in\xset$ we have that 
    \[  \abs*{ \groupObjective(x)  - \Lsm(x)} \le \epsilon/2  \]
\end{lem}
\begin{proof}
    First note that \[ \groupObjective(x) \defeq \max_{i \in [\ngroups]}  \sum_{j=1}^N w_{ij} \ell_j(x) = \max_{q \in \Delta^\ngroups} \sum_{i \in [\ngroups]} q_i\mc{L}_i(x) \] 
    and \[ \Lsm(x) = \epsilon' \log \prn*{\sum_{i \in [\ngroups]}e^{\frac{\mc{L}_i(x)}{ \eps'} }} = \max_{q \in \Delta^{\ngroups}} \crl*{\sum_{i \in [\ngroups]} q_i \mc{L}_i(x) - \eps'q_i \log q_i }.\]
    In addition, for $q \in \Delta^{\ngroups}$ we have that $\sum_{i \in [\ngroups]}q_i \log qi \in [- \log \ngroups, 0]$.
    Combining these facts gives   
    \begin{align*}
        \abs*{  \Lsm(x)- \groupObjective(x) } &= 
          \abs*{ \max_{q \in \Delta^{\ngroups}} \crl*{\sum_{i \in [\ngroups]} q_i \mc{L}_i(x) - \eps'q_i \log q_i } -  \max_{q \in \Delta^\ngroups}\sum_{i \in [\ngroups]} q_i\mc{L}_i(x)  } 
          \\ & \le \abs*{\eps'\sum_{i \in [\ngroups]}q_i \log q_i } \le \eps' \log \ngroups = \eps / 2 
    \end{align*}
\end{proof}
\SmGammaProperties*
\begin{proof}
     This lemma is a simple extension of \cite[][Lemma 1]{carmon2021thinking},  
     that considers the exponentiated-softmax: 
      \[ \sum_{i \in \ngroups} \eps' \frac{e^{l_i(\bx)/\eps'}}{\sum_{j \in \ngroups}e^{f_j(\bx)/\eps'}} e^{\frac{l_i(x)-l_i(\bx)+\lambda \norm*{x-\bx}}{\eps'}},\] 
     for some $l_1,\ldots,l_N$. 
    The only assumption that \cite{carmon2021thinking} have on $l_i$ (for the guarantees we state in \Cref{lem:SmGammaProperties}) is that each  $l_i$ is $G$-Lipschitz. 
    Note that each $\mc{L}_i$ is $G$-Lipschitz since it is  a weighted average of $G$-Lipschitz functions.
    Therefore, we can replace $l_i$ in \cite[][Lemma 1]{carmon2021thinking} with the group average $\mc{L}_i$ and obtain \Cref{lem:SmGammaProperties}.
\end{proof}
\subsection{MLMC estimator moment bounds}\label{ssec:MLMCbounds-proof}
To make the MLMC estimator suitable for both Epoch-SGD and variance reduction methods, we rewrite its definition using more general notation. Specifically, for every $x,x'\in\xset$ and $S_1^n \in [N]^n$, let 
\begin{equation}\label{eq:gammahatDef}
    \gammahat(x,x'; S_1^n) \defeq  \eps' e^{\frac{1}{n}\sum_{j=1}^n \frac{\ell_{S_j}(x)- \ell_{S_j}(x')+\frac{\lambda}{2}\norm{x-\bx}^2}{\epsilon'}}
\end{equation}
so that $\gammahat(x,\bx; S_1^n) = \gammahat(x;S_1^n)$. The MLMC estimator is 
\[ 
\text{Draw~} J\sim \mathrm{Geom}\prn*{1-\tfrac{1}{\sqrt{8}}} \text{~,~} S_1, \ldots, S_n \overset{\textup{iid}}{\sim} w_i  \text{~and let~}  \gammamlmc \defeq   \gammahat(x,\bx; S_1) + \frac{\dhat_{2^J}}{p_J},
 \] 
where $p_j \defeq \P(J=j) = \prn*{1/\sqrt{8}}^j \prn*{1-\frac{1}{\sqrt{8}}}$ and, for $n \in 2\mathbb{N}$  we define 
\begin{equation}\label{eq:DhatDef}
    \dhat_n \defeq \gammahat(x,x'; S_1^n) -  \frac{ \gammahat\prn*{x,x'; S_1^{\frac{n}{2}}} + \gammahat\prn*{x,x'; S_{\frac{n}{2} + 1}^n}}{2}.
\end{equation}
\begin{lem}\label{lem:DhatBound}
	Let each $\ell_i$ satisfy \Cref{assumption:GlobalAssumption}, and let $r \le \frac{\eps'}{G}$, $\lambda \le \frac{G}{r}$, $\norm{x-x'}\le 2r$ and $\norm{x-\bx}\le r$. 
    For $\dhat_n$ defined in \eqref{eq:DhatDef}
    we have  $  \E\abs*{\dhat_n}^2 \le O\prn*{\frac{G^4\norm*{x-x'}^4}{n^2 \eps'^2}}. $
\end{lem}
\begin{proof}
    For abbreviation let $M = \frac{1}{n}\sum_{j \in [n]}\frac{ \hat{\ell}_{S_j}(x)+\frac{\lambda}{2}\norm*{x-\bx}^2}{\eps'}$ and $\delta =  \frac{1}{n}\sum_{j \in [n/2]}\prn*{\frac{\hat{\ell}_{S_j}(x) -\hat{\ell}_{S_{j+n/2}}(x)}{\eps'}}$
    where $\hat{\ell}_{S_j}(x) =\ell_{S_j}(x)- \ell_{S_j}(x')$.
    We have the following bound on $\abs*{M}$: 
    \begin{align}\label{eq:Mbound}
        \abs{M} \stackrel{(\romannumeral1)}{\le}\frac{1}{n}\sum_{j \in [n]}\frac{G\norm{x-x'}+\frac{G}{2r}\norm*{x-\bx}^2 }{\eps'}\stackrel{(\romannumeral2)}{\le} \frac{2Gr+Gr / 2 }{\eps'} \stackrel{(\romannumeral3)}{\le} 2.5
    \end{align}
    with $(\romannumeral1)$ following since each $\ell_j$ is $G$-Lipschitz and $\lambda \le G/r$, $(\romannumeral2)$ follows since $\norm*{x-x'}\le 2r$ and $\norm*{x-\bx}\le r$, and $(\romannumeral3)$ since $r \le \eps'/ G$. For $\abs*{\delta}$ we have the bound, 
    \begin{align}\label{eq:deltaBound}
        \abs*{\delta} \le \abs*{\frac{1}{n}\sum_{j \in [n/2]}\frac{\hat{\ell}_{S_j}(x)}{\eps'}} +  \abs*{\frac{1}{n}\sum_{j \in [n/2]}\frac{\hat{\ell}_{S_{j+n/2}}(x)}{\eps'}}
        \le \frac{G\norm*{x-x'}}{\eps'} \le 2,
    \end{align}
    with the second inequality following since each $\ell_j$ is $G$-Lipschitz and the last inequality since $\norm*{x-x'} \le 2r$ and $r \le \eps'/G$.
    We bound $\abs*{\dhat_n}$ using the previous guarantees on $\abs*{M}$ and $\abs*{\delta}$:
    \begin{flalign*}
        \abs*{\dhat_n} &= \epsilon' \abs*{e^M -\frac{e^{M+\delta}+e^{M-\delta}}{2} }
        \le \epsilon'  e^M \prn*{\frac{e^\delta+e^{-\delta}}{2} - 1}
         \stackrel{(\romannumeral1)}{\le} \epsilon'  e^{2.5} \prn*{\frac{e^\delta+e^{-\delta}}{2} - 1}
         \stackrel{(\romannumeral2)}{\le}   2e^{2.5} \epsilon'\delta ^2
    \end{flalign*}
    where 
    $(\romannumeral1)$ follows from \eqref{eq:Mbound} and 
    $(\romannumeral2)$ from \eqref{eq:deltaBound} and the inequality $e^x \le 1+x+2x^2$ for all $x \le 3$ with $x=\delta$.
    Therefore, we have that 
    \begin{flalign*}
        \E\abs*{\dhat_n}^2 & \le 4e^5 \epsilon'^2  \E \brk*{\delta^4}.
    \end{flalign*}
    Let $Y_i = \frac{\hat{\ell}_{S_i}(x) - \hat{\ell}_{S_{i+n/2}}(x)}{n\eps'}$ and
    note that the Lipschitz property of each $\ell_j$ gives $\abs*{ Y_i} \le  2\frac{G\norm*{x-x'}}{n\eps'}$, 
    in addition, since the samples  $S_1, \ldots, S_n$ are i.i.d we have  $\E\frac{\hat{\ell}_{S_i}(x)}{n\eps'}=\E\frac{\hat{\ell}_{S_{i+n/2}}(x)}{n\eps'}$
    and therefore $\E Y_i =0$. 
    Thus, we can use \Cref{lem:ExpectationOfBoundedCenteredVariables} below, with $Y_i = \frac{\hat{\ell}_{S_i}(x) - \hat{\ell}_{S_{i+n/2}}(x)}{n\eps'}$ and $c = 2\frac{G\norm*{x-x'}}{n\eps'}$,
    to obtain the bound $\E\brk*{\delta}^4 \le O\prn*{\frac{G^4\norm*{x-x'}^4}{n^2 \eps'^4}}$. Therefore,
    \begin{equation*}\label{eq:MLMCbiasBound}
         \E\abs*{\dhat_n}^2 \le O\prn*{\frac{G^4\norm*{x-x'}^4}{n^2 \eps'^2}}.
    \end{equation*}
\end{proof}
\MLMCbounds*
\begin{proof}
	We first prove the bias and moment bounds, and then address complexity.
	
    \paragraph{Properties of the MLMC estimator.}
	We first show that the MLMC estimator is unbiased. For every $n \in 2\N$ we have that $\E\gammahat(x; S_1^{n/2}) = \E\gammahat(x; S_{n/2+1}^{n})$, therefore $\E \dhat_n = \E\gammahat(x; S_1^{n}) - \E\gammahat(x; S_1^{n/2})$ and we get 
    \begin{align}\label{eq:MLMCunbiased}
        \E \brk*{\gammamlmc} & = \E \gammahat(x; S_1) + \sum_{j=1}^{\infty}\prn*{ \E \gammahat(x; S_1^{2^j}) - \E  \gammahat(x; S_1^{2^{j-1}})}  = \E \gammahat(x; S_1^\infty) = \gamma_i(x).
    \end{align}
    To bound the second moment of the estimator we 
    use the inequality $(a+b)^2 \le 2a^2 +2b^2$, yielding
    \begin{equation}\label{eq:mlmcbound}
        \E\abs*{\gammamlmc[x]}^2 \le 
        2 \E \abs*{\gammahat(x; S_1)}^2  + 2\sum_{j=1}^{\infty} \frac{1}{p_j}\E\abs*{\dhat_{2^j}}^2.
    \end{equation}
    \Cref{lem:DhatBound} with $x'=\bx$ gives  $  \E\abs*{\dhat_n}^2 \le O\prn*{\frac{G^4\norm*{x-\bx}^4}{n^2 \eps'^2}} $
    and substituting this bound into \eqref{eq:mlmcbound} while noting that 
    $ \E \abs*{\gammahat(x; S_1)}^2  \le \eps'^2e^3$ 
    gives
    \begin{equation}\label{eq:mlmcSecondMomentBound}
        \E\abs*{\gammamlmc}^2 \le  O \prn*{ \eps'^2  + \frac{G^4\norm*{x-\bx}^4}{\epsilon'^2}   \sum_{j=1}^\infty \prn*{1-\frac{1}{\sqrt{8}}} \frac{2^{1.5j}}{2^{2j}}} 
        = O\prn*{\frac{G^4\norm*{x-\bx}^4}{\epsilon'^2} + \eps'^2}.
    \end{equation}
    \paragraph{Properties of the gradient estimator.} 
    We use the fact that  $\gammamlmc$ is unbiased for $\gamma_i(x)$ (shown in eq.~\eqref{eq:MLMCunbiased} above) to argue that gradient estimator is also unbiased:
    \begin{align*}
        \E\brk*{\ghat(x)} &= \E\brk*{\frac{1}{\epsilon'}\gammamlmc \prn*{\nabla \ell_{j}(x)+\lambda\prn*{x-\bx} }}=\frac{1}{\eps'} \sum_{i \in [M]} \sum_{j \in [N]} \bp_i w_{ij} \E \gammamlmc \prn*{\nabla \ell_j(x) + \lambda\prn*{x-\bx}}
       \\ & = \frac{1}{\eps'}\sum_{i \in [M]} \bp_i \gamma_i(x) \prn*{\nabla \mc{L}_i(x) +\lambda\prn*{x-\bx}} = \sum_{i \in [M]} \bp_i \nabla \gamma_i(x) = \nabla \Gamma_{\eps, \lambda}(x).
    \end{align*}
    Next we bound the second moment of the gradient estimator 
    \begin{flalign*}
        \E \norm*{\ghat(x)}^2 &=   \frac{1}{\eps'^2} \sum_{i \in [\ngroups]}\bp_i   \E \prn*{\gammamlmc }^2  \sum_{j \in [N]}w_{ij}\norm*{\nabla \ell_{j}(x)+\lambda \prn*{x-\bx}}^2 
        \\ & \stackrel{(\romannumeral1)}{\le}   O\prn*{ \prn*{\frac{G^4 \norm*{x-\bx}^4}{\eps'^4}+1}  \sum_{i \in [\ngroups]}\bp_i \sum_{j \in [N]}w_{ij}\norm*{\nabla \ell_{j}(x)+\lambda \prn*{x-\bx}}^2 } 
        \\ &  \stackrel{(\romannumeral2)}{\le}   O\prn*{ \prn*{\frac{G^4 \norm*{x-\bx}^4}{\eps'^4}+1} G^2} 
        \stackrel{(\romannumeral3)}{\le}  O\prn*{G^2}
    \end{flalign*}
where $(\romannumeral1)$ follows from \eqref{eq:mlmcSecondMomentBound}, $(\romannumeral2)$ follows since each $\ell_j$ is $G$-Lipschitz, $\lambda \le \frac{G}{r}$ and $\norm*{x-\bx}\le r$ and  $(\romannumeral3)$ since $\frac{G^4 \norm*{x-\bx}^4}{\eps'^4}\le \frac{G^4 r^4}{\eps'^4} \le 1$ . 
\paragraph{Complexity of the MLMC and gradient estimators.}
$J \sim \text{Geom}(1-\frac{1}{\sqrt{8}})$, therefore 
\begin{equation*}
    \E\brk*{2^J} = \sum_{j=1}^{\infty} \frac{1}{1-\frac{1}{\sqrt{8}}}\prn*{\frac{1}{\sqrt{8}}}^{j} 2^j = O(1).
\end{equation*}
Note that the estimator $\gammamlmc  = \gammahat(x,S_1) + \frac{1}{P_J}\dhat_{2^J}$ requires a single function evaluation for $\gammahat(x,S_1)$
and $2^J$ function evaluations for the term $\dhat_{2^J}$. As a consequence the computation of $\gammamlmc$ requires
only $O(1)$ function evaluations in expectation. To compute $\ghat(x)$ we need to compute $\gammamlmc$ and a single sub-gradient, hence, the complexity of computing 
$\ghat(x)$ is also $O(1)$ in expectation. 
\end{proof}
\subsection{Epoch-SGD BROO implementation}\label{ssec:GroupDROEpochSGD-proof}
We state below the convergence rate of the Epoch-SGD algorithm. 
\begin{lem}[{Theorem 5, \cite{hazan2014beyond}}]\label{lem:EpochSgd}
    Let $F :\xset \to \R$ be $\lambda$-strongly convex with an unbiased stochastic gradient estimator $\ghat$ satisfying $\E\norm*{\ghat(x)}^2 \le O\prn*{G^2}$ for all $x \in \xset$, and let 
     $x_\star = \argmin_{x \in \xset} F(x)$. Epoch-SGD  finds an approximate minimizer $x$ that satisfies
    \begin{equation*}
        \E F(x) - F(x_\star) \le O\prn*{\frac{G^2}{\lambda T}} 
    \end{equation*}
    using $T$ stochastic gradient queries.
\end{lem}
Applying this lemma with $F=\Gamma_{\eps, \lambda}$, $x \in \ball_{\reps}(\bx)$ and $T=O\prn*{\frac{G^2}{\lambda^2 \delta^2}}$ immediately gives the following guarantee on the BROO implementation complexity. 
\BrooComplexGoupDRO*
\begin{proof}
	We divide the proof into correctness and complexity arguments, addressing the BROO implementation and then the overall algorithm.
	
    \paragraph{BROO implementation: correctness.}
    Following \Cref{lem:SmGammaProperties} 
    we have that $\Gamma_{\eps,\lambda}$ is $\Omega(\lambda)$-strongly convex and \Cref{lem:MLMCbounds} gives $ \E \norm*{\ghat(x)}^2 \le O\prn*{G^2}$.
    Thus, we can directly apply \Cref{lem:EpochSgd} with $F=\Gamma_{\eps,\lambda}$, $\xset = \ball_{\reps}(\bx)$, the gradient estimator $\ghat(x)$ defined in  \eqref{eq:mlmcGradEst} and  $T=\frac{2c^2G^2}{\lambda^2 \delta^2}$ for a constant $c >0$
    for which Epoch-SGD outputs $x$ that satisfies
    \begin{align}\label{eq:EpochSGDbound}
       \E \Gamma_{\eps,\lambda}(x) - \Gamma_{\eps,\lambda}(x_\star) \le \frac{cG^2}{\lambda T} \le  \frac{\lambda \delta^2}{2c}.
    \end{align}
    Following \Cref{lem:SmGammaProperties} there is a value of $c$ such that  
    $ \E \Lsm^\lambda(x) -\Lsm^\lambda(x_\star) \le c \prn*{\E \Gamma_{\epsilon,\lambda}(x) - \Gamma_{\epsilon,\lambda}(x_\star)} $
    and from \eqref{eq:EpochSGDbound} we obtain 
    \begin{align*} \E \Lsm^\lambda(x) -\Lsm^\lambda(x_\star) \le c \prn*{\E \Gamma_{\epsilon,\lambda}(x) - \Gamma_{\epsilon,\lambda}(x_\star)}  \le \frac{\lambda \delta^2}{2}.
    \end{align*}
    Therefore, Epoch-SGD outputs a valid $\reps$-BROO response for $\Lsm$.
    \paragraph{BROO implementation: complexity.}
    For the BROO implementation we run Epoch-SGD with the gradient estimator $\ghat(x)$ defined in  \eqref{eq:mlmcGradEst} and computation budget $T=O\prn*{\frac{G^2}{\lambda^2 \delta^2}}$.
    Therefore, we need to evaluate  $O\prn*{\frac{G^2}{\lambda^2 \delta^2}}$ stochastic gradient estimators with complexity $O(1)$, and 
    our gradient estimator requires additional $N$ functions evaluations for precomputing the sampling probabilities $\crl*{\bp_i}$. Thus, the total complexity of the BROO implementation  is
    \begin{align}\label{eq:GroupDROBROOcomplex}
        O\prn*{\frac{G^2}{\lambda^2 \delta^2}+N}.
    \end{align}
    \paragraph{Minimizing $\groupObjective$: correctness.}
    For any $q \in \Delta^{\ngroups}$ note that $\mc{L}_q(x) \defeq{\sum_{i \in[\ngroups]} q_i \mc{L}_i(x)-\eps' q_i\log q_i}$ is $G$-Lipschitz,
    since $\mc{L}_i$ is $G$-Lipschitz for all $i \in [\ngroups]$ and therefore for all $x \in \xset$ we have $\norm{\nabla \mc{L}_q(x)} = \norm*{\sum_{i \in [\ngroups]} q_i \nabla \mc{L}_i(x)} \le G$. Maximum operations preserve the Lipschitz continuity 
    and therefore $\Lsm(x) = \max_{q \in \Delta^{\ngroups}} \mc{L}_q(x)$ is also $G$-Lipschitz. Thus, we can use
    \Cref{prop:MainProp} with $F=\Lsm$ and obtain that the output $\bx$ of \Cref{alg:acceleratedProxPoint} with probability at least $\half$ will satisfy 
   $ \Lsm(\bx) - \min_{x_\star \in \xset}\Lsm(x_\star) \le \eps / 2 $. In addition, from  \Cref{lem:GroupSMApproxGroupObj} for every $x \in \xset$ we have that
   $\abs*{\groupObjective(x) -\Lsm(x)} \le \eps /2$.
    Therefore, with probability at least $\half$ 
    \begin{align*} 
        \groupObjective(\bx)  - \min_{x_\star \in \xset}\groupObjective(x_\star) \le \Lsm(\bx) - \min_{x_\star \in \xset}\Lsm(x_\star) + \eps / 2 \le \eps.
    \end{align*} 
    \paragraph{Minimizing $\groupObjective$: complexity.}
    The complexity of finding an $\eps/2$-suboptimal solution for $\Lsm$ (and therefore an $\eps$-suboptimal solution for $\groupObjective$) is bounded by \Cref{prop:MainProp} as:
    \begin{align*}
        O\prn*{
    		\prn*{\frac{R}{\reps}}^{2/3} \brk*{
                \prn*{\sum_{j=0}^{\meps}\frac{1}{2^j}\broocost[\frac{\reps}{ 2^{j/2}\meps^2}][\lmin]}\meps 
    			+
    			\prn*{ \broocost[\reps][\lmin] + N }\meps^3
    		}
    	}
    \end{align*}
where $\meps=O\prn*{\log\prn*{\frac{GR^2}{\eps\reps}}} = O\prn*{\log\prn*{\frac{GR}{\eps}\log \ngroups}}$.
    To obtain the total complexity we evaluate the complexity of running $\reps$-BROO 
    with accuracy $\delta_{j} =  \frac{r}{ 2^{j/2}\meps^2}$ (for the MLMC implementation),
    and accuracy $\delta_\textup{Bisection} = \frac{\reps}{30}$ (for the bisection procedure). 
    Using \eqref{eq:GroupDROBROOcomplex} and noting that $\lmin = \frac{\eps}{\reps^{4/3}R^{2/3} }\meps^2$ we get the following BROO complexities:
    \begin{enumerate}
        \item {$\broocost[\frac{\reps}{\meps^2 2^{j/2}}][\lmin] = O\prn*{\frac{G^2 2^j \meps^4}{\lmin^2 \reps^2  }+N} = O\prn*{\frac{\prn*{\frac{GR}{\eps}}^{4/3}}{\prn*{\log \ngroups}^{2/3}}2^j+N}$}
        \item {$\broocost[\frac{\reps}{30}][\lmin]=O\prn*{\frac{G^2}{\lmin^2\reps^2}+N} = O\prn*{\frac{\prn*{\frac{GR}{\eps }}^{4/3}}{\meps^4 \prn*{\log \ngroups}^{2/3} }+N}.$}
    \end{enumerate}
    Therefore 
    \begin{align*}
        O\prn*{
            \meps \sum_{j=0}^{\meps}\frac{1}{2^j}\broocost[\frac{\reps}{ 2^{j/2}\meps^2}][\lmin]
    		} =
             O\prn*{
                \meps \sum_{j=0}^{\meps}\frac{1}{2^j}\prn*{\frac{\prn*{\frac{GR}{\eps}}^{4/3}}{\prn*{\log \ngroups}^{2/3}}2^j+N} }
                \le O\prn*{\meps^2 \prn*{\prn*{\frac{GR}{\eps}}^{4/3} + N}}
    \end{align*}
    and
    \begin{align*}
        O  \prn*{ \meps^3\prn*{\broocost[\reps/30][\lambda_k]+N }  }
        \le O\prn*{ \prn*{\frac{GR}{\eps } }^{4/3}+N\meps^3}.
    \end{align*}
    Substituting the bounds into \Cref{prop:MainProp} with $\meps = \log \prn*{\frac{GR}{\eps}\ngroups}$ and $\reps=\frac{\eps}{2G \log \ngroups}$, the total complexity is
    \begin{align*}
       O\prn*{
    		\prn*{\frac{R}{\reps}}^{2/3} \brk*{N\meps^3 +\meps^2 \prn*{\frac{GR}{\eps}}^{4/3}   } 
            }
      \le O\prn*{
    		\prn*{\frac{GR}{\eps}}^{2/3} N \log^{11/3} \prn*{\frac{GR}{\eps}\ngroups} + \prn*{\frac{GR}{\eps}}^{2} \log^{2} \prn*{\frac{GR}{\eps}\ngroups} } .
    \end{align*}
\end{proof}
\subsection{SVRG-like estimator properties}\label{ssec:GroupDROSVRGproperties-proof}
We first give a definition of $\Gamma_{\eps, \lambda}$ that is more conducive to formulating variance reduction methods:
\begin{align*}
   & \Gamma_{\eps, \lambda}(x) \defeq  \sum_{i \in [\ngroups]} \mlmcConstant p_i(x')\gamma_i(x,x'),
\end{align*}
where  $\gamma_i(x,x') \defeq \eps' e^{\frac{\mc{L}_i(x)-\mc{L}_i(x')+\frac{\lambda}{2}\norm*{x-\bx}^2}{\eps'}} $, $\mlmcConstant \defeq \prn[\Bigg]{\frac{\sum_{j \in [\ngroups]}e^{\frac{\mc{L}_j(x')}{\eps'}}}{\sum_{j \in [\ngroups]}e^{\frac{\mc{L}_j(\bx)}{\eps'}}}}$
and $p_i(x') \defeq \frac{e^{\frac{\mc{L}_i(x')}{\eps'}}}{\sum_{j \in [\ngroups]}e^{\frac{\mc{L}_j(x')}{\eps'}}}.$
Therefore, the MLMC estimator of $\gamma_i(x,x')$ is 
\[ 
\text{Draw~} J\sim \mathrm{Geom}\prn*{1-\tfrac{1}{\sqrt{8}}} \text{~,~} S_1, \ldots, S_n \overset{\textup{iid}}{\sim} w_i  \text{~and let~}  \gammamlmc[x,x'] \defeq   \gammahat(x,x'; S_1) + \frac{\dhat_{2^J}}{p_J}
 \] 
 with $\dhat_n$ defined in \eqref{eq:DhatDef} and $\gammahat(x,x';S_1^n)$ defined in \eqref{eq:gammahatDef}. 

\begin{lem}\label{lem:unbiasedSVRG}
    The SVRG-like estimator \eqref{eq:SVRG} is unbiased 
    \begin{align*}
        \E \brk*{\ghat_{x'}(x)} = \nabla \Gamma_{\eps, \lambda}(x)
    \end{align*}
\end{lem}
\begin{proof}   
    \begin{flalign*}
        \E \brk*{\ghat_{\bx}(x)}  &= 
        \nabla \Gamma_{\eps, \lambda}(x') + \sum_{i \in [\ngroups]} \sum_{j \in [N]} p_i(x') w_{ij}  \frac{1}{\epsilon'}\brk*{\E\prn*{\gammamlmc[x,x']} \nabla \ell^\lambda_{j}(x) - \gamma_i(x',x') \nabla \ell^\lambda_{j}(x')}
        \\ & \stackrel{(\romannumeral1)}{=} \nabla \Gamma_{\eps, \lambda}(x') + \sum_{i \in [\ngroups]} \sum_{j \in [N]} p_i(x') w_{ij}  \frac{1}{\epsilon'}\brk*{ \gamma_i(x,x') \nabla \ell^\lambda_{j}(x) - \gamma_i(x',x') \nabla \ell^\lambda_{j}(x')}
        \\ & =  \nabla \Gamma_{\eps, \lambda}(x') + \sum_{i \in [\ngroups]} p_i(x') \brk*{\nabla \gamma_i(x,x') - \nabla \gamma_i(x',x')}
        \\ & = \nabla \Gamma_{\eps, \lambda}(x)
    \end{flalign*}
    with $(\romannumeral1)$ following from the unbiased property of the MLMC estimator stated in \Cref{lem:MLMCbounds} (that still holds for $\gammamlmc[x,x']$).
\end{proof}
\SVRGvarianceBound*
\begin{proof}
    \begin{flalign}\label{eq:SVRGvarianceBound}
        \nonumber
        \Var\prn*{\ghat_{x'}(x)} &= \E \norm*{\ghat_{x'}(x) - \E \ghat_{x'}(x)}^2 
        \nonumber
        \\ &= \E \norm*{\nabla \Gamma_{\eps, \lambda}(x') 
        + \frac{\mlmcConstant}{\epsilon'}
        \brk*{\gammamlmc[x,x'] \nabla \ell^\lambda_{j}(x) - \gamma_i(x',x') \nabla \ell^\lambda_{j}(x')}
        - \nabla \Gamma_{\eps, \lambda}(x)}^2 
        \nonumber
        \\ & \stackrel{(\romannumeral1)}{\le} \frac{\mlmcConstant^2}{\epsilon'^2} \E \norm*{\gammamlmc[x,x'] \nabla \ell^\lambda_{j}(x) - \gamma_i(x',x') \nabla \ell^\lambda_{j}(x')}^2 
        \nonumber
        \\ & = \frac{\mlmcConstant^2}{\epsilon'^2} \E \norm*{\prn*{\gammamlmc[x,x'] - \gamma_i(x',x')} \nabla \ell^\lambda_{j}(x) + \gamma_i(x',x') \prn*{ \nabla \ell^\lambda_{j}(x) - \nabla \ell^\lambda_{j}(x')  }}^2 
        \nonumber
        \\ & \stackrel{(\romannumeral2)}{\le} \frac{\mlmcConstant^2}{\epsilon'^2} 
        \brk*{2\E \norm*{\prn*{\gammamlmc[x,x'] - \gamma_i(x',x')} \nabla \ell^\lambda_{j}(x)}^2 
             + 2\E\norm*{\gamma_i(x',x') \prn*{ \nabla \ell^\lambda_{j}(x) - \nabla \ell^\lambda_{j}(x')  }}^2 }
    \end{flalign}
where$(\romannumeral1)$ follows from the inequality $\E \brk*{X -\E X}^2 = \E[X^2] - \brk*{\E X}^2 \le \E[X^2]$
and $(\romannumeral2)$ from the inequality $(a+b)^2 \le 2a^2 + 2b^2$. 
Next we bound separately each of the expectation terms.
Note that the ball constraint $x \in \ball_{r}(\bx)$ with $r=\frac{\eps'}{G}$ and $\lambda \le \frac{G}{r}$ gives:
   \[ \gamma_i(x',x') = \eps' e^{\frac{\lambda}{2\eps'}\norm*{x-\bx}^2  }
   \le e^{\frac{Gr}{2\eps'}} = O(\eps')
\]
therefore
\begin{flalign*}
    \E \norm*{\gamma_i(x',x') \brk*{\nabla \ell_{j}(x) - \nabla \ell_{j}(x')+ \lambda \prn*{x-x'}}}^2 
   & \stackrel{(\romannumeral1)}{\le} O\prn*{\eps'^2}\prn*{ 2\E \norm*{ \nabla \ell_{j}(x)  - \nabla \ell_{j}(x')}^2 + 2  \norm*{ \lambda\prn*{x-x'}}^2}
   \\& \stackrel{(\romannumeral2)}{\le} O\prn*{ \epsilon'^2 \prn*{\lambda^2 + L^2}\norm*{x-x'}^2}
\end{flalign*}
with $(\romannumeral1)$ following from the inequality $(a+b)^2 \le 2a^2 + 2b^2$  
and $(\romannumeral2)$ from \Cref{assumption:smooth}.
For the second expectation term we use the fact that each $\ell_j$ is $G$-Lipschitz, $\lambda \le \frac{G}{r}$ and $\norm*{x-\bx}\le r$ and thus
$\norm*{\nabla \ell_{j}(x)+ \lambda\prn*{x-\bx}} \le \norm*{\nabla \ell_{j}(x)} + \norm*{\lambda\prn*{x-\bx}} \le 2G$. Therefore, 
\begin{flalign*}
    \E \norm*{\prn*{ \gammamlmc[x,x'] - \gamma_i(x',x')} \prn*{\nabla \ell_{j}(x)+ \lambda\prn*{x-\bx}} }^2  
       & \le O\prn*{ G^2\prn*{ \E \abs*{ \gammamlmc[x,x'] - \gamma_i(x',x')}^2  }}.
\end{flalign*}
From the definition of  $\gammamlmc[x,x']$ we get: 
\begin{flalign*}
    \E \abs*{ \gammamlmc[x,x'] - \gamma_i(x',x')}^2 & \le 2\E \abs*{ \gammahat(x,x'; S_1) -\gamma_i(x',x')}^2  + 2\sum_{j=1}^{\infty}\prn*{1-\frac{1}{\sqrt{8}}}2^{1.5j}\E\abs*{ \dhat_{2^j}}^2
    \\ & \stackrel{(\romannumeral1)}{\le} 2\E \abs*{ \gammahat(x,x'; S_1) -\gamma_i(x',x')}^2  + O\prn*{\frac{G^4\norm*{x-x'}^4}{\eps'^2}}
    \\ & =  O\prn*{\eps'^2 e^{\frac{\lambda \norm*{x-\bx}^2}{\eps'}}\E \prn*{e^{\frac{\ell_{S_1}(x)-\ell_{S_1}(x')}{\eps'}}-1}^2+\prn*{\frac{G^4\norm*{x-x'}^4}{\eps'^2}}}
    \\ & \stackrel{(\romannumeral2)}{\le} O\prn*{\eps'^2 e^{\frac{\lambda \norm*{x-\bx}^2}{\eps'}}\E \prn*{\frac{\ell_{S_1}(x)-\ell_{S_1}(x')}{\eps'}}^4+\prn*{\frac{G^4\norm*{x-x'}^4}{\eps'^2}}}
    \\ & \stackrel{(\romannumeral3)}{\le} O\prn*{ \frac{G^4\norm*{x-x'}^4}{\eps'^2}}
    \\ & \stackrel{(\romannumeral4)}{\le} O\prn*{G^2 \norm*{x-x'}^2}
\end{flalign*}
with $(\romannumeral1)$ following from \Cref{lem:DhatBound},
$(\romannumeral2)$ follows from the inequality $e^x -1 \le x+ 2x^2 = O(x^2)$ for $x \le 3$ with $x=\frac{\ell_{S_1}(x)-\ell_{S_1}(x')}{\eps'} \le 2$,
 $(\romannumeral3)$ follows since each $\ell_j$ is $G$-Lipschitz and since $e^{\frac{\lambda \norm*{x-\bx}^2}{\eps'}}=O(1)$
 and $(\romannumeral4)$ since $\frac{G^2\norm*{x-x'}^2}{\eps'^2} \le  \frac{G^2 4r^2}{\eps'^2} = 4$.
Therefore,
\begin{flalign*}
    \E \norm*{\prn*{ \gammamlmc[x,x'] - \gamma_i(x',x')} \prn*{\nabla \ell_{j}(x)+ \lambda\prn*{x-\bx}} }^2  
     \le  O\prn*{G^4\norm*{x-x'}^2 }.
\end{flalign*}
Finally we bound $\mlmcConstant$ using  the ball constraint $x' \in \ball_r(\bx)$  and the fact that each $\mc{L}_i$ is $G$-Lipschitz, 
therefore \[\mlmcConstant = \frac{\sum_{j \in [\ngroups]}e^{\mc{L}_j(x')/\eps'}}{\sum_{j \in [\ngroups]}e^{\mc{L}_j(\bx) / \eps'}} = \frac{\sum_{j \in [\ngroups]}e^{\frac{\mc{L}_j(x')-\mc{L}_j(\bx)}{\eps'}}e^{\frac{\mc{L}_j(\bx)}{\eps'}} }{\sum_{j \in [\ngroups]}e^{\frac{\mc{L}_j(\bx)}{\eps'} }}  \le e.\] 
Substituting back the bounds on each expectaion term and the bound on $\mlmcConstant$ into \eqref{eq:SVRGvarianceBound} we get
\begin{flalign*}
    \Var\prn*{\ghat_{\bx}(x)} \le O \prn*{L^2+  \lambda ^2 + \frac{G^4}{\eps'^2}}\norm*{x-x'}^2
    \le O\prn*{\prn*{L+  \lambda  + \frac{G^2}{\eps'}}^2\norm*{x-x'}^2 }.
\end{flalign*}
\end{proof}
\subsection{Complexity of the reduced-variance BROO implementation}\label{ssec:GroupDROSVRG-proof}
We first state the complexity bounds of KatyushaX$^{s}$ \cite{allen2018katyusha}
\begin{lem}[{\cite[][Theorems 1 and 4.3]{allen2018katyusha}}]\label{lem:KatyushaGuarantee}
    Let F be a $\lambda$-strongly convex function with minimizer $x_\star$
    and let $\ghat_{x'}(x)$ be a stochastic gradient estimator satisfying the properties
    \begin{enumerate}
        \item {$\E \brk*{\ghat_{x'}(x)} = \E \nabla F(x)$}
        \item {$\E \brk*{\ghat_{x'}(x) - \nabla F(x)}^2 \le \widetilde{L}^2\norm*{x - \bx}^2$}
        \item $\ghat_{x'}(\cdot)$ has evaluation complexity $O(1)$ and preprocessing complexity $O(N)$, 
    \end{enumerate}
    then KatyushaX$^{s}$ with the stochastic gradient estimator $\ghat_{x'}$ finds a point $x$ satisfying $\E\brk*{F(x)- F(x_\star)} \le \eps$
    with complexity
    \begin{align*}
        O\prn*{
            \prn*{N+\frac{N^{3/4}\sqrt{\widetilde{L}}}{\sqrt{\lambda }}} \log\prn*{\frac{F(x_0) - F(x_\star)}{\eps}}
        }.
    \end{align*}
\end{lem}
Applying \Cref{lem:KatyushaGuarantee} with $\widetilde{L}=O\prn*{L+\frac{G^2}{\eps}}$ and accuracy $\frac{\lambda \delta^2}{2}$ gives the following result.
\svrgBrooComplex*
\begin{proof}
	The proof is structured similarly to the proof of \Cref{thm:BrooComplexGoupDRO}.
	
\paragraph{BROO implementation: correctness.}
From \Cref{lem:SmGammaProperties} we have that $\Gamma_{\eps, \lambda}$ is $\Omega(\lambda)$-strongly convex,
in addition, \Cref{lem:SVRGvarianceBound} and \Cref{lem:unbiasedSVRG} show that the stochastic gradient estimator  defined in \eqref{eq:SVRG} is unbiased with $\Var\prn*{\ghat_{x'}(x)}\le \widetilde{L}^2\norm*{x-x'}^2$.
Thus, we can directly apply \Cref{lem:KatyushaGuarantee}  with $F=\Gamma_{\eps, \lambda}$ and accuracy $\frac{\lambda \delta^2}{2}$
and obtain a valid BROO response. 
\paragraph{BROO implementation: complexity.}
We use KatyushaX$^{s}$ \cite{allen2018katyusha} for the BROO implementation. 
Applying \Cref{lem:KatyushaGuarantee} with $F=\Gamma_{\eps,\lambda}$, $x_0 = \bx$,  $\widetilde{L} = O\prn*{L+\lambda+\frac{G^2}{\eps'}} \le O\prn*{L+\frac{G^2}{\eps'}}$ and accuracy $\frac{ \lambda \delta^2}{2}$ the complexity of our implementation is
\begin{align}\label{eq:BROOcomplexyGroupDROSVRG}
    O\prn*{
        \prn*{N+N^{3/4}\frac{\sqrt{L \eps' }+ G}{\sqrt{\eps'}}} \log\prn*{\frac{\Gamma_{\eps,\lambda}(\bx) - \min_{x_\star \in \ball_{\reps}(\bx)}\Gamma_{\eps,\lambda}(x_\star)}{\lambda\delta^2}}
    }
\end{align}
and note that  $\Gamma_{\eps, \lambda}(\bx) - \min_{x_\star \in \ball_{\reps}(\bx)}\Gamma_{\eps, \lambda}(x_\star) \le G \reps $,
since (from \Cref{lem:SmGammaProperties}) $\Gamma_{\eps, \lambda}$ is $O(G)$-Lipschitz.  
\paragraph{Minimizing $\groupObjective$: correctness.}
Similarly to the proof of \Cref{thm:BrooComplexGoupDRO}, combining the guarantees of \Cref{prop:MainProp} and \Cref{lem:GroupSMApproxGroupObj},
with probability at least $\half$ the output $\bx$ of \Cref{alg:acceleratedProxPoint} satisfies $\groupObjective(\bx)-\min_{x_\star \in \xset}\groupObjective(x_\star) \le \eps$.
\paragraph{Minimizing $\groupObjective$: complexity.}
    The complexity of finding $\eps/2$-suboptimal solution for $\Lsm$ and therefore an $\eps$-suboptimal solution for $\groupObjective$, is bounded by \Cref{prop:MainProp} as:
    \begin{equation}
    	O\prn*{
    		\prn*{\frac{R}{\reps}}^{2/3} \brk*{
    			 \prn*{\sum_{j=0}^{\meps}\frac{1}{2^j}\broocost[\frac{\reps}{ 2^{j/2}\meps^2}][\lmin]}\meps 
    			+
    			\prn*{ \broocost[\reps][\lmin] + N }\meps^3
    		}
    	}
    \end{equation} 
    where $\meps = O\prn[\big]{\log \frac{GR^2}{\epsilon \reps}}=\log\prn*{\frac{GR}{\eps}\log \ngroups}$ and $\lmin = O\prn[\big]{\frac{\meps^2 \epsilon}{r^{4/3}R^{2/3}}}$.   
    We first show the complexity of the  BROO implementation
    with $\delta_j =\frac{\reps}{ 2^{j/2}\meps^2}$ (for the MLMC implementation) and with $\delta = \prn*{\frac{\reps}{30}}$ for the bisection procedure. 
    Using \eqref{eq:BROOcomplexyGroupDROSVRG}  we get:
    \begin{enumerate}
        \item {$\broocost[\frac{\reps}{ 2^{j/2}\meps^2}][\lmin] = O\prn*{\prn*{N+N^{3/4} \prn*{\frac{G \sqrt{\log \ngroups}}{\sqrt{\eps}}+ \sqrt{L}} \frac{1}{\sqrt{\lmin}}}\log \prn*{\frac{ \eps' 2^{j}\meps^4}{\lmin \reps^2}}}$}
        \item { $ \broocost[\frac{\reps}{30}][\lmin] = O\prn*{\prn*{N+N^{3/4} \prn*{\frac{G \sqrt{\log \ngroups}}{\sqrt{\eps}}+ \sqrt{L}} \frac{1}{\sqrt{\lambda_k}}} \log \prn*{ \frac{ \eps'}{\lmin \reps^2}  }} 
         $ }
    \end{enumerate}
    From the definitions of $\lmin$ and $\reps$  we have $\frac{ \eps'}{\lmin \reps^2} = O\prn*{\prn*{\frac{GR}{\eps}}^{2/3}\frac{1}{\meps^2}}$ and
    $\frac{1}{\sqrt{\lmin}} = O\prn*{\frac{R^{1/3}\reps^{2/3}}{\meps \sqrt{\eps}}} $,    
   therefore, 
   \begin{align*}
    \meps\sum_{j=0}^{\meps}\frac{1}{2^j}\broocost[\frac{\reps}{ 2^{j/2}\meps^2}][\lmin] & = 
    O\prn*{\meps\sum_{j=0}^{\meps}\frac{1}{2^j} \prn*{N+N^{3/4} \prn*{\frac{G \sqrt{\log \ngroups}+ \sqrt{L \eps}}{\sqrt{\eps}}}  \sqrt{\frac{R^{2/3}\reps^{4/3}}{\eps\meps^2}}} \log \prn*{\frac{ GR\meps^22^{j}}{\eps} }}
   \\ & \le O\prn*{\meps^2 \brk*{ N +N^{3/4}\prn*{\frac{GR^{1/3}}{\eps}+\sqrt{\frac{LR^{2/3}}{\eps}}}\reps^{2/3}}}.
\end{align*}
Similarly, we have that 
\begin{align*}
    O\prn*{ \prn*{ \broocost[\reps][\lmin] + N }\meps^3 } & 
    = O\prn*{\meps^3
        \prn*{
            N+N^{3/4} \prn*{
                \frac{G \sqrt{\log \ngroups}+ \sqrt{L}}{\sqrt{\eps}} 
                \sqrt{\frac{R^{2/3}\reps^{4/3}}{\eps\meps^2}}
                }
            } 
            \log \prn*{\prn*{\frac{GR}{\eps}}\frac{1}{\meps^2} }
        }
        \\ & \le O\prn*{\meps^4N + \meps^{3.5} N^{3/4} \prn*{
            \frac{GR^{1/3}}{\eps} + \sqrt{\frac{LR^{2/3}}{\eps}} 
            }\reps^{2/3}
        }.
\end{align*}
Substituting the bounds into \Cref{prop:MainProp} with $\meps=\log \prn*{\frac{GR}{\eps}\log \ngroups}$ and $\reps=\frac{\eps}{2 \log \ngroups}$ the total complexity is 
\[ O\prn*{
    N\prn*{ \frac{GR}{\eps}}^{2/3} \log^{14/3}\prn*{\frac{GR}{\eps}\log \ngroups} +
    N^{3/4} \prn*{
    \frac{GR}{\eps} + \sqrt{\frac{LR^2}{\eps}} 
    }  \log^{7/2}\prn*{\frac{GR}{\eps}\log \ngroups}}. \] 
\end{proof}
\subsection{Helper lemmas}
\begin{lem}\label{lem:ExpectationOfBoundedCenteredVariables}
    Let $Y_1,\ldots,Y_n$ be a sequence of random i.i.d variables
    such that for every $i \in [n]$ and a constant $c > 0$ we have that $\E\brk*{Y_i}=0$ and $|Y_i|\le c$ with probability 1. Then
    \begin{equation*}
        \E \prn*{\sum_{i=1}^n Y_i}^4 \le O\prn*{n^2 c^4}.
    \end{equation*}
\end{lem}
\begin{proof}
    For $i \ne j$ we have that $\E\brk*{Y_i Y_j} = \E\brk*{Y_i}\E\brk*{Y_j} = 0$, therefore  
    \begin{align*}
        \E \prn*{\sum_{i=1}^n Y_i}^4 & = 
         \sum_{i=1}^n \E\brk*{Y_i^4 }
         + 3 \sum_{i=1}^n \sum_{j\ne i} \E\brk*{Y_i^2}\E\brk*{ Y_j^2}
         \\ & \le nc^4 + 3n \prn*{\frac{n-1}{2}}c^4 =  O\prn*{n^2 c^4}
    \end{align*}
\end{proof}
\section{DRO with $f$-divergence}\label{sec:psiDiverProofs}
In this section we provide the proofs for the results in \Cref{sec:psiDivergence}. 
In \Cref{ssec:dualFormulation} we provide the derivation of the dual formulation in \eqref{eq:penalizedFdiver}. 
In \Cref{ssec:RegConstraintProblem} we show how to reduce the constrained problem \eqref{eq:mainProblemfDiverDRO} to a regularized 
problem of the form \eqref{eq:penalizedFdiver}, then in \Cref{ssec:LpsiProperties} we describe the properties of $\Leps$, the approximation of \eqref{eq:penalizedFdiver}. 
In \Cref{ssec:StabilityLikelihoodProofs} we provide the proofs for our main technical contribution and give guarantees on the stability of the gradient estimators.
Last, in \Cref{ssec:DualEpochSGDProofs,ssec:VarianceReductionSVRGproofsproofs} we give the complexity guarantees of our implementation in the non-smooth and slightly smooth cases. 
\subsection{Dual formulation of DRO with $f$-divergence}\label{ssec:dualFormulation}
Here we give the derivation of the objective in \eqref{eq:penalizedFdiver}, also considered in prior work  \cite[e.g.,]{namkoong2016stochastic,levy2020large,jin2021non}.
Recall the DRO with $f$-divergence objective: 
\[
	\mathcal{L}_{f\text{-div}}(x) \defeq \max_{q\in\Delta^N: \sum_{i\in[N]}\frac{f(Nq_i)}{N} \le 1}\sum_{i \in [N]}q_i \ell_i(x).
\]
We first show the relation between $\mathcal{L}_{f\text{-div}}$
and its regularized form \eqref{eq:penalizedFdiver}. 
Using the Lagrange multiplier $\nu$ for the constraint $\sum_{i\in[N]}\frac{f(Nq_i)}{N} \le 1$ and strong duality we get
\[
	\mathcal{L}_{f\text{-div}}(x) = \min_{\nu \ge 0}\crl[\Bigg]{ \nu +\max_{q\in\Delta^N}
		\sum_{i\in[N]} \prn*{q_i \ell_i(x) - \frac{\nu}{N}f(Nq_i)}
		}= \min_{\nu \ge 0}\crl[\Bigg]{ \nu + \mc{L}_{\nu \cdot f}(x)},
\]
where $\mc{L}_{\nu \cdot f}(x)$ is the regularized form of $\mathcal{L}_{f\text{-div}}$: writing $\psi(x) = \frac{\nu}{N}f(Nx)$, with slight abuse of notation  we have 
\[
	\mc{L}_{\nu \cdot f}(x) = \Lpenalized(x) = \max_{q\in\Delta^N}\crl[\Bigg]{ 
		\sum_{i\in[N]} \prn*{q_i \ell_i(x) - \psi(q_i)}
		}.
\]
Adding a Lagrange multiplier $y$ for the constraint that $q\in\Delta^N$ and using strong duality again gives
\begin{align*}
    \Lpenalized(x) &=  \max_{q\in\R^N_+}\min_{y \in \R}  \crl[\Bigg]{ \sum_{i\in[N]} \prn*{q_i \ell_i(x) - \psi(q_i) -Gy \cdot q_i } +Gy} \\ &=   \min_{y \in \R} \crl[\Bigg]{ \sum_{i\in[N]}\max_{q_i\in\R_+} \prn*{q_i \ell_i(x) - \psi(q_i) -Gy \cdot q_i } +Gy}.
\end{align*}
Finally, using $\psi^{*}(v)\defeq \max_{t \in \mathrm{dom}(\psi)}\crl*{vt - \psi t}$ (the Fenchel conjugate of $\psi$), we have 
\begin{align*}
    \Lpenalized(x) =    \min_{y \in \R}  \crl[\Bigg]{\sum_{i\in[N]} \psi^* \prn*{ \ell_i(x) -Gy} +Gy}.
\end{align*}
\subsection{Minimizing the constrained objective using the regularized objective}\label{ssec:RegConstraintProblem}
In this section, we show that under the following \Cref{assumption:BoundedLoss,assumption:BoundedF} 
we can reduce the constrained problem of minimizing \eqref{eq:mainProblemfDiverDRO} to the regularized problem of minimizing  \eqref{eq:penalizedFdiver}
by computing a polylogarithmic number of $O(\eps)$-accurate minimizers of \eqref{eq:penalizedFdiver}.
\begin{assumption}\label{assumption:BoundedLoss}
    Each loss function $\ell_i$ is bounded, i.e., $\ell_i: \xset \rightarrow [0,B_\ell]$ for every $i \in [N]$.
\end{assumption}
\begin{assumption}\label{assumption:BoundedF}
    For any uncertainty set of the form $\mc{U} = \crl{q \in \Delta^N:  D_f(q,p) \le 1}$, 
    the divergence function $f$ is bounded, i.e., $f:\R_+ \rightarrow [0, B_f]$ for some $B_f \ge 1$.
\end{assumption}
We note that the above assumptions are weak since the complexity of our approach only depends logarithmically on on $\frac{B_f B_\ell}{\epsilon}$.

We first cite a result on noisy one dimensional bisection, as given by a guarantees on the OneDimMinimizer algorithm in \citet{cohen2016geometric}. 
\begin{lem}[{Lemma 33, \citet{cohen2016geometric}}]\label{lem:bisection}
    let $f: \R \rightarrow \R$ be a $B$-Lipschitz convex function defined on the interval $[\ell, u]$, 
    and let $\mc{G}: \R \rightarrow \R$ be an oracle such that $\abs*{\mc{G}(y)-f(y)}\le \widetilde{\eps}$ for all y. 
    With $O\prn*{\log\prn*{\frac{B(u-\ell)}{\widetilde{\eps}}}}$ calls to $\mc{G}$, the algorithm \textup{OneDimMinimizer} \cite[Algorithm 8]{cohen2016geometric} outputs $y'$
    such that 
    \begin{align*}
        f(y') - \min_y f(y) \le 4 \widetilde{\eps}
    \end{align*}
\end{lem}
We specialize \Cref{lem:bisection} to our settings and provide the complexity guarantees of minimizing \eqref{eq:mainProblemfDiverDRO} to an $\eps$-accurate solution using a noisy oracle $\mc{G}$.

\begin{restatable}{prop}{BisectionConstraintProblem}\label{lem:BisectionConstraintProblem}
    Let each $\ell_i$ satisfy \Cref{assumption:BoundedLoss} and let $f$ satisfy \Cref{assumption:BoundedF}.
    Define the function $h(\constraintMul) \defeq \min_{x\in \xset} \mc{L}_{\constraintMul \cdot f }(x) +\constraintMul$
    with $\mc{L}_{\constraintMul  \cdot f }$ defined in \eqref{eq:penalizedFdiver} 
    and let $\mc{G}$ be an oracle such that $\mc{G}(\nu) \ge h(\nu)$ with probability 1 and 
    $\mc{G}(\nu) -  h(\constraintMul) \le \frac{\eps}{5}$ with probability at least $\half$.   
    Then applying \textup{OneDimMinimizer} \cite[Algorithm 8]{cohen2016geometric} on the interval $[0, B_\ell]$
    outputs $\constraintMul'$ that with probability at leat $\frac{99}{100}$ satisfies 
    \begin{align*}
        \mc{G}(\constraintMul') - \min_{\constraintMul \ge 0} h(\constraintMul) = \mc{G}(\constraintMul') - \min_{x\in\xset} \mc{L}_{\mathrm{f-div}}(x) \le \eps
    \end{align*}
    using $O\prn*{\log (H) \log\prn*{\log H}}$ calls to $\mc{G}$, where $H=\frac{B_f B_\ell }{\eps}$. 
\end{restatable}

\begin{proof}
    Let $\hat{h}_{q,x}(\constraintMul) \defeq \sum_{i \in [N]} q_i \ell_i(x) - \constraintMul \prn*{\frac{1}{N}\sum_{i\in [N]}f(Nq_i)-1}$
    and note that for any $q$ and $x$ the function $\hat{h}_{q,x}$ is $B_f$-Lipschitz, since it is linear in $\constraintMul$ and $\abs*{\frac{1}{N}\sum_{i\in [N]}f(Nq_i)-1} \le B_f$.
    Minimization and maximization operations preserve the Lipschitz continuity and therefore the function $h(\constraintMul) = \min_x \mc{L}_{\nu \cdot f}(x) = \min_x \max_q \hat{h}_{q,x}(\constraintMul)$ is also $B_f$-Lipschitz continuous.
    In addition for the $q^\star \in \Delta^N$ that maximizes $\mc{L}_{\constraintMul  \cdot f }(x)$ we have that 
     \[\sum_{i\in [N]} \brk*{q^\star_i \ell_i(x) - \constraintMul \frac{1}{N}f(Nq^\star_i) } \ge \sum_{i\in [N]} \frac{1}{N}\ell_i(x)  \] and rearranging gives 
    \[ \frac{1}{N}\sum_{i \in [N]}f(Nq^\star_i) \le \frac{\sum_{i \in [N]} \brk{q^\star_i \ell_i(x)- \frac{1}{N}\ell_i(x)}}{\constraintMul}\le \frac{B_\ell}{\constraintMul}.\]
    Therefore, for all  $\nu > B_\ell$ we have $h'(\nu) = 1 - \frac{1}{N}\sum_{i \in [N]}f(Nq^\star_i) > 0$ and therefore it suffices to restrict  $h(\nu)$ to  $[0, B_\ell]$.
    Next, to turn $\mc{G}$ into a high-probability oracle, we call it $\log_2\prn*{100 \log\prn*{\frac{B_f B_\ell}{\eps}} }$ times and choose the smallest output.
    Therefore, with probability at least $1-\prn*{\half}^{\log\prn*{100 \log\prn*{\frac{B_f B_\ell}{\eps}} }} = 1 -1/\prn*{100 \log\prn*{\frac{B_f B_\ell}{\eps}} }$
    the result is within $\frac{\epsilon}{5}$ of $h(\nu)$.
    Since $h$ is $B_f$-Lipschitz and defined on $[0,B_\ell]$ 
    we can use  \Cref{lem:bisection} with $\ell=0$, $u=B_f$, $\widetilde{\eps} = \eps / 5$, $B=B_f$ and the high-probability version of $\mc{G}$.
    Therefore, using $O\prn*{\log\prn*{\frac{B_f B_\ell}{\eps}}}$ calls to the high probability version of $\mc{G}$ and applying the union bound,
    we obtain that with probability at least $\frac{99}{100}$ OneDimMinimizer outputs $\constraintMul'$ that satisfies
    \[h(\constraintMul') - \min_{\constraintMul}h(v) \le 4\eps/5 \]
    and therefore
    \[ \mc{G}(\constraintMul') - \min_{\constraintMul}h(v) = \mc{G}(\constraintMul') - h(\constraintMul')+ h(\constraintMul') - \min_{\constraintMul}h(v) \le \eps  .\]
\end{proof}
Finally, in the following corollary we show that finding an $\eps$-suboptimal solution for \eqref{eq:mainProblemfDiverDRO} requires a polylogarithmic number of $O(\eps)$-accurate minimizers of \eqref{eq:penalizedFdiver} and a polylogarithmic number of evaluations of \eqref{eq:penalizedFdiver}. 
\begin{corollary}\label{corol:BisectionConstraintProblem}
    Let each $\ell_i$ satisfy \Cref{assumption:BoundedLoss} and let $f$ satisfy \Cref{assumption:BoundedF}, then 
    minimizing \eqref{eq:mainProblemfDiverDRO} to accuracy $\eps$ with probability at least $\frac{99}{100}$ requires 
     $O\prn*{\log (H) \log\prn*{\log H}}$ evaluations of \eqref{eq:penalizedFdiver} and  $O\prn*{\log (H) \log\prn*{\log H}}$ 
    calls to an algorithm that with probability at least $\half$ returns an $O(\eps)$-suboptimal point of \eqref{eq:penalizedFdiver},
     where $H=\frac{B_f  B_\ell }{\eps}$.
\end{corollary}
\begin{proof}
    Note that $\mc{L}_{\constraintMul \cdot f}$ is defined on  \eqref{eq:penalizedFdiver}. Let $\widetilde{\mc{G}}(\constraintMul)\defeq\mc{L}_{\constraintMul \cdot f}(\widetilde{x}) + \constraintMul$ where $\widetilde{x}$ is the output of an
    algorithm that with probability at least $\half$ returns an $\frac{\eps}{5}$-suboptimal point of $\mc{L}_{\constraintMul \cdot f}$
    and let $h(\constraintMul) \defeq \min_{x\in \xset} \mc{L}_{\constraintMul \cdot f }(x) +\constraintMul$. 
    We have that $\widetilde{\mc{G}}(\constraintMul) - h(\constraintMul) \le \frac{\eps}{5}$ with probability at least $\half$, 
    therefore, we can apply \Cref{lem:BisectionConstraintProblem} with $\mc{G} = \widetilde{\mc{G}}$, and obtain that with $O\prn*{\log \prn*{\frac{B_f \cdot B_\ell }{\eps}}  \log\prn*{\log \frac{B_f \cdot B_\ell }{\eps}}}$ calls to $\widetilde{\mc{G}}$ (i.e., to an algorithm that outputs $\frac{\eps}{5}$-suboptimal minimizer of $\mc{L}_{\constraintMul \cdot f}$ with probability at least $\half$), OneDimMinimizer outputs $\constraintMul'$ that satisfies with probability at least $\frac{99}{100}$  
    \[  \widetilde{\mc{G}}(\constraintMul' ) - \min_{\constraintMul}h(\constraintMul) =\widetilde{\mc{G}}(\constraintMul' ) -  \min_x \mc{L}_{f\text{-div}}(x) \le \eps.\] 
    Noting that $\mc{L}_{f\text{-div}}(x) = \min_{\constraintMul \ge 0}\crl*{\mc{L}_{\constraintMul \cdot f}(x) + \constraintMul}$ we obtain 
    \[ \mc{L}_{f\text{-div}}(\widetilde{x}) - \min_x \mc{L}_{f\text{-div}}(x)
    \le \mc{L}_{\constraintMul' \cdot f}(\widetilde{x}) + \constraintMul' -   \min_x \mc{L}_{f\text{-div}}(x) = \widetilde{\mc{G}}(\constraintMul') - \min_x \mc{L}_{f\text{-div}}(x)  \le \eps.\] 
\end{proof}

\Cref{corol:BisectionConstraintProblem} means that the complexity bounds for approximately minimizing the objective $\Lpenalized$ established by \Cref{thm:BROOcomplexityDualProblem,thm:DualSVRGBROOcomplexity} also apply (with slightly larger logarithmic factors) to approximately minimizing the constrained $f$-divergence objective $\mathcal{L}_{f\text{-div}}$.

\subsection{Properties of $\Lpenalized$ and $\Leps$}\label{ssec:LpsiProperties}

\begin{lem}\label{lem:SMapproxMaxFdiv}
    For $\Lpenalized$ defined in \eqref{eq:penalizedFdiver} and $\Leps$ defined in \eqref{eq:upsiloneps defintion} we have that 
    \[ \abs*{ \Leps(x) - \Lpenalized(x)} \le \frac{\eps}{2} \text{~~for all~~} x \in \R^d.\]  
\end{lem}
\begin{proof}
    Recall that $\eps' = \frac{\eps}{2 \log N}$ and for $q \in \Delta^N$ we have that $\sum_{i \in N}q_i \log q_i \in [-\log N, 0]$, therefore: 
    \begin{align*}
        \abs*{\Leps(x) -\Lpenalized(x)  } & = 
        \abs*{  \max_{q\in\Delta^N} \crl*{\sum_{i\in[N]} \prn*{q_i \ell_i(x) - \psi(q_i) - \eps' q_i \log q_i}} - \max_{q\in\Delta^N} \crl*{\sum_{i\in[N]} \prn*{q_i \ell_i(x) - \psi(q_i)}}  }
        \\ & \le \abs*{\eps' \sum_{i \in [N]}q_i \log q_i} \le \eps' \log N = \eps / 2.
    \end{align*}
\end{proof}
\subsection{Gradient estimator stability  proofs}\label{ssec:StabilityLikelihoodProofs}
\psiProperties*
\begin{proof} 
	While we write the proof as though the function $\psi$ is differentiable with derivative $\psi'$, one may readily interpret $\psi'$ as an element in the subdifferential of $\psi$ and the proof continues to hold.
	
	Let $\phi_\eps = \epsilon' q \log q$ and recall that 
	$\psieps(q) = \psi(q) + \phi_\eps(q)$. Fix any two numbers $v_1,v_2\in\R$ and assume without loss of generality that $v_2 > v_1$. For $i=1,2$, let
	\begin{equation}\label{eq:qi-def}
		q_i \defeq {\psieps^{*}}'(v_i)
		~~\mbox{and}~~
		p_i \defeq {\phi_\eps^{*}}'(v_i) = e^{v_i/\epsilon'-1}.
	\end{equation}

	Note that by definition of the Fenchel dual (and strict convexity of $\psieps$), $q_i$ is the unique solution to
	\begin{equation*}
		v_i = \psieps'(q_i) = \psi'(q_i) + \phi_\eps'(q_i)
	\end{equation*}
	and moreover that $q_2 \ge q_1$ since $\psieps^{*}$ is convex and therefore ${\psieps^{*}}'$ is non-deceasing. Similarly, $p_i$ is the unique solution to
	\begin{equation*}
		v_i =  \phi_\eps'(p_i)
	\end{equation*}
	and $p_2 > p_1$. 
	Combining the two equalities yields
	\begin{equation*}
		v_2-v_1 = \psi'(q_2) + \phi_\eps'(q_2) - \psi'(q_1) - \phi_\eps'(q_1)
		= \phi_\eps'(p_2) - \phi_\eps'(p_1).
	\end{equation*}
	Rearranging, we find that
	\begin{equation*}
		0 \le \phi_\eps'(q_2) - \phi_\eps'(q_1) = \phi_\eps'(p_2) - \phi_\eps'(p_1)
		- [\psi'(q_2)-\psi'(q_1)] \le \phi_\eps'(p_2) - \phi_\eps'(p_1).
	\end{equation*}
	where $\psi'(q_2)-\psi'(q_1) \ge 0$ holds by convexity of $\psi$ and $q_2\ge q_1$. Recalling that $\phi_\eps'(q) = \eps' + \eps' \log q$, we have
	\begin{equation*}
		0 \le \log q_2 - \log q_1 \le \log p_2 - \log p_1 = \frac{1}{\epsilon'}(v_2 - v_1),
	\end{equation*}
	where the last equality follows by substituting the definition of $p_i$.  The proof is complete upon recalling that $q_i = {\psieps^{*}}'(v_i)$.
\end{proof}

\BoundEtaDistance*
\begin{proof}
    For $x, x' \in \xset$ w.l.o.g.\ assume that $y^\star(x) \le y^\star(x')$  and observe that for every $u \in \xset$
    \begin{equation}\label{eq:etaDerivative}
        \sum_{i \in [N]}{\psieps^{*}}'\prn*{\ell_i(u)- y^\star(u)}  = 1.
    \end{equation}
    Let $\widetilde{\ell}(x) =\ell(x)  + \delta$ with $\delta \defeq \norm{\ell(x') -\ell(x) }_\infty$ 
    and  \
    $\widetilde{y}(x) \defeq \argmin_{y \in \R} \crl*{ \sum_{i \in [N]}\psieps^{*}\prn*{\widetilde{\ell}_i(x)- Gy} + Gy }$.
    Then, according to \eqref{eq:etaDerivative}
    \begin{equation*}
        \sum_{i \in [N]}{\psieps^{*}}'\prn*{\widetilde{\ell}_i(x)- G\widetilde{y}(x)}  = \sum_{i \in [N]}{\psieps^{*}}'\prn*{\ell_i(x) + \delta - G\widetilde{y}(x)} = \sum_{i \in [N]}{\psieps^{*}}'\prn*{\ell_i(x)- Gy^\star(x)}= 1
    \end{equation*}
    and therefore \[G\widetilde{y}(x)=   Gy^\star(x) + \delta. \]
    By convexity, ${\psieps^{*}}'$ is monotonically non decreasing, thus 
    \begin{equation}\label{eq:derivativeSum}
        \sum_{i \in [N]}{\psieps^{*}}'\prn*{\ell_i(x')- G\widetilde{y}(x)}
        \stackrel{(\romannumeral1)}{\le}  \sum_{i \in [N]}{\psieps^{*}}'\prn*{\widetilde{\ell}_i(x)- G\widetilde{y}(x)} =  \sum_{i \in [N]}{\psieps^{*}}'\prn*{\ell_i(x')- Gy^\star(x')} = 1
    \end{equation}
    where $(\romannumeral1)$ follows from noting that $\ell_i(x') \le \ell_i(x) + \max_{i \in [N]}\abs{\ell_i(x')-\ell_i(x)}=\widetilde{\ell}_i(x)$.
    Therefore, 
    $Gy^\star(x') \le G\widetilde{y}(x)=   Gy^\star(x) + \delta$ giving  
    \[G\abs*{y^\star(x')-y^\star(x)} \le \norm{\ell(x') -\ell(x) }_\infty.\]
    In addition, if each $\ell_i$ is $G$-Lipschitz we have 
    \[G\abs*{y^\star(x')-y^\star(x)} \le \norm{\ell(x') -\ell(x) }_\infty \le G\norm{x' -x}.\]
\end{proof}

\subsection{Epoch-SGD BROO implementation}\label{ssec:DualEpochSGDProofs}
In this section we provide the analysis of our algorithm in the non-smooth case, which consists of combining our general BROO acceleration scheme (\Cref{alg:acceleratedProxPoint}) with a variant of Epoch-SGD \cite{hazan2014beyond} that we specialize in order to implement a BROO for $\upsilonepsreg$ (\Cref{alg:dualEpochSGD}).
We organize this section as follows. First, we prove \Cref{lem:dualGradEstProperties} showing that our gradient estimators are unbiased with bounded second moment, and therefore 
can be used in  \Cref{alg:dualEpochSGD}. Then, in \Cref{prop:epochSGDbounds} we give the convergence rate of \Cref{alg:dualEpochSGD}. 
Combining the previous statements with the guarantees of \Cref{prop:MainProp} we prove  \Cref{thm:BROOcomplexityDualProblem}.

For convenience, we restate the definitions of $\upsiloneps$ and our stochastic estimators for $\nabla_x \upsiloneps(x,y)$ and $\nabla_y \upsiloneps(x,y)$:
 \[ \upsiloneps(x,y) \defeq \sum_{i \in [N]} \psieps^*\prn*{\ell_i(x) -G y} + Gy,\]
and
\begin{equation*}
    \gradx(x,y) = \frac{{\psieps^{*}}'(\ell_i(x)-Gy)}{\bp_i} \nabla \ell_i(x,y) \text{~~,~~}
	\grady(x,y) = G \prn*{ 1 -\frac{{\psieps^{*}}'(\ell_i(x)-Gy)}{\bp_i}}
\end{equation*} 
where $\bp_i = {\psieps^{*}}'\prn*{\ell_i(\bx)-G\by}$. 
\dualGradEstProperties*
\begin{proof}
    We first show that the stochastic gradients $ \gradx,  \grady$ are unbiased
    \begin{align*}
         \E_{i \sim \bp_i}\brk*{\gradx(x,y)} =\sum_{i \in [N]} \bp_i \cdot  \frac{{\psieps^{*}}'\prn*{\ell_i(x)-Gy}}{\bp_i}  \nabla \ell_i(x) =  \nabla_x  \upsiloneps(x, y),
    \end{align*}
	and
    \begin{align*}
        \E_{i \sim \bp_i}\brk*{\grady(x,y)} =\sum_{i \in [N]} \bp_i \cdot  \prn*{ G \prn*{1 -\frac{{\psieps^{*}}'(\ell_i(x)-Gy)}{\bp_i}}  } = G - G \sum_{i \in [N]}  {\psieps^{*}}'\prn*{\ell_i(x)-Gy}=  \nabla_y  \upsiloneps(x, y).
    \end{align*}
    Next, we bound the second moment of the stochastic gradients. For any $i$ we have
    \begin{equation*}
        \norm*{\gradx(x,y)}  = \frac{{\psieps^{*}}'(\ell_i(x)-Gy)}{{\psieps^{*}}'(\ell_i(\bx)-G\by)} \norm*{\nabla \ell_i(x)} 
        \overle{(i)}
        e^{\frac{\ell_i(x)-Gy - \prn*{\ell_i(\bx)-G\by}}{\eps'}} G
        \overle{(ii)} e^2 G 
    \end{equation*}
   where $(\romannumeral1)$ follows from  \Cref{lem:psiProperties} and the fact that $\ell_i$ is $G$-Lipschitz and $(\romannumeral2)$ uses $G$-Lipschitzness again together with $x\in\ball_{\reps}(\bx)$ and $y\in[\by-\reps,\by+\reps]$ to deduce that $\ell_i(x)-Gy - \prn*{\ell_i(\bx)-G\by} \le 2G \reps \le 2\epsilon'$. Therefore, we have $\E\norm*{\gradx(x,y)}^2 \le e^4 G^2$ as required. The second moment bound on $\grady(x,y)$ follows similarly, since
   	\begin{equation*}
   		\abs{\grady(x,y)} \le G \max\crl*{1, \frac{{\psieps^{*}}'(\ell_i(x)-Gy)}{{\psieps^{*}}'(\ell_i(\bx)-G\by)}} \le e^2 G.
   	\end{equation*}
\end{proof}

\begin{algorithm2e}[t]
	\DontPrintSemicolon
	\caption{Dual EpochSGD}	\label{alg:dualEpochSGD}
	\KwInput{The function $\upsiloneps$ defined in \eqref{eq:upsiloneps defintion}, 
    ball center $\bx$, 
     ball radius $r_{\eps}$, regularization parameter $\lambda$, smoothing parameter $\eps'$ and iteration budget $T$.}
	\KwParameters{Initial step size $\gamma_1 = 1/(16\lambda)$, epoch length $T_1 = 128$ and threshold $\Tthreshold=\frac{G^4}{\lambda^2 \epsilon'^2}$.}
    Initialize $x^{(0)}_1  = \bx$\; 
    Initialize $y^{(0)}_1 = \by = \argmin_{y \in \R}\upsiloneps(\bx, y)$\;
    Precomupte sampling probabilities $\bp_i = {\psieps'}^*\prn*{\ell_i\prn*{\bx } - G\by}$ \;
	\For{$k=1,\ldots,\ceil{\log\prn*{T/128 + 1}}$}{
	\For{$t = 0,2,\cdots T_k - 1$}{
        Sample $i \sim \bp_i$\;
        Query stochastic gradients $\gradx\prn*{x_k^{(t)},y_k^{(t)} }  \text{~and~} \grady\prn*{x_k^{(t)},y_k^{(t)} }  $ defined in \eqref{eq:dualGradients}\;
        Update $ x_k^{(t+1)} = \argmin_{x \in \ball_{\reps}(\bx)}\crl*{ \gamma_k \prn*{\tri*{\gradx, x}
         + \frac{\lambda}{2}\norm{\bx -x}^2} + \frac{1}{2}\norm{x_k^{(t)}-x}^2 } $\;
        Update $ y_k^{(t+1)} = \argmin_{y \in [\by  -\frac{\eps'}{G}, \by + \frac{\eps'}{G}]}\crl*{ \gamma_k \prn*{\grady\cdot y} + \frac{1}{2}\prn*{y_k^{(t)}-y}^2} $\;
	}
	Set $x^{(0)}_{k+1} = \tfrac{1}{T_k}\sum_{t\in[T_k]}x^{(t)}_k$\;
    Set $y^{(0)}_{k+1} = \tfrac{1}{T_k}\sum_{t\in[T_k]}y^{(t)}_k$\;
    Update $T_{k+1}= 2T_k$\;
    Update  $\gamma_{k+1}= \gamma_k/2$ \; 
    $k\gets k+1$\; 
    \If {$T_k \ge  \Tthreshold$}{
    Recompute $y_{k+1}^{(0)} = \argmin_{y \in \R}\upsiloneps(x_{k+1}^{(0)}, y) $
    }
    }
	\Return $x = x^{(0)}_k$
\end{algorithm2e}

\begin{restatable}{prop}{epochSGDbounds}\label{prop:epochSGDbounds}
     Let $\eps, \lambda  > 0$, $\eps'=\frac{\eps}{2 \log N}$ and $\reps = \frac{\eps'}{G}$. 
     For any query point $\bx$ let $\by = \argmin_{y \in \R} \upsilonepsreg(\bx,y)$ and let $x_\star,y_\star = \argmin_{x \in \ball_{\reps}(\bx),y \in [\by-\reps, \by+\reps]} \upsilonepsreg(x,y)$. 
    For  $\gamma_k =  \frac{1}{8\lambda  2^{k}}$, $T \ge 1$
    and threshold 
    $\Tthreshold =  \frac{G^4}{\lambda^2 \epsilon'^2}$ the output $(x, y)$ of \Cref{alg:dualEpochSGD} satisfies
    \begin{align*}
        \E \upsilonepsreg(x, y) -  \upsilonepsreg(x_\star, y_\star)  \le O\prn*{\frac{  G^2}{ \lambda T} }.
    \end{align*}
\end{restatable}
\begin{proof}
    For convenience let $x_k =x_k^{(0)} $ and $y_k =y_k^{(0)} $, and in addition let $\bp_i={\psieps'}^*\prn*{\ell_i\prn*{\bx } - G\by}$ be the sampling probability from \Cref{alg:dualEpochSGD}. We use induction to prove that \[ \E \upsilonepsreg(x_k, y_k) - \upsilonepsreg(x_\star, y_\star)  \le\frac{ e^4G^2}{ \lambda 2^{k}} \] for all $k$.
    We start with the base case (k=1). 
    \begin{flalign*}
        \upsilonepsreg(x_\star, y_\star) - \upsilonepsreg(x_1, y_1)  &= \upsiloneps(x_\star, y_\star) -\upsiloneps(x_1, y_1)  +\frac{\lambda}{2}\norm{x_\star-x_1}^2
        \\ & \stackrel{(\romannumeral1)}{\ge} \tri*{\nabla_x \upsiloneps(x_1, y_1), x_\star - x_1} + \tri*{\nabla_y \upsiloneps(x_1, y_1), y_\star - y_1} + \frac{\lambda}{2}\norm{x_\star-x_1}^2
        \\ & \stackrel{(\romannumeral2)}{\ge}  \tri*{\nabla_x \upsiloneps(x_1, y_1), x_\star - x_1} + \frac{\lambda}{2}\norm{x_\star-x_1}^2
    \end{flalign*}
    where $(\romannumeral1)$ follows from convexity of $\upsiloneps$ 
    and $(\romannumeral2)$ since due to optimality conditions 
    we have that $\tri*{\nabla_y \upsiloneps(x_1, y_1), y_\star - y_1} \ge 0$.
    Therefore, 
    \begin{flalign*}
        \upsilonepsreg(x_1, y_1) -  \upsilonepsreg(x_\star, y_\star) & \le -\tri*{\nabla_x  \upsiloneps(x_1, y_1), x_\star - x_1} - \frac{\lambda}{2}\norm{x_\star-x_1}^2
        \\ & \le \max_{x \in \ball_{\reps}(\bx)} \crl*{-\tri*{\nabla_x  \upsiloneps(x_1, y_1), x - x_1} - \frac{\lambda}{2}\norm{x-x_1}^2} 
        \\ &  = \frac{\norm{\nabla_x  \upsiloneps(x_1, y_1)}^2}{2\lambda}
        \le \frac{e^4G^2}{2\lambda}
    \end{flalign*}
    where the last inequality follows 
    from Jensen's inequality and \Cref{lem:dualGradEstProperties}. This gives the base case of our induction.
    Let $V_x(x')=\half\norm*{x-x'}^2$ and $V_y(y')= \half \abs*{y-y'}^2$ and
    suppose that there is a $k$ such that $ \E\upsilonepsreg(x_k, y_k)  - \upsilonepsreg(x_\star, y_\star)  \le\frac{ e^4 G^2}{ \lambda 2^{k}} $. Then,  
    using the mirror descent regret bound \cite[see, e.g.,][Lemma 3]{asi2021stochastic} for $k+1$ we get 
    \begin{flalign*}
        \E \upsilonepsreg(x_{k+1}, y_{k+1}) - \upsilonepsreg(x_\star,y_\star)  
        & \le  \frac{\E V_{x_k}(x_\star)}{\gamma_{k}T_k}  +  \frac{\E V_{y_k}(y_\star)}{\gamma_{k}T_k} \\ & + \frac{\gamma_k}{2}  \frac{1}{T_k}\sum_{t=1}^{T_k}\E\norm{\gradx(x_k, y_k)}^2 +  \frac{\gamma_k}{2}  \frac{1}{T_k}\sum_{t=1}^{T_k}\E\norm{\grady(x_k, _k)}^2
        \\ & \stackrel{(\romannumeral1)}{\le}  \frac{\E V_{x_k}(x_\star)}{\gamma_{k}T_k} + \frac{\E V_{y_k}(y_\star)}{\gamma_{k}T_k}+2e^4 \gamma_k G^2
        \\ & \stackrel{(\romannumeral2)}{\le}  \frac{\lambda \E V_{x_k}(x_\star)}{8} + \frac{\lambda  \E V_{y_k}(y_\star)}{8}+\frac{e^4 G^2}{2 \lambda 2^{k+1}} 
    \end{flalign*}
with $(\romannumeral1)$ following from \Cref{lem:dualGradEstProperties} and $(\romannumeral2)$ from the choice of $T_k=64\cdot2^k$ and $\gamma_k= \frac{1}{\lambda 2^{k+3}}$.
Next, note that the strong convexity of $\upsiloneps(x,y)$ in $x$ implies $\lambda \E V_{x_k}(x_\star) \le \frac{ e^4 G^2}{ \lambda 2^{k}} $
since $\lambda \E V_{x_k}(x_\star) \le \E\lambda \upsilonepsreg(x_k, y_k)  - \E \upsilonepsreg(x_\star, y_k) \le  \E\lambda \upsilonepsreg(x_k, y_k) -\upsilonepsreg(x_\star, y_\star ) \le\frac{ e^4 G^2}{ \lambda 2^{k}}$ by the induction hypothesis.
In addition from \Cref{lem:BoundEtaDistance} we have 
\begin{align*}
    G\abs*{y_k - y_\star} \le G\reps = \eps'
\end{align*} 
by the constraint on $y$.
We now bound $\E V_{y_k}(y_\star) $ in each scenario $T_k \le \Tthreshold = \frac{G^4}{\lambda^2 \eps'^2}$ or $T_k > \Tthreshold$. 
\begin{enumerate}
    \item If $T_k \le \Tthreshold$ we have that $\abs{y_k - y_\star}\le \frac{\eps'}{G} \le \frac{G}{\lambda 2^{k/2+2}}$ and thus $V_{y_k}(y_\star) \le \frac{G^2}{\lambda^2 2^{k}}$.
    \item If $T_k > \Tthreshold$ \Cref{alg:dualEpochSGD} will recompute the optimal $y_{k}$ 
    and using \Cref{lem:BoundEtaDistance} we have $\abs{y_k - y_\star} \le \norm{x_k -x_\star}$ and as a result $ V_{y_k}(y_\star) \le V_{x_k}(x_\star)$.
\end{enumerate}
Therefore $\E V_{y_k}(y_\star) \le \frac{e^4G^2}{\lambda^2 2^{k}}$ and substituting back the bounds on $\E V_{y_k}(y_\star)$ and $\E V_{x_k}(x_\star)$ we obtain 
\begin{flalign*}
    \E \upsilonepsreg(x_{k+1}, y_{k+1}) - \upsilonepsreg(x_\star,y_\star)  &\le \frac{e^4 G^2}{ \lambda 2^{k+1}} 
\end{flalign*}
which completes the induction.
Let $K$ be the iteration where the algorithm  outputs $x = x_K^{(0)}$ and let $y=y_K^{(0)}$. Noting that $T= O(2^K) $, we have 
\[    \E \upsilonepsreg(x, y) - \upsilonepsreg(x_\star,y_\star) \le O\prn*{\frac{G^2}{\lambda T}}.\] 

\end{proof}
\BROOcomplexityDualProblem*
\begin{proof}
	We divide the proof into correctness arguments and complexity bounds.
    \paragraph{BROO implementation: correctness.}
    We use \Cref{alg:dualEpochSGD} with  $T=O\prn*{\frac{ G^2}{\lambda^2 \delta^2}}$ for the BROO implementation. 
    Applying \Cref{prop:epochSGDbounds} the output $(x,y)$ of \Cref{alg:dualEpochSGD}  satisfies 
    \[    \E \upsilonepsreg(x, y) - \upsilonepsreg(x_\star,y_\star) \le O\prn*{\lambda \delta^2}.\] 
    Therefore, there is a constant $c>0$ for which  the output $(x,y)$ of \Cref{alg:dualEpochSGD} with $T=\frac{ c G^2}{\lambda^2 \delta^2}$ satisfies
    $ \E \upsilonepsreg(x, y) - \upsilonepsreg(x_\star,y_\star) \le\frac{\lambda \delta^2}{2} $,
   and since
    \[ \E \Leps(x) - \min_x \Leps(x) =  \E \min_y \upsiloneps(x,y) -  \upsiloneps(x_\star,y_\star)\le \E \upsilonepsreg(x, y) - \upsilonepsreg(x_\star,y_\star), \] 
    \Cref{alg:dualEpochSGD} returns a valid BROO response for $\Leps$. 
    \paragraph{BROO implementation: complexity.}
    The total number of epochs that \Cref{alg:dualEpochSGD} performs is $K= O\prn*{\log(\frac{G^2}{\delta^2 \lambda^2})}$.
    In  $O\prn*{\log\prn*{\frac{G^2}{\lambda^2 \delta^2}} - \log\prn*{\frac{G^2}{\lambda^2 \reps^2}}} = O\prn*{\log\prn*{\frac{\reps}{\delta}}}$ epochs with $T \ge \frac{G^4}{\lambda^2 \eps'^2}= \frac{G^2}{\reps^2 \lambda^2}$ 
    the algorithm  performs $O(N)$ function evaluations to recompute the optimal $y$.
    In addition, the algorithm performs $O\prn*{\frac{G^2}{\lambda^2 \delta^2}}$ stochastic gradient computations (each computation involves a single loss function $\ell_i$ and a single sub-gradient $\nabla \ell_i$ evaluation). 
    Therefore the total complexity of the BROO implementation is 
    \begin{align}\label{eq:DualBROOcomplexity}
        O\prn*{\frac{G^2}{\lambda^2 \delta^2}+ N\log\prn*{\frac{\reps}{\delta}}}.
    \end{align}
    \paragraph{Minimizing $\Lpenalized$: correctness.}
    For any $q \in \Delta^N$ let $\mc{L}_q(x) \defeq{\sum_{i \in[N]} q_i \ell_i(x)-\psieps(q_i)}$ and note that $\mc{L}_q$ is $G$-Lipschitz,
    since  for all $x \in \xset$ we have $\norm{\nabla \mc{L}_q(x)} = \norm*{\sum_{i \in [N]} q_i \nabla \ell_i(x)} \le G$ . Maximum operations preserve the Lipschitz continuity 
    and therefore $\Leps(x) = \max_{q \in \Delta^N} \mc{L}_q(x)$ is also $G$-Lipschitz. 
    Since $\Leps$ is $G$-Lipschitz, we can use \Cref{prop:MainProp} with $F=\Leps$
    and obtain that  with probability at least $\half$ \Cref{alg:acceleratedProxPoint} outputs a point $x$ such that  $ \Leps(x) - \min_{x}\Leps(x) \le \eps/2$. 
    In addition, from \Cref{lem:SMapproxMaxFdiv} we have 
    \[\Lpenalized(x)- \min_{x}\Lpenalized(x)  \le  \Leps(x) - \min_{x}\Leps(x) +\frac{\eps}{2} \le \eps. \] 
    \paragraph{Minimizing $\Lpenalized$: complexity.}
    We apply \Cref{prop:MainProp} with $F= \upsilonepsreg$ and $r_\epsilon = \epsilon /(2\log N \cdot G)$, thus 
    the complexity of finding an $\eps/2$-suboptimal minimizer of $\Leps$ (and therefore an $\eps$-suboptimal minimizer of $\Lpenalized$) is bounded as: 
    \begin{equation*}
    	O\prn*{
    		\prn*{\frac{R}{\reps}}^{2/3} \brk*{
    			 \prn*{\sum_{j=0}^{\meps}\frac{1}{2^j}\broocost[\frac{\reps}{ 2^{j/2}\meps^2}][\lmin]}\meps 
    			+ \prn*{ \broocost[\reps][\lmin] + N }\meps^3, 
    		}
    	}
    \end{equation*}  
    where $\meps = O\prn*{\log \frac{GR^2}{\epsilon \reps}}= O\prn*{\log \frac{GR}{\epsilon }\log N}$ and $\lmin = O\prn*{\frac{\meps^2 \epsilon}{\reps^{4/3}R^{2/3}}}$.
    To obtain the total complexity bound,  we evaluate the complexity of running $\reps$-BROO with accuracy $\delta_{j} = \frac{\reps}{ 2^{j/2}\meps^2}$
    and  $\delta = \frac{\reps}{30} $. Using \eqref{eq:DualBROOcomplexity} we get 
    \begin{enumerate}
        \item {$\broocost[\frac{\reps}{ \meps^2 2^{j/2}}][\lmin] = O\prn*{\frac{G^2 2^j \meps^4}{\lmin^2 \reps^2  }+ N\log\prn*{\meps^2 2^{j/2}}} = O\prn*{\frac{\prn*{\frac{GR}{\eps}}^{4/3}}{\prn*{\log N}^{2/3}}2^j+N\prn*{\meps + \log\prn*{2^{j/2}}}}$}
        \item {$\broocost[\frac{\reps}{30}][\lmin]=O\prn*{\frac{G^2}{\lmin^2\reps^2}+N} = O\prn*{\frac{\prn*{\frac{GR}{\eps }}^{4/3}}{\meps^4 \prn*{\log N}^{2/3} }+N}.$}
    \end{enumerate}

    Thus 
    \begin{flalign*}
        O\prn*{\meps\sum_{j=0}^{\meps}\frac{1}{2^j}\broocost[\frac{\reps}{ 2^{j/2}\meps^2}][\lmin]} 
        \le   O\prn*{ 
            \meps^2 \frac{\prn*{\frac{GR}{\eps}}^{4/3}}{\prn*{\log N}^{2/3}} + \meps^2 N
            }
    \end{flalign*} 
    and 
    \begin{flalign*}
        O\prn*{\prn*{ \broocost[\reps][\lmin] + N }\meps^3} 
        \le O\prn*{\frac{\prn*{\frac{GR}{\eps }}^{4/3}}{\meps \prn*{\log N}^{2/3} }+N \meps^3}.
    \end{flalign*}
    Substituting the bounds into \Cref{prop:MainProp}, the total complexity is bounded as    
    \begin{flalign*}
      &  O\prn*{\prn*{\frac{R}{\reps}}^{2/3} \brk*{ N \meps^3 +  \frac{\meps^2\prn*{\frac{GR}{\eps}}^{4/3}}{\prn*{\log N}^{2/3}} }}  
    \le O\prn*{ N\prn*{\frac{GR}{\eps}}^{2/3}\log^{11/3}\prn*{\frac{GR}{\eps}\log N} + \prn*{\frac{GR}{\eps}}^{2}  \log^{2}\prn*{\frac{GR}{\eps}\log N}}
\end{flalign*} 
\end{proof}
\subsection{Accelerated variance reduction BROO implementation}\label{ssec:VarianceReductionSVRGproofs}
In this section we prove the complexity guarantees of our BROO implementation for (potentially only slightly) smooth losses.
We first prove \Cref{lem:smoothnessOfGamma}, showing that $\upsilonepsreg$ (the approximation of our objective \eqref{eq:penalizedFdiver}) is a finite sum of smooth functions, and thus for the BROO implementation
we can use a variance reduction method for a finite (weighted) sums. Then, we give \Cref{def:ValidVR} of a ``valid accelerated variance reduction'' (VR) method, 
and in \Cref{lem:VARGguarantee} we prove the convergence rate of our BROO implementation (\Cref{alg:DualKatyusha}) which is simply a restart scheme utilizing any valid accelerated VR method. 
 We then combine the guarantees of \Cref{lem:VARGguarantee} and \Cref{prop:MainProp} to prove \Cref{thm:DualSVRGBROOcomplexity}, our final  complexity guarantee in the smooth.

To begin, let us restate here the definition of $\upsilonepsreg$ (that has the form of a weighted finite sum):
\begin{equation}\label{eq:upsilonepsreg-def-repeat}
    \upsilonepsreg(x, y) = \sum_{i \in [N]} \bp_i \upsilon_i(x, y)  
    \text{~~where~~}
    \upsilon_i(x, y) \defeq \frac{\psieps^*\prn*{\ell_i(x) - G y}}{\bp_i} +  Gy + \frac{\lambda}{2}\norm{x-\bx}^2
\end{equation}
and $\bp_i = {\psieps^*}'(\ell_i(\bx)-G\by)$.

\smoothnessOfGamma*
\begin{proof}
   To show the Lipschitz property we compute the gradient of  $\upsilon_i(x,y)$ with respect to $x$ and $y$.
    \begin{flalign*}
        \nabla_x \upsilon_i(x,y) =  \frac{{\psieps^{*}}'\prn*{\ell_i(x) - G y}}{\bp_i} \nabla \ell_i(x) + \lambda(x-\bx)
    \end{flalign*}
    \begin{flalign*}
        \nabla_y \upsilon_i(x,y) = G -G \frac{{\psieps^{*}}'\prn*{\ell_i(x) - G y}}{\bp_i} 
    \end{flalign*}
    Similarly to the proof of \Cref{lem:dualGradEstProperties}, we have that $ \frac{{\psieps^{*}}'\prn*{\ell_i(x) - G y}}{\bp_i}  \le e^2$ and therefore 
    $\norm*{ \nabla_x \upsilon_i(x,y)} \le e^2 G  = O(G)$
    and $\abs*{ \nabla_y \upsilon_i(x,y)} \le e^2G = O(G)$, giving the first statement. 
    To bound the smoothness of $\upsilon_i$, we compute the Hessian of $\upsilon_i(x,y)$.
\begin{flalign*}
    \nabla^2_x \upsilon_i(x,y) &=  \frac{{\psieps^{*}}'\prn*{\ell_i(x) - G y}}{\bp_i} \nabla^2 \ell_i(x) +\frac{{\psieps^*}''\prn*{\ell_i(x) - G y}}{\bp_i} \nabla \ell_i(x)\nabla \ell_i(x)^T+ \lambda I 
\end{flalign*}
\begin{flalign*}
    \nabla^2_y \upsilon_i(x,y) =  G^2 \frac{{\psieps^*}''\prn*{\ell_i(x) - G y}}{\bp_i} 
\end{flalign*}
\begin{flalign*}
    \nabla_{xy} \upsilon_i(x,y)  -G \frac{{\psieps^*}''\prn*{\ell_i(x) - G y}}{\bp_i} \nabla \ell_i(x).
\end{flalign*}
\Cref{lem:psiProperties} implies that, for all $v$,  $
\frac{{\psieps^{*}}''(v)}{{\psieps^{*}}'(v)}
	 = \prn*{\log {\psieps^{*}}'(v)}' 
	 \le \frac{1}{\eps'}$ 
and $\frac{{\psieps^{*}}'\prn*{\ell_i(x) - G y}}{\bp_i} \le e^2$. In addition
note that each $\ell_i$ is $L$-smooth and $G$-Lipschitz and $\lambda = O\prn*{\frac{G}{\reps}}=O\prn*{\frac{G^2}{\eps'}}$, therefore
\begin{flalign*}
   \norm{\nabla^2_x \upsilon_i(x,y)}_\textup{op} & =
   \norm*{\frac{{\psieps^{*}}'\prn*{\ell_i(x) - G y}}{\bp_i} \nabla^2 \ell_i(x) +
   \frac{{\psieps^*}''\prn*{\ell_i(x) - G y}}{\psieps^{*'}\prn*{\ell_i(x) - G y}}  \frac{\psieps^{*'}\prn*{\ell_i(x) - G y}}{\bp_i}\nabla \ell_i(x)\nabla \ell_i(x)^T
   + \lambda I }_\textup{op} \\ &  \le 
 O\prn*{e^2\prn*{L+ 2\frac{G^2}{\eps'}}}
\end{flalign*}
\begin{flalign*}
    \norm{\nabla_{xy} \upsilon_i(x,y)} =\norm*{G \frac{{\psieps^*}''\prn*{\ell_i(x) - G y}}{\psieps^{*'}\prn*{\ell_i(x) - G y}}  \frac{\psieps^{*'}\prn*{\ell_i(x) - G y}}{\bp_i} \nabla \ell_i(x)}\le e^2 \frac{G^2}{\eps'}
\end{flalign*}
\begin{flalign*}
   \nabla^2_y \upsilon_i(x,y)= G^2 \frac{{\psieps^*}''\prn*{\ell_i(x) - G y}}{\psieps^{*'}\prn*{\ell_i(x) - G y}}  \frac{\psieps^{*'}\prn*{\ell_i(x) - G y}}{\bp_i} \nabla \ell_i(x) \le  e^2 \frac{G^2}{\eps'}.
\end{flalign*}
Applying \Cref{lem:OpNormBound} with $h=\upsilon_i$ we get that
\begin{align*}
    \norm*{\nabla^2 \upsilon_i(x,y)}_\textup{op} \le e^2\prn*{ L + 2 \frac{G^2}{\eps'}},
\end{align*}
proving that each $\upsilon_i(x,y)$ is $ O\prn*{ L +  \frac{G^2}{\eps'}}$-smooth.
\end{proof}

\begin{algorithm2e}[t]
	\DontPrintSemicolon
	\caption{Restarting Accelerated Variance Reduction with Optimal Dual Values}	\label{alg:DualKatyusha}
	\KwInput{The function  $\upsilonepsreg(x, y) = \sum_{i \in [N]} \bp_i \upsilon_i(x, y)$  defined in \eqref{eq:upsilonepsreg-def-repeat}, number of total restarts $K$, and an algorithm $\varianceReduction$ that takes in $x,y\in \xset\times \R$ and complexity budget $T$, and outputs $x',y'\in \xset\times \R$.
	}
	$x_0, y_0 = \bx, \by = \bx, \argmin_{y\in\R} \upsiloneps(\bx, y)$ \; 
	\For{$k = 1, \ldots, K$ }{
		$x_k, y_k' = \varianceReduction(x_{k-1}, y_{k-1};T)$\;
		$y_k = \argmin_{y\in\R} \upsiloneps(x_k, y)$ \; 
	}
	\Return $x_K$
\end{algorithm2e}

\newcommand{\wtL}{\widetilde{L}}

\begin{definition}\label{def:ValidVR}
	For a given ball center $\bx\in\xset$, 
	let $z_\star \in \ball_{\reps}(\bx)\times\R$ minimize the function  $\upsilonepsreg:\ball_{\reps}(\bx)\times\R\to \R$, and let $\by=\argmin_{y\in\R}\upsilonepsreg(\bx,y)$. Let  $\varianceReduction$ be a procedure that takes in $z\in \xset\times \R$ and complexity budget $T$, and outputs $z'\in \xset\times \R$. We say that $\varianceReduction$ is \emph{a valid accelerated VR method} if it has complexity $T$ and satisfies the following: there a constant $C$ such that for any $\alpha$, input $z$, and 
	\begin{equation*}
		T\ge C\prn*{N \frac{\upsilonepsreg(z)-\upsilonepsreg(z_\star)}{\alpha}
		+ \sqrt{\frac{\wtL \norm{z-z_\star}^2}{\alpha}} }
	~~\mbox{for}~~\wtL = L + \frac{G^2}{\epsilon'},
	\end{equation*}
	the output $z'$ of $\varianceReduction(z;T)$ satisfies 
	\begin{equation*}
		\E\upsilonepsreg(z') -\upsilonepsreg(z_\star) \le \alpha.
	\end{equation*} 
\end{definition}

\begin{lem}\label{lem:katyusha-is-valid}
	$\textup{Katyusha}^{\textsf{sf}}$~\cite{lan2019unified} is a valid accelerated VR method for some $C=O(1)$.
\end{lem}
\begin{proof}
	Immediate from \cite[Theorem 4.1]{allen2017katyusha} and \Cref{lem:smoothnessOfGamma}. (We note that the theorem is stated for a finite sum with uniform weights, but the extension of the theorem and the method to non-uniform sampling is standard).
\end{proof}

The following lemma shows that \Cref{alg:DualKatyusha}, when coupled with any valid accelerated VR method yields a BROO implementation for 
\begin{equation*}
	\Leps = \max_{q\in\Delta^N} \sum_{i\in[N]} \prn*{q_i \ell_i(x) - \psieps(q_i)} = \min_{y\in\R} \upsiloneps(x,y),
\end{equation*}
i.e., it outputs an approximate minimizer of
\begin{equation*}
	\Leps[\psi,\epsilon,\lambda](x) \defeq \Leps(x) + \frac{\lambda}{2}\norm{x-\bx}^2  =  \min_{y\in\R} \upsilonepsreg(x,y)
\end{equation*}
in $\ball_{\reps}(\bx)$. 

\newcommand{\Lepsreg}{\Leps[\psi,\epsilon,\lambda]}

\begin{lem}\label{lem:VARGguarantee}
	Let \Cref{assumption:GlobalAssumption,assumption:SmoothFunctions} hold, 
	and suppose 
	\Cref{alg:DualKatyusha} uses a valid accelerated VR method with constant $C$ (defined above). Then, for $\wt{L}=L+G^2/\epsilon'$ and $T \ge 2C(N + \sqrt{N\wt{L}/\lambda})$, for any $K\ge 0$ the output $x$ of \Cref{alg:DualKatyusha} satisfies
	\begin{equation*}
		\E \Lepsreg(x) - \min_{\xopt\in \ball_{\reps}(\bx)} 	\E \Lepsreg(\xopt)
		\le 2^{-K} \prn*{\Lepsreg(\bx) - \min_{\xopt\in \ball_{\reps}(\bx)} 	\E \Lepsreg(\xopt)}.
	\end{equation*}
	Moreover, the complexity of \Cref{alg:DualKatyusha} is $K(N+T)=O\prn*{K(N + \sqrt{N\wtL/\lambda})}$
\end{lem}
\begin{proof}
	Let $z_\star = (x_\star,y_\star)$ minimize $\upsilonepsreg$ in $\ball_{\reps}(\bx)\times\R$, so that $\xopt=\argmin_{x\in \ball_{\reps}(\bx)}\Leps[\psi,\epsilon,\lambda](x)$ as well.
	Note that for all of the outer loop iterations $z_k = (x_k, y_k)$ we have, by the optimality of $y_k$ and \Cref{lem:BoundEtaDistance},
	\begin{equation*}
		\norm{z_k - z_\star}^2 = \norm{x_k-x_\star}^2 + (y_k - y_\star)^2
		\le 2\norm{x_k-x_\star}^2.
	\end{equation*}
	Moreover, the $\lambda$-strong convexity of $\Lepsreg$ implies that
	\begin{equation*}
		\norm{x_k-x_\star}^2 \le \frac{2 ( \Lepsreg(x_k) -\Lepsreg(x_\star))}{\lambda}.
	\end{equation*}
	Furthermore, note that $\upsilonepsreg(z_k) = \min_{y\in\R} \upsilonepsreg(x_k, y) = \Lepsreg(x_k)$. Recalling that by~\Cref{lem:BoundEtaDistance} restricting the domain of $y$ to $[\by-\reps,\by+\reps]$ does not change the optimal $y$, we conclude that for $\varianceReduction$ to guarantee $\upsilonepsreg(x_{k+1},y_{k+1}')-\upsilonepsreg(z_\star) \le \alpha$ it suffices to choose
	\begin{equation*}
		T\ge C\prn*{N \frac{\Lepsreg(x_k) -\Lepsreg(x_\star)}{\alpha}
			+ \sqrt{\frac{4\wtL( \Lepsreg(x_k) -\Lepsreg(x_\star) }{\lambda \alpha}} }
	\end{equation*}
	In particular, we see that $T \ge 2C\prn*{N + \sqrt{N\wtL/\lambda}}$ suffices for $\alpha = \frac{\Lepsreg(x_k) -\Lepsreg(x_\star)}{2}$, from which we conclude that
	\begin{flalign*}
		\Lepsreg(x_{k+1}) - \Lepsreg(x_\star) & = 
		\upsilonepsreg(x_{k+1},y_{k+1})-\upsilonepsreg(z_\star) \\& \le  \upsilonepsreg(x_{k+1},y_{k+1}')-\upsilonepsreg(z_\star) \le \frac{\Lepsreg(x_k) -\Lepsreg(x_\star)}{2},
	\end{flalign*}
	giving the claimed optimality bound. Finally, the complexity of the method is clearly $K(T+N)$ since we make $K$ calls to $\varianceReduction$ with complexity $T$ and $K$ exact minimizations over $y$ with complexity $N$.
\end{proof}

\DualSVRGBROOcomplexity*
\begin{proof}
	We first show correctness and complexity of the BROO implementation and then argue the same points for the overall method.
	
    \paragraph{BROO implementation: correctness and complexity.}
    Combining \Cref{lem:VARGguarantee,lem:katyusha-is-valid} and setting $K=\log_2 \frac{\upsiloneps(x_0,y_0)-\argmin_{x \in \ball_{\reps}(\bx), y\ \in \R}\upsiloneps(x, y)}{\lambda \delta^2}$, we obain a valid BROO implementation with complexity 
\begin{align}\label{dualSVRGBROOguarantee}
    O\prn*{
        \prn*{N+\frac{\sqrt{N}\prn*{\sqrt{L\eps'}+G}}{\sqrt{\eps'\lambda}}} 
        \log\prn*{\frac{\upsiloneps(x_0,y_0)-\argmin_{x \in \ball_{\reps}(\bx), y \in \R}\upsiloneps(x, y)}{\lambda \delta^2}}}
\end{align}
    Furthermore, we note that $\upsiloneps$ is $O\prn*{G}$-Lipschitz and therefore $\upsiloneps(x_0,y_0)-\argmin_{x \in \ball_{\reps}(x), y \in \R}\upsiloneps(x, y) \le O\prn*{G\reps}$. 

    \paragraph{Minimizing $\Lpenalized$: correctness.}
    Similarly to the proof of \Cref{thm:BROOcomplexityDualProblem}, we note that $\Leps$ is $G$-Lipschitz, 
    and therefore we can apply \Cref{prop:MainProp} with $F=\Leps$ and obtain that with probability at least $\half$
    \Cref{alg:acceleratedProxPoint} outputs $x$ such that $\Leps(x) - \min_{x}\Leps(x) \le \eps/2$ and \Cref{lem:SMapproxMaxFdiv} gives
    \[\Lpenalized(x)- \min_{x}\Lpenalized(x)  \le  \Leps(x) - \min_{x}\Leps(x) +\frac{\eps}{2} \le \eps. \] 
    
    \paragraph{Minimizing $\Lpenalized$: complexity.}
    Applying \Cref{prop:MainProp} with $F=\upsiloneps$ and $\reps=\frac{\eps}{2 G \log N}$, the complexity of finding an $\eps$-suboptimal minimizer of $\Lpenalized$ is 
    \begin{equation}
    	O\prn*{
    		\prn*{\frac{R}{\reps}}^{2/3} \brk*{
    			 \prn*{\sum_{j=0}^{\meps}\frac{1}{2^j}\broocost[\frac{\reps}{ 2^{j/2}\meps^2}][\lmin]}\meps 
    			+
    			\prn*{ \broocost[\reps][\lmin] + N }\meps^3
    		}
    	}
    \end{equation} 
    where $\meps = O\prn[\big]{\log \frac{GR^2}{\epsilon \reps}}=\log\prn*{\frac{GR}{\eps}\log N}$ and $\lmin = O\prn[\big]{\frac{\meps^2 \epsilon}{r^{4/3}R^{2/3}}}$.   
    Using similar calculations to the proof of \Cref{thm:svrgBrooComplex} we obtain that 
    \begin{enumerate}
        \item {$\broocost[\frac{\reps}{ 2^{j/2}\meps^2}][\lmin] = O\prn*{\prn*{N+N^{1/2} \prn*{\frac{G \sqrt{\log N}}{\sqrt{\eps}}+ \sqrt{L}} \frac{1}{\sqrt{\lmin}}}\log \prn*{\frac{ \eps' 2^{j}\meps^4}{\lmin \reps^2}}}$}
        \item { $ \broocost[\frac{\reps}{30}][\lmin] = O\prn*{\prn*{N+N^{1/2} \prn*{\frac{G \sqrt{\log N}}{\sqrt{\eps}}+ \sqrt{L}} \frac{1}{\sqrt{\lmin}}} \log \prn*{ \frac{ \eps'}{\lmin \reps^2}  }}
         $ }
    \end{enumerate}
    with $\frac{ \eps'}{\lmin \reps^2} = O\prn*{\prn*{\frac{GR}{\eps}}^{2/3}\frac{1}{\meps^2}}$ and
    $\frac{1}{\sqrt{\lmin}} = O\prn*{\frac{R^{1/3}\reps^{2/3}}{\meps \sqrt{\eps}}} $.  
   Therefore, 
   \begin{align*}
    \meps\sum_{j=0}^{\infty}\frac{1}{2^j}\broocost[\frac{\reps}{ 2^{j/2}\meps^2}][\lmin] \le O\prn*{\meps^2 \brk*{ N +N^{1/2}\prn*{\frac{GR^{1/3}}{\eps}+\sqrt{\frac{LR^{2/3}}{\eps}}}\reps^{2/3}}}.
\end{align*}
and
\begin{align*}
    O\prn*{ \prn*{ \broocost[\reps][\lmin] + N }\meps^3 }  \le O\prn*{\meps^4N + \meps^{3.5} N^{1/2} \prn*{
            \frac{GR^{1/3}}{\eps} + \sqrt{\frac{LR^{2/3}}{\eps}} 
            }\reps^{2/3}
        }.
\end{align*}
Substituting the bounds into \Cref{prop:MainProp} with $\meps=\log \prn*{\frac{GR}{\eps}\log N}$ and $\reps=\frac{\eps}{2 \log N}$ the total complexity is 
\[ O\prn*{
    N\prn*{ \frac{GR}{\eps}}^{2/3} \log^{14/3}\prn*{\frac{GR}{\eps}\log N} +
    N^{1/2} \prn*{
    \frac{GR}{\eps} + \sqrt{\frac{LR^2}{\eps}} 
    }  \log^{7/2}\prn*{\frac{GR}{\eps}\log N}}. \] 

\end{proof}

\subsection{Helper lemmas}\label{ssec:DualDROHelperLemmas}
\begin{lem}\label{lem:OpNormBound}
    Let $x \in \R^d$, $y \in \R$, then for any $h: \R^d \times \R \rightarrow \R$ we have 
    \begin{align*}
        \norm*{\nabla^2 h(x, y)}_\textup{op} \le \max\crl*{\nabla_y^2h(x, y) + \norm*{\nabla_{xy}h(x, y)}, \norm*{\nabla^2_x h(x, y)}_\textup{op}+ \norm*{\nabla_{xy} h(x, y)}}
    \end{align*}
\end{lem}
\begin{proof}
    \begin{flalign*}
        \norm*{\nabla^2 h(x, y)}_\text{op} &= \sup_{\norm*{v}^2 +u^2 = 1 }
        v^T \nabla_x^2 h(x, y)v + 2u\prn*{\nabla_{xy}h(x, y)}^T v+ u^2 \nabla_y^2 h(x, y)
        \\ & \stackrel{(\romannumeral1)}{\le} \norm*{v}^2 \norm*{\nabla_x^2 h(x, y)}_\text{op} 
        +2u \norm*{\nabla_{xy}h(x, y)}\norm*{v} + u^2 \nabla_y^2 h(x, y)
        \\ & \stackrel{(\romannumeral2)}{\le} \norm*{v}^2 \norm*{\nabla_x^2 h(x, y)}_\text{op}
        +\prn*{u^2  +\norm*{v}^2 }\norm*{\nabla_{xy}h(x, y)}
        + u^2 \nabla_y^2 h(x, y)
        \\ & = u^2 \prn*{ \nabla_y^2 h(x, y) + \norm*{\nabla_{xy}h(x, y)}} + \prn*{1-u^2}
        \prn*{  \norm*{\nabla_x^2 h(x, y)}_\text{op} + 
        \norm*{\nabla_{xy}h(x, y)}}
        \\ &  \le \max\crl*{\nabla_y^2h(x, y) + 
        \norm*{\nabla_{xy} h(x, y)}, \norm*{\nabla^2_x h(x, y)}_\text{op}
        + \norm*{\nabla_{xy} h(x, y)}}
    \end{flalign*}
    with $(\romannumeral1)$ following due to Hölder's inequality and
    $(\romannumeral2)$ follows from the inequality $2ab \le a^2 + b^2$.
\end{proof} 


\begin{thebibliography}{55}
\providecommand{\natexlab}[1]{#1}
\providecommand{\url}[1]{\texttt{#1}}
\expandafter\ifx\csname urlstyle\endcsname\relax
  \providecommand{\doi}[1]{doi: #1}\else
  \providecommand{\doi}{doi: \begingroup \urlstyle{rm}\Url}\fi

\bibitem[Allen-Zhu(2017)]{allen2017katyusha}
Z.~Allen-Zhu.
\newblock Katyusha: The first direct acceleration of stochastic gradient
  methods.
\newblock \emph{Journal of Machine Learning Research}, 18\penalty0
  (221):\penalty0 1--51, 2017.

\bibitem[Allen-Zhu(2018)]{allen2018katyusha}
Z.~Allen-Zhu.
\newblock Katyusha {X}: Practical momentum method for stochastic
  sum-of-nonconvex optimization.
\newblock In \emph{International Conference on Machine Learning (ICML)}, 2018.

\bibitem[Asi et~al.(2021)Asi, Carmon, Jambulapati, Jin, and
  Sidford]{asi2021stochastic}
H.~Asi, Y.~Carmon, A.~Jambulapati, Y.~Jin, and A.~Sidford.
\newblock Stochastic bias-reduced gradient methods.
\newblock In \emph{Advances in Neural Information Processing Systems
  (NeurIPS)}, 2021.

\bibitem[Ben-Tal et~al.(2013)Ben-Tal, den Hertog, De~Waegenaere, Melenberg, and
  Rennen]{ben2013robust}
A.~Ben-Tal, D.~den Hertog, A.~De~Waegenaere, B.~Melenberg, and G.~Rennen.
\newblock Robust solutions of optimization problems affected by uncertain
  probabilities.
\newblock \emph{Management Science}, 59\penalty0 (2):\penalty0 341--357, 2013.

\bibitem[Blanchet and Glynn(2015)]{blanchet2015unbiased}
J.~H. Blanchet and P.~W. Glynn.
\newblock Unbiased {M}onte {C}arlo for optimization and functions of
  expectations via multi-level randomization.
\newblock In \emph{2015 Winter Simulation Conference (WSC)}, 2015.

\bibitem[Blanchet et~al.(2019)Blanchet, Glynn, and Pei]{blanchet2019unbiased}
J.~H. Blanchet, P.~W. Glynn, and Y.~Pei.
\newblock Unbiased multilevel {M}onte {C}arlo: Stochastic optimization,
  steady-state simulation, quantiles, and other applications.
\newblock \emph{arXiv:1904.09929}, 2019.

\bibitem[Buolamwini and Gebru(2018)]{buolamwini2018gender}
J.~Buolamwini and T.~Gebru.
\newblock Gender shades: Intersectional accuracy disparities in commercial
  gender classification.
\newblock In \emph{Conference on fairness, accountability and transparency},
  pages 77--91, 2018.

\bibitem[Carmon et~al.(2017)Carmon, Duchi, Hinder, and
  Sidford]{carmon2017convex}
Y.~Carmon, J.~C. Duchi, O.~Hinder, and A.~Sidford.
\newblock “convex until proven guilty”: Dimension-free acceleration of
  gradient descent on non-convex functions.
\newblock In \emph{International Conference on Machine Learning (ICML)}, 2017.

\bibitem[Carmon et~al.(2020{\natexlab{a}})Carmon, Jambulapati, Jiang, Jin, Lee,
  Sidford, and Tian]{carmon2020acceleration}
Y.~Carmon, A.~Jambulapati, Q.~Jiang, Y.~Jin, Y.~T. Lee, A.~Sidford, and
  K.~Tian.
\newblock Acceleration with a ball optimization oracle.
\newblock In \emph{Advances in Neural Information Processing Systems
  (NeurIPS)}, 2020{\natexlab{a}}.

\bibitem[Carmon et~al.(2020{\natexlab{b}})Carmon, Jin, Sidford, and
  Tian]{carmon2020coordinate}
Y.~Carmon, Y.~Jin, A.~Sidford, and K.~Tian.
\newblock Coordinate methods for matrix games.
\newblock In \emph{Foundations of Computer Science (FOCS)}, 2020{\natexlab{b}}.

\bibitem[Carmon et~al.(2021)Carmon, Jambulapati, Jin, and
  Sidford]{carmon2021thinking}
Y.~Carmon, A.~Jambulapati, Y.~Jin, and A.~Sidford.
\newblock Thinking inside the ball: Near-optimal minimization of the maximal
  loss.
\newblock In \emph{Conference on Learning Theory (COLT)}, 2021.

\bibitem[Chow et~al.(2015)Chow, Tamar, Mannor, and Pavone]{chow2015risk}
Y.~Chow, A.~Tamar, S.~Mannor, and M.~Pavone.
\newblock Risk-sensitive and robust decision-making: a {CVaR} optimization
  approach.
\newblock In \emph{Advances in neural information processing systems
  (NeurIPS)}, 2015.

\bibitem[Cohen et~al.(2016)Cohen, Lee, Miller, Pachocki, and
  Sidford]{cohen2016geometric}
M.~B. Cohen, Y.~T. Lee, G.~Miller, J.~Pachocki, and A.~Sidford.
\newblock Geometric median in nearly linear time.
\newblock In \emph{ACM symposium on Theory of Computing}, 2016.

\bibitem[Curi et~al.(2020)Curi, Levy, Jegelka, and Krause]{curi2020adaptive}
S.~Curi, K.~Y. Levy, S.~Jegelka, and A.~Krause.
\newblock Adaptive sampling for stochastic risk-averse learning.
\newblock In \emph{Advances in Neural Information Processing Systems
  (NeurIPS)}, 2020.

\bibitem[Duchi and Namkoong(2019)]{duchi2019variance}
J.~Duchi and H.~Namkoong.
\newblock Variance-based regularization with convex objectives.
\newblock \emph{Journal of Machine Learning Research}, 20\penalty0
  (68):\penalty0 1--55, 2019.

\bibitem[Duchi and Namkoong(2021)]{duchi2021learning}
J.~C. Duchi and H.~Namkoong.
\newblock Learning models with uniform performance via distributionally robust
  optimization.
\newblock \emph{The Annals of Statistics}, 49\penalty0 (3):\penalty0
  1378--1406, 2021.

\bibitem[Esfahani and Kuhn(2018)]{esfahani2018data}
P.~M. Esfahani and D.~Kuhn.
\newblock Data-driven distributionally robust optimization using the
  wasserstein metric: Performance guarantees and tractable reformulations.
\newblock \emph{Mathematical Programming, Series A}, 171\penalty0 (1):\penalty0
  115--166, 2018.

\bibitem[Gao and Kleywegt(2016)]{gao2016distributionally}
R.~Gao and A.~J. Kleywegt.
\newblock Distributionally robust stochastic optimization with {W}asserstein
  distance.
\newblock \emph{arXiv:1604.02199}, 2016.

\bibitem[Gasnikov et~al.(2019)Gasnikov, Dvurechensky, Gorbunov, Vorontsova,
  Selikhanovych, Uribe, Jiang, Wang, Zhang, Bubeck, Jiang, Lee, Li, and
  Sidford]{gasnikov19near}
A.~Gasnikov, P.~Dvurechensky, E.~Gorbunov, E.~Vorontsova, D.~Selikhanovych,
  C.~A. Uribe, B.~Jiang, H.~Wang, S.~Zhang, S.~Bubeck, Q.~Jiang, Y.~T. Lee,
  Y.~Li, and A.~Sidford.
\newblock Near optimal methods for minimizing convex functions with {L}ipschitz
  $p$-th derivatives.
\newblock In \emph{Conference on Learning Theory (COLT)}, 2019.

\bibitem[Giles(2008)]{giles2008multilevel}
M.~B. Giles.
\newblock Multilevel monte carlo path simulation.
\newblock \emph{Operations research}, 56\penalty0 (3):\penalty0 607--617, 2008.

\bibitem[Hashimoto et~al.(2018)Hashimoto, Srivastava, Namkoong, and
  Liang]{hashimoto2018fairness}
T.~Hashimoto, M.~Srivastava, H.~Namkoong, and P.~Liang.
\newblock Fairness without demographics in repeated loss minimization.
\newblock In \emph{International Conference on Machine Learning}, 2018.

\bibitem[Hazan and Kale(2014)]{hazan2014beyond}
E.~Hazan and S.~Kale.
\newblock Beyond the regret minimization barrier: Optimal algorithms for
  stochastic strongly-convex optimization.
\newblock \emph{Journal of Machine Learning Research}, 15\penalty0
  (71):\penalty0 2489--2512, 2014.

\bibitem[Heinrich(2001)]{heinrich2001multilevel}
S.~Heinrich.
\newblock Multilevel {M}onte {C}arlo methods.
\newblock In \emph{International Conference on Large-Scale Scientific
  Computing}, 2001.

\bibitem[Hovy and S{\o}gaard(2015)]{hovy2015tagging}
D.~Hovy and A.~S{\o}gaard.
\newblock Tagging performance correlates with author age.
\newblock In \emph{Association for Computational Linguistics (ACL)}, pages
  483--488, 2015.

\bibitem[Hu et~al.(2018)Hu, Niu, Sato, and Sugiyama]{hu2018does}
W.~Hu, G.~Niu, I.~Sato, and M.~Sugiyama.
\newblock Does distributionally robust supervised learning give robust
  classifiers?
\newblock In \emph{International Conference on Machine Learning (ICML)}, 2018.

\bibitem[Jin et~al.(2021)Jin, Zhang, Wang, and Wang]{jin2021non}
J.~Jin, B.~Zhang, H.~Wang, and L.~Wang.
\newblock Non-convex distributionally robust optimization: Non-asymptotic
  analysis.
\newblock In \emph{Advances in Neural Information Processing Systems
  (NeurIPS)}, 2021.

\bibitem[Johnson and Zhang(2013)]{johnson2013accelerating}
R.~Johnson and T.~Zhang.
\newblock Accelerating stochastic gradient descent using predictive variance
  reduction.
\newblock In \emph{Advances in neural information processing systems
  (NeurIPS)}, 2013.

\bibitem[Kent et~al.(2021)Kent, Blanchet, and Glynn]{kent2021frank}
C.~Kent, J.~Blanchet, and P.~Glynn.
\newblock Frank-wolfe methods in probability space.
\newblock \emph{arXiv:2105.05352}, 2021.

\bibitem[Koh et~al.(2021)Koh, Sagawa, Marklund, Xie, Zhang, Balsubramani, Hu,
  Yasunaga, Phillips, Gao, Lee, David, Stavness, Guo, Earnshaw, Haque, Beery,
  Leskovec, Kundaje, Pierson, Levine, Finn, and Liang]{koh2021wilds}
P.~W. Koh, S.~Sagawa, H.~Marklund, S.~M. Xie, M.~Zhang, A.~Balsubramani, W.~Hu,
  M.~Yasunaga, R.~L. Phillips, I.~Gao, T.~Lee, E.~David, I.~Stavness, W.~Guo,
  B.~A. Earnshaw, I.~S. Haque, S.~Beery, J.~Leskovec, A.~Kundaje, E.~Pierson,
  S.~Levine, C.~Finn, and P.~Liang.
\newblock Wilds: A benchmark of in-the-wild distribution shifts.
\newblock In \emph{International Conference on Machine Learning (ICML)}, 2021.

\bibitem[Krokhmal et~al.(2002)Krokhmal, Palmquist, and
  Uryasev]{krokhmal2002portfolio}
P.~Krokhmal, J.~Palmquist, and S.~Uryasev.
\newblock Portfolio optimization with conditional value-at-risk objective and
  constraints.
\newblock \emph{Journal of risk}, 4:\penalty0 43--68, 2002.

\bibitem[Lan et~al.(2019)Lan, Li, and Zhou]{lan2019unified}
G.~Lan, Z.~Li, and Y.~Zhou.
\newblock A unified variance-reduced accelerated gradient method for convex
  optimization.
\newblock \emph{Advances in Neural Information Processing Systems (NeurIPS)},
  32, 2019.

\bibitem[Levy et~al.(2020)Levy, Carmon, Duchi, and Sidford]{levy2020large}
D.~Levy, Y.~Carmon, J.~C. Duchi, and A.~Sidford.
\newblock Large-scale methods for distributionally robust optimization.
\newblock In \emph{Advances in Neural Information Processing Systems
  (NeurIPS)}, 2020.

\bibitem[Lin et~al.(2022)Lin, Fang, and Gao]{lin2022distributionally}
F.~Lin, X.~Fang, and Z.~Gao.
\newblock Distributionally robust optimization: a review on theory and
  applications.
\newblock \emph{Numerical Algebra, Control \& Optimization}, 12\penalty0
  (1):\penalty0 159, 2022.

\bibitem[Monteiro and Svaiter(2013)]{monteiro2013accelerated}
R.~D. Monteiro and B.~F. Svaiter.
\newblock An accelerated hybrid proximal extragradient method for convex
  optimization and its implications to second-order methods.
\newblock \emph{SIAM Journal on Optimization}, 23\penalty0 (2):\penalty0
  1092--1125, 2013.

\bibitem[Namkoong and Duchi(2016)]{namkoong2016stochastic}
H.~Namkoong and J.~C. Duchi.
\newblock Stochastic gradient methods for distributionally robust optimization
  with f-divergences.
\newblock In \emph{Advances in Neural Information Processing Systems
  (NeurIPS)}, 2016.

\bibitem[Nemirovski et~al.(2009)Nemirovski, Juditsky, Lan, and
  Shapiro]{nemirovski2009robust}
A.~Nemirovski, A.~Juditsky, G.~Lan, and A.~Shapiro.
\newblock Robust stochastic approximation approach to stochastic programming.
\newblock \emph{SIAM Journal on optimization}, 19\penalty0 (4):\penalty0
  1574--1609, 2009.

\bibitem[Nesterov(2005)]{nesterov2005smooth}
Y.~Nesterov.
\newblock Smooth minimization of non-smooth functions.
\newblock \emph{Mathematical Programming, Series A}, 103\penalty0 (1):\penalty0
  127--152, 2005.

\bibitem[Nesterov(2018)]{nesterov2018lectures}
Y.~Nesterov.
\newblock \emph{Lectures on convex optimization}, volume 137.
\newblock Springer, 2018.

\bibitem[Oakden-Rayner et~al.(2020)Oakden-Rayner, Dunnmon, Carneiro, and
  R{\'e}]{oakden2020hidden}
L.~Oakden-Rayner, J.~Dunnmon, G.~Carneiro, and C.~R{\'e}.
\newblock Hidden stratification causes clinically meaningful failures in
  machine learning for medical imaging.
\newblock In \emph{ACM conference on health, inference, and learning}, 2020.

\bibitem[Oren et~al.(2019)Oren, Sagawa, Hashimoto, and
  Liang]{oren2019distributionally}
Y.~Oren, S.~Sagawa, T.~B. Hashimoto, and P.~Liang.
\newblock Distributionally robust language modeling.
\newblock In \emph{Empirical Methods in Natural Language Processing (EMNLP)},
  2019.

\bibitem[Rockafellar and Uryasev(2000)]{rockafellar2000optimization}
R.~T. Rockafellar and S.~Uryasev.
\newblock Optimization of conditional value-at-risk.
\newblock \emph{Journal of Risk}, 2:\penalty0 21--41, 2000.

\bibitem[Sagawa et~al.(2020)Sagawa, Koh, Hashimoto, and
  Liang]{sagawa2020distributionally}
S.~Sagawa, P.~W. Koh, T.~B. Hashimoto, and P.~Liang.
\newblock Distributionally robust neural networks.
\newblock In \emph{International Conference on Learning Representations
  (ICLR)}, 2020.

\bibitem[Shapiro(2017)]{shapiro2017distributionally}
A.~Shapiro.
\newblock Distributionally robust stochastic programming.
\newblock \emph{SIAM Journal on Optimization}, 27\penalty0 (4):\penalty0
  2258--2275, 2017.

\bibitem[Sinha et~al.(2018)Sinha, Namkoong, and Duchi]{sinha2018certifiable}
A.~Sinha, H.~Namkoong, and J.~Duchi.
\newblock Certifiable distributional robustness with principled adversarial
  training.
\newblock In \emph{International Conference on Learning Representations
  (ICLR)}, 2018.

\bibitem[Song et~al.(2021{\natexlab{a}})Song, Jiang, and Ma]{song2021unified}
C.~Song, Y.~Jiang, and Y.~Ma.
\newblock Unified acceleration of high-order algorithms under {H}\"older
  continuity and uniform convexity.
\newblock \emph{SIAM journal of optimization}, 2021{\natexlab{a}}.

\bibitem[Song et~al.(2021{\natexlab{b}})Song, Lin, Wright, and
  Diakonikolas]{song2021coordinate}
C.~Song, C.~Y. Lin, S.~J. Wright, and J.~Diakonikolas.
\newblock Coordinate linear variance reduction for generalized linear
  programming.
\newblock \emph{arXiv:2111.01842}, 2021{\natexlab{b}}.

\bibitem[Urp{\'\i} et~al.(2021)Urp{\'\i}, Curi, and Krause]{urpi2021riskaverse}
N.~A. Urp{\'\i}, S.~Curi, and A.~Krause.
\newblock Risk-averse offline reinforcement learning.
\newblock In \emph{International Conference on Learning Representations
  (ICLR)}, 2021.

\bibitem[Wang et~al.(2022)Wang, Chatterji, Haque, and Hashimoto]{wang2022is}
K.~A. Wang, N.~S. Chatterji, S.~Haque, and T.~Hashimoto.
\newblock Is importance weighting incompatible with interpolating classifiers?
\newblock In \emph{International Conference on Learning Representations}, 2022.

\bibitem[Wang et~al.(2020)Wang, Guo, Narasimhan, Cotter, Gupta, and
  Jordan]{wang2020robust}
S.~Wang, W.~Guo, H.~Narasimhan, A.~Cotter, M.~R. Gupta, and M.~I. Jordan.
\newblock Robust optimization for fairness with noisy protected groups.
\newblock In \emph{Advances in Neural Information Processing Systems
  (NeurIPS)}, 2020.

\bibitem[Wen et~al.(2014)Wen, Yu, and Greiner]{wen2014robust}
J.~Wen, C.-N. Yu, and R.~Greiner.
\newblock Robust learning under uncertain test distributions: Relating
  covariate shift to model misspecification.
\newblock In \emph{International Conference on Machine Learning (ICML)}, pages
  631--639, 2014.

\bibitem[Woodworth and Srebro(2016)]{woodworth2016tight}
B.~E. Woodworth and N.~Srebro.
\newblock Tight complexity bounds for optimizing composite objectives.
\newblock In \emph{Advances in neural information processing systems
  (NeurIPS)}, 2016.

\bibitem[Zhai et~al.(2021{\natexlab{a}})Zhai, Dan, Kolter, and
  Ravikumar]{zhai2021doro}
R.~Zhai, C.~Dan, Z.~Kolter, and P.~Ravikumar.
\newblock Doro: Distributional and outlier robust optimization.
\newblock In \emph{International Conference on Machine Learning (ICML)},
  2021{\natexlab{a}}.

\bibitem[Zhai et~al.(2021{\natexlab{b}})Zhai, Dan, Suggala, Kolter, and
  Ravikumar]{zhai2021boosted}
R.~Zhai, C.~Dan, A.~Suggala, J.~Z. Kolter, and P.~Ravikumar.
\newblock Boosted {CVaR} classification.
\newblock In \emph{Advances in neural information processing systems
  (NeurIPS)}, 2021{\natexlab{b}}.

\bibitem[Zhou et~al.(2021)Zhou, Levy, Li, Ghazvininejad, and
  Neubig]{zhou2021distributionally}
C.~Zhou, D.~Levy, X.~Li, M.~Ghazvininejad, and G.~Neubig.
\newblock Distributionally robust multilingual machine translation.
\newblock In \emph{Empirical Methods in Natural Language Processing (EMNLP)},
  2021.

\bibitem[Zhou and Gu(2019)]{zhou2019lower}
D.~Zhou and Q.~Gu.
\newblock Lower bounds for smooth nonconvex finite-sum optimization.
\newblock In \emph{International Conference on Machine Learning (ICML)}, 2019.

\end{thebibliography}
\end{document}